\documentclass[11pt]{amsart}
\usepackage{amsmath}
\usepackage{amsfonts}
\usepackage{amsthm}
\usepackage{amssymb}

\theoremstyle{plain}
\numberwithin{equation}{section}
\newtheorem{Theorem}{Theorem}
\newtheorem{Lemma}[Theorem]{Lemma}

\newtheorem{Proposition}[Theorem]{Proposition}

\newtheorem{Corollary}[Theorem]{Corollary}

\newtheorem{Remark}[Theorem]{Remark}

\newtheorem{Trick}[Theorem]{Trick}
 \theoremstyle{remark}

\title[Schr\"odinger operators with singular potentials]
{Spectral gaps of Schr\"odinger operators with periodic singular
potentials}

\author{Plamen Djakov}

\author{Boris Mityagin}

\begin{document}
\begin{abstract}
By using quasi--derivatives we develop a Fourier method for
studying the spectral gaps of one dimensional Schr\"odinger
operators with periodic singular potentials $v.$ Our results
reveal a close relationship between smoothness of potentials and
spectral gap asymptotics under a priori assumption $v \in
H^{-1}_{loc} (\mathbb{R}).$  They extend and strengthen similar
results proved in the classical case $v \in
L^2_{loc}(\mathbb{R}).$
\end{abstract}
\address{Sabanci University, Orhanli,
34956 Tuzla, Istanbul, Turkey}
 \email{djakov@sabanciuniv.edu}
\address{Department of Mathematics,
The Ohio State University,
 231 West 18th Ave,
Columbus, OH 43210, USA} \email{mityagin.1@osu.edu}
\thanks{B. Mityagin acknowledges the support of the Scientific and Technological
Research Council of Turkey and  the hospitality of Sabanci
University, April--June, 2008} \maketitle
\section*{Content}
\begin{enumerate}
\item[Section 1]  Introduction \vspace{3mm}
\item[Section 2]
Preliminaries
\item[Section 3]
Basic equation
\item[Section 4]  Estimates of $\alpha (v;n,z)$ and $\beta^\pm (v;n,z) $
\item[Section 5]  Estimates of spectral gaps $ \gamma_n $
\item[Section 6] Main Results for real--valued potentials
\item[Section 7] Non--self--adjoint case: complex--valued
$H^{-1}$--potentials
\item[Section 8] Comments; miscellaneous
\item[Section 9] Appendix:\\
Deviations of Riesz projections of Hill operators with singular
potentials \vspace{3mm}
\item[9.1] Preliminaries
\item[9.2] Main results on the deviations $P_n - P_n^0$
\item[9.3] Technical inequalities and their proofs
\item[9.4]  Proof of the main theorem on the deviations
$\|P_n - P_n^0\|_{L^1 \to L^\infty}$
\item[9.5] Miscellaneous
\item[References]
\end{enumerate}

\section{Inroduction}
We consider the Hill operator
\begin{equation}
\label{001} Ly = - y^{\prime \prime} + v(x) y, \qquad x \in
I=[0,\pi],
\end{equation}
with a singular complex--valued $\pi$--periodic potential $v \in
H^{-1}_{loc} (\mathbb{R}),$ i.e.,
\begin{equation}
\label{002} v(x) = v_0 + Q^\prime (x),\quad Q \in
L^2_{loc}(\mathbb{R}), \quad Q(x+\pi) = Q(x),
\end{equation}
with $Q$ having zero mean
\begin{equation}
\label{002a}
 q(0) = \int_0^\pi
Q(x) dx =0;
\end{equation}
then
\begin{equation}
\label{003} Q (x) = \sum_{m \in 2\mathbb{Z}\setminus \{0\}} q (m)
e^{ imx},\quad  \|Q\|^2_{L^2(I)} = \|q\|^2=\sum_{m \in
2\mathbb{Z}\setminus \{0\}} |q(m)|^2 <\infty,
\end{equation}
where $q=(q(m))_{m \in 2\mathbb{Z}}.$

Analysis of the Hill or Sturm--Liouville operators, or their
multi--dimensional analogues $- \Delta + v(x) $ with point (surface)
interaction ($\delta$--type) potentials has a long history. From the
early 1960's (F. Berezin, L. Faddeev, R. Minlos
\cite{BF61,B66,MF61}) to around 2000 the topic has been studied in
detail; see books \cite{AGHH,AK} and references there. For specific
potentials see W. N. Everitt and A. Zettl \cite{EZ79,EZ86}.

A more general approach which allows to consider {\em any}
singular potential (beyond $\delta$--functions or Coulomb type) in
negative Sobolev spaces has been initiated by A. Shkalikov and his
coauthors Dzh. Bak, M. Ne\u\i man-zade and A. Savchuk
\cite{BakS,NZS99,SS99}. It led to the spectral theory of
Sturm--Liouville operators with distribution potentials developed
by A. Savchuk and A. Shkalikov \cite{Sav,SS01,SS03,SS06}, and R.
Hryniv and Ya. Mykytyuk \cite{HM01,HM02,HM03,HM04,HM06,HM066}).

Another approach to the study of the Sturm--Liouville operators with
non--classical potentials comes from M. Krein \cite{Kr52}. E.
Korotyaev (see \cite{Kor03,Kor06} and the references therein) uses
this approach very successfully but it seems to be limited to the
case of {\em real} potentials.

A. Savchuk and A. Shkalikov \cite{SS03} consider a broad class of
boundary conditions ($bc$) -- see Formula (1.6) in Theorem 1.5
there -- in terms of a function $y$ and its quasi--derivative $$ u
= y^\prime - Q y. $$ In particular, the proper form of periodic
$Per^+$ and antiperiodic $(Per^-)$ {\it bc} is
\begin{equation}
\label{02}
 Per^\pm: \quad y(\pi) = \pm y(0), \quad u(\pi) = \pm u(0).
\end{equation}
If the potential $v$ happens to be an $L^2$-function those $bc$
are identical to the classical ones (see discussion in
\cite{DM16}, Section 6.2).

 The Dirichlet  $Dir$ $bc$ is more simple:
\begin{equation}
\label{02a} Dir: \quad  y(0) =0,  \quad y(\pi) =0;
\end{equation}
it does not require quasi--derivatives, so it is defined in the
same way as for $L^2$--potentials.

In our analysis of instability zones of the Hill and Dirac operators
(see \cite{DM15} and the comments there) we follow an approach
(\cite{KM1,KM2,DM3,DM5,Mit04,DM6}) based on Fourier Method. But in
the case of singular potentials it may happen that the functions $$
u_k = e^{ikx} \quad \mbox{or} \quad \sin kx, \;\;k \in \mathbb{Z},
$$ have their $L_{bc}$--images outside $L^2.$ Moreover, for some
singular potentials $v$ we have $L_{bc}f \not \in L^2$ for {\em any
nonzero smooth} (say $C^2-$function) $f.$ (For example, choose $$
v(x) = \sum_{r} a(r) \delta_* (x-r),\quad r  \; \mbox{rational},\;
r\in I, $$ with $a(r) >0, \; \sum_{r} a(r) = 1 $ and $\delta_* (x) =
\sum_{k \in \mathbb{Z}} \delta (x-k \pi).$) This implies, for any
reasonable {\it bc}, that the eigenfunctions $\{u_k\}$ of the free
operator $L^0_{bc}$ are not necessarily in the domain of $L_{bc}.$

Yet, in \cite{DM17,DM16} we gave a justification of the Fourier
method for operators $L_{bc}$ with $H^{-1}$--potentials and $bc =
Per^\pm $ or $Dir.$ Our results are announced in \cite{DM17}, and
all technical details of justification of the Fourier method are
provided in \cite{DM16}. In Section 2 we remind our constructions
from \cite{DM16}; this paper is essentially a general introduction
to the present paper. A proper understanding of the boundary
conditions ($Per^\pm $ or $Dir$) -- see (\ref{02}) and the
formulas (\ref{00.1}), $(a^*),\, (b^*)$ below -- and careful
definitions of the corresponding operators and their domains are
provided by using quasi--derivatives.  To great extend we follow
the approach suggested and developed by A. Savchuk and A.
Shkalikov \cite{SS99,SS03} (see also \cite{SS01,SS06}) and
further, by R. Hryniv and Ya. Mykytyuk \cite{HM01} and
\cite{HM04,HM066}.

The Hill--Schr\"odinger operator $L$ with a singular potential
$v\in H^{-1}$ has,  for each $n\geq n_0(v)$ in the disc of center
$n^2$ and radius $n,$ one Dirichlet eigenvalue $\mu_n $ and two
(counted with their algebraic multiplicity) periodic (if $n$ is
even) or anti--periodic (if $n$ is odd) eigenvalues $\lambda_n^-,
\lambda_n^+$ (see Proposition~\ref{prop004} below, or Theorem~21
in \cite{DM16}).

Our main goal in the present paper is to study, for singular
potentials $v\in H^{-1},$ the relationship between the smoothness of
$v$ and the asymptotic behavior of {\em spectral gaps} $\gamma_n
=|\lambda^+_n -\lambda^-_n|$ and {\em deviations} $\delta_n =|\mu_n
-(\lambda^+_n +\lambda^-_n)/2|.$ In the classical case $v \in L^2 $
this relationship means, roughly speaking, that the sequences
$(\gamma_n ) $ and $(\delta_n)$ decay faster if the potential is
smoother, and vise versa. Of course, to make this statement precise
one needs to consider appropriate classes of smooth functions and
related classes of sequences.

This phenomenon was discovered by H. Hochstadt \cite{Ho1,Ho2},   who
showed for {\em  real--valued potentials} $v \in L^2$ the following
connection between the smoothness of $v$ and the rate of decay of
spectral gaps (or, the lengths of instability zones) $\gamma_n =
\lambda^+_n - \lambda^-_n:$ {\em If

 $(i) \; v \in C^\infty, $ i.e., $v$ is infinitely
differentiable, then

$(ii) \; \gamma_n $ decreases more rapidly than any power
 of $1/n.$

If a continuous function $v$ is a finite--zone potential, i.e.,
$\gamma_n =0$ for large enough $n,$ then} $v \in C^\infty. $

In the middle of 70's (see \cite{MO1}, \cite{MT}) the latter
statement was extended, namely, it was shown, for real $L^2
([0,\pi])$--potentials $v,$ that a power decay of spectral gaps
implies infinite differentiability, i.e., $ (ii) \Rightarrow (i).
$

E. Trubowitz~\cite{Tr} has used the Gelfand--Levitan \cite{GL} trace
formula and the Dubrovin equations~\cite{Dub75,DMN} to explain, that
{\em a real $L^2 ([0,\pi])$--potential $v(x)= \sum_{k \in
\mathbb{Z}} V(2k) \exp (2ikx) $ is analytic, i.e.,
$$\exists A>0: \quad |V(2k)| \leq M e^{-A|k|}, $$ if and only if
the spectral gaps decay exponentially, i.e.,} $$\exists a>0: \quad
\gamma_n \leq C e^{-a|k|} $$

M. Gasymov \cite{Gas} showed that if a potential $v \in L^2
([0,\pi]) $ has the form $v(x) =\sum_{k=0}^\infty v_k \exp (2ikx)  $
then $\gamma_n =0 \; \forall n.$ Therefore, in general the decay of
$ \gamma_n $ cannot give any restriction on the smoothness of
complex potentials. V. Tkachenko~\cite{Tk92,Tk94} suggested to bring
the deviations $\delta_n$ into consideration. As a further
development, J.-J. Sansuc and V. Tkachenko \cite{ST} gave a
statement of Hochstadt type: {\em A potential $v \in L^2 ([0,\pi]) $
belongs to the Sobolev space $H^m, \, m\in \mathbb{N}$ if and only
if} $$ \sum \left ( |\gamma_n|^2 + |\delta_n|^2 \right ) \left
(1+n^{2m} \right ) < \infty. $$ \vspace{3mm}

When talking about general classes of $\pi$-periodic smooth
functions we characterize the smoothness by a weight  $\Omega =
(\Omega (n)), $ and consider the ``Sobolev`` space

\begin{equation}
\label{sob} H(\Omega) = \left \{ v(x) = \sum_{k \in \mathbb{Z}} v_k
e^{2ikx}, \quad \sum_{k \in \mathbb{Z}} |v_k|^2 (\Omega (k))^2 <
\infty \right \}.
\end{equation}
The related sequence space is determined as
\begin{equation}
\label{sob1}  \ell^2_{\Omega} = \left \{\xi = (\xi_n)_0^\infty :
\;\; \sum_{n\in \mathbb{N}} |\xi_n|^2 (\Omega (n))^2 < \infty \right
\}.
\end{equation}

In this terminology and under a priori assumption $v\in L^2, $ one
may consider the following general question  on the relationship
between the potential smoothness and decay rate of spectral gaps
and deviations:
 Is it true  that the following conditions $(A)$ and $(B)$ are
equivalent: $$ (A) \;\; v \in H(\Omega); \quad \quad (B) \;\; \gamma
\in \ell^2 (\Omega) \; \;and \;\; \delta \in \ell^2 (\Omega). $$ The
answer is positive for weights $\Omega$
 in a broad range of growths from being constant to
 growing exponentially -- see more
detailed discussion and further results in \cite{DM15}, in
particular Theorems 54 and 67.

Let us note that in the classical case $v \in L^2 $ we have $v_k
\to 0 $ as $|k| \to \infty, $ and $\gamma_n \to 0, \; \delta_n \to
0 $ as $n \to \infty. $ Therefore, if $v \in L^2 $ then we
consider weights which satisfy the condition
\begin{equation}
\label{022} \inf_n \, \Omega (n) > 0.
\end{equation}

In the case $v \in H^{-1} $ the sequences $(v_k), \, (\gamma_n),
(\delta_n)$ may not converge to zero, but $v_k/k \to 0 $ as $|k|
\to \infty, $ and $\gamma_n/n \to 0, \; \delta_n/n \to 0 $ as $n
\to \infty. $ Therefore, in the case $v \in H^{-1} $ it is natural
to consider weights which satisfy
\begin{equation}
\label{012}
 \inf_n  n \, \Omega (n) > 0.
 \end{equation}
The main results of the present paper assert, under a priori
assumption $v\in H^{-1},$ that the conditions (A) and (B) above are
equivalent if $\Omega $ satisfies (\ref{012}) and some other mild
restrictions (see for precise formulations Theorem~\ref{thm33.1}
and Theorem~\ref{thm44.2}). Since the condition (\ref{022}) is more
restrictive than (\ref{012}), these results extend and strengthen
our previous results in the classical case $v \in L^2 $  as well.

Sections 3--7, step by step, lead us to the proofs of Theorems
\ref{thm33.1} and \ref{thm44.2}. As before in the regular case $v
\in L^2 $  (\cite{DM5}, Proposition 4) an important ingredient of
the proof (see in particular Section 7) is the following assertion
about the deviations of Riesz projections $P_n -P_n^0 $ of Hill
operators (with singular potentials) $L_{bc}$ and free operators
$L^0_{bc}:$
\begin{equation}
\label{023}
 \|P_n -P_n^0\|_{L^1 \to L^\infty} \to 0.
 \end{equation}
This fact is important not only in the context of
Theorem~\ref{thm44.2} but also for a series of results on
convergence of spectral decompositions, both in the case of Hill
operators with singular potentials and 1D Dirac operators with
periodic $L^2$--potentials (see \cite{DM19,DM20,DM23}). The proof
and analysis of the statement (\ref{023}) is put aside of the main
text as Appendix, Section 9.

Section 8 gives a few comments (in historical context) on different
parts of the general scheme and its realization. In particular, we
remind and extend in Proposition \ref{prop22} the observation of J.
Meixner and F. W. Sch\"afke  \cite{MS54,MSW80} that the Dirichlet
eigenvalues of the Hill--Mathieu operator could be analytic only in
a bounded disc. They gave upper bounds of the radii of these discs
(see Satz 8, Section 1.5 in \cite{MS54} and p. 87, the last
paragraph, in \cite{MSW80}).
 This is an interesting topic of its own (see  \cite{V96,V00}
and \cite{DM14}). We'll extend this analysis to families of
tri--diagonal matrix operators in the paper \cite{ADM22}.

\section{Preliminaries}

It is known (e.g., see \cite{HM01},   Remark 2.3) that every
$\pi$--periodic potential $ v \in H^{-1}_{loc} (\mathbb{R}) $ has
the form $$  v = C + Q^\prime, \quad \mbox{where} \;\; C= const,
\;\; Q \;\; \mbox{is} \; \pi -\mbox{periodic}, \quad Q \in L^2_{loc}
(\mathbb{R}). $$ Therefore, formally we have $$ -y^{\prime \prime} +
v \cdot y = \ell (y) := -( y^\prime - Qy)^\prime - Q (y^\prime -Qy)
+ (C- Q^2) y.  $$ So, one may introduce  the quasi--derivative $u =
y^\prime -Qy $ and replace the distribution equation $-y^{\prime
\prime} + vy = 0 $ by the following system of two linear
differential equations with coefficients in $L^1_{loc} (\mathbb{R})$
\begin{equation}
\label{00.1} y^\prime = Qy + u, \quad  u^\prime = (C-Q^2)y - Qu.
\end{equation}
By  the  Existence--Uniqueness Theorem for systems of linear o.d.e.
with $L^1_{loc} (\mathbb{R})$--coefficients (e.g., see
\cite{At,Naim}), the Cauchy initial value problem for the
system~(\ref{00.1}) has, for each pair of numbers $(a,b),$ a unique
solution $(y,u) $ such that $y(0) =a, \; u(0)= b.$ This makes
possible to apply the Floquet theory to the system~(\ref{00.1}), to
define a Lyapunov function, etc.

We define the Schr\"odinger operator $L(v) $ in the domain
\begin{equation}
\label{00.01} D(L(v)) =
 \left  \{ y \in H^1 (\mathbb{R}):
 \;\; y^\prime - Qy \in L^2 (\mathbb{R}) \cap W^1_{1,loc} (\mathbb{R}), \;\;
 \ell (y) \in L^2 (\mathbb{R})   \right \},
\end{equation}
by
\begin{equation}
\label{00.02} L(v) y = \ell (y) = - (y^\prime -Qy)^\prime -
Qy^\prime.
\end{equation}
The domain $D(L(v)) $ is dense in $L^2 (\mathbb{R}), $ the operator
$L(v)$ is a closed, and its spectrum could be described in terms of
the corresponding Lyapunov function  (see Theorem 4 in \cite{DM16}).

 In the classical case $v \in L^2_{loc} (\mathbb{R}),$  if $v$ is
a real--valued then by the Floquet--Lyapunov theory (see
\cite{E,Kuch,MW,W}) the spectrum of $L(v)$ is absolutely continuous
and has a band--gap structure, i.e., it is a union of closed
intervals separated by {\em spectral gaps} $$ (-\infty, \lambda_0),
\; (\lambda^-_1,\lambda^+_1), \; (\lambda^-_2,\lambda^+_2),  \ldots,
(\lambda^-_n,\lambda^+_n), \ldots .$$ The points $(\lambda^\pm_n) $
are defined by the spectra of the corresponding Hill operator
considered on the interval $[0,\pi],$ respectively, with periodic
(for even $n$) and anti--periodic (for odd $n$) boundary conditions
$(bc):$

(a) periodic $\quad Per^+: \quad y(\pi)= y(0), \;y^\prime (\pi)=
y^\prime (0);   $

(b) antiperiodic $\quad  Per^-: \quad y(\pi)= - y(0), \;y^\prime
(\pi)= - y^\prime (0); $

Recently a similar statement was proved in the case of real singular
potentials $v\in H^{-1}$ by R. Hryniv and Ya. Mykytyuk \cite{HM01}.

Following A. Savchuk and A. Shkalikov \cite{SS99,SS03}, let us
consider (in the case of singular potentials) periodic and
anti--periodic boundary conditions $Per^\pm $ of the form

($a^*$) $\quad   Per^+: \quad y(\pi)= y(0), \;\left ( y^\prime - Qy
\right ) (\pi)= \left ( y^\prime - Qy  \right ) (0). $

($b^*$) $\quad  Per^-: \quad y(\pi)= -y(0), \;\left ( y^\prime - Qy
\right ) (\pi)= - \left ( y^\prime - Qy  \right ) (0). $

R. Hryniv and Ya. Mykytyuk \cite{HM01} showed, that the Floquet
theory for the system (\ref{00.1}) could be used to explain that if
the potential $v\in H_{loc}^{-1}$ is real--valued, then $L(v) $ is a
self--adjoint operator having absolutely continuous spectrum with
band--gap structure, and the spectral gaps are determined by the
spectra of the corresponding Hill operators $L_{Per^\pm} $ defined
on $[0,\pi]$ by $L_{Per^\pm}  ( y )= \ell (y)$ for $  y \in
D(L_{Per^\pm}),$ where
$$ D(L_{Per^\pm}) =\left  \{y \in H^1  : \; y^\prime - Qy \in
W^1_1 ([0,\pi]),
 \;\; (a^*)\; \text{or} \;(b^*) \; \text{holds}
 ,
 \; \ell (y) \in H^0  \right \}.
$$ (Hereafter the short notations $ H^1 = H^1 ([0,\pi]), \quad H^0 =
L^2 ([0,\pi]) $ are used.)

We set $$H^1_{Per^\pm} = \left \{f \in H^1: \; f(\pi) = \pm f(0)
\right \},\;\;
 H^1_{Dir} = \left \{f \in H^1: \; f(\pi) = f(0) = 0
\right \}. $$ One can easily see that $\{u_k =e^{ikx}, \; k \in
\Gamma_{Per^+}= 2\mathbb{Z} \} $ is an orthogonal basis in
$H^1_{Per^+},$ $\{u_k=e^{ikx}, \; k \in \Gamma_{Per^-}=  1+
2\mathbb{Z} \} $ is an orthogonal basis in $H^1_{Per^-},$ and
$\{u_k =\sqrt{2} \sin kx, \; k \in \Gamma_{Dir}=  \mathbb{N} \} $
is an orthogonal basis in $H^1_{Dir}.$ From here it follows, for
$bc=Per^\pm $  or $Dir,$ that $$H^1_{bc} = \left \{ f(x) =
\sum_{k\in \Gamma_{bc}} f_k u_k (x) \;: \quad \|f\|^2_{H^1}
=\sum_{k\in \Gamma_{bc}} (1+k^2)|f_k|^2 < \infty \right \}. $$

Now, we are ready to explain the Fourier method for studying the
spectra of the operators $L_{Per^\pm}.$ We set
\begin{equation}
\label{0016} V(k) = ik q(k), \quad k \in 2\mathbb{Z},
\end{equation}
where $q(k) $ are the Fourier coefficients of $Q$ defined in
(\ref{003}).

 Let $ \mathcal{F} : H^0 \to \ell^2 (\Gamma_{Per^\pm}) $ be
the Fourier isomorphisms defined by mapping a function $f \in H^0
$ to the sequence $(f_k) $ of its Fourier coefficients $ f_k =
(f,u_k),$ where $\{u_k, \; k \in \Gamma_{Per^\pm}\} $ is the
corresponding  basis introduced above. Let $\mathcal{F}^{-1} $ be
the inverse Fourier isomorphism.

Consider the operators $\mathcal{L}_+ $  and $\mathcal{L}_- $
acting
 as $$ \mathcal{L}_\pm (z) = \left (
h_k (z)  \right )_{k \in \Gamma_{Per^\pm}}, \; h_k (z) = k^2 z_k +
\sum_{m \in \Gamma_{Per^\pm}} V(k-m) z_m +C z_k, $$ respectively,
in  $$ D(\mathcal{L}_\pm ) = \left \{ z \in \ell^2 (|k|,
\Gamma_{Per^\pm}):  \;\mathcal{L}_\pm (z) \in \ell^2
(\Gamma_{Per^\pm})  \right \}, $$  where  $$ \ell^2 (|k|,
\Gamma_{Per^\pm})= \left \{z=(z_k)_{k \in \Gamma_{Per^\pm}}: \;\;
\sum_{k} (1+|k|^2) |z_k|^2 < \infty \right \}.$$

\begin{Proposition}
\label{prop001} (Theorem 11 in \cite{DM16}) In the above
notations, we have
\begin{equation}
\label{0033} D(L_{Per^\pm})=\mathcal{F}^{-1} \left (
D(\mathcal{L}_\pm ) \right ), \qquad
 L_{Per^\pm} = \mathcal{F}^{-1} \circ
 \mathcal{L}_\pm \circ \mathcal{F}.
\end{equation}
\end{Proposition}

 Next we study the Hill--Schr\"odinger operator $L_{Dir} (v),
\; v=C + Q^\prime, $ generated by the differential expression
$\ell_Q (y) $ when considered with
 Dirichlet boundary conditions $ Dir: \; y(0) =
y(\pi) = 0.$ We set $$L_{Dir} (v) y= \ell_Q (y), \qquad y \in
D(L_{Dir} (v)), $$ where $$ D(L_{Dir} ) = \left \{ y \in H^1 : \;
y^\prime - Qy \in W^1_1 ([0,\pi]), \; y(0) = y(\pi) = 0, \; \ell_Q
(y) \in H^0 \right \}. $$

\begin{Proposition}
\label{prop002} (Theorem 13 in \cite{DM16}) Suppose $ v\in
H^{-1}_{loc} (\mathbb{R})$ is $\pi$--periodic. Then:

(a) the domain $D(L_{Dir} (v)) $ is dense in $H^0; $

(b) the operator $L_{Dir} (v)$ is closed and its conjugate
operator is $$ \left (L_{Dir} (v) \right )^*  = L_{Dir}
(\overline{v}),$$ so, in particular,  if $v$ is real, then the
operator $L_{Dir} (v)$ is self--adjoint;

(c) the spectrum $Sp (L_{Dir} (v))$ of the operator $L_{Dir} (v) $
is discrete, and  $$ Sp (L_{Dir} (v)) = \{\lambda \in \mathbb{C}
\; : \;\; y_2 (\pi, \lambda) = 0 \}. $$
\end{Proposition}

Let
\begin{equation}
\label{0035} Q(x) = \sum_{k \in \mathbb{N}} \tilde{q}(k) \sqrt{2}
\sin kx
\end{equation}
be the sine Fourier series of $Q.$ We set
\begin{equation}
\label{0036} \tilde{V}(k) = k \tilde{q}(k), \qquad k \in
\mathbb{N}.
\end{equation}

Let $ \mathcal{F} : H^0 \to \ell^2 (\mathbb{N}) $ be the Fourier
isomorphisms that maps a function $f \in H^0 $ to the sequence
$(f_k)_{k\in\mathbb{N}} $ of its Fourier coefficients $ f_k =
(f,\sqrt{2} \sin kx),$ and let $\mathcal{F}^{-1} $ be the inverse
Fourier isomorphism.

We set, for each $z = (z_k) \in  \ell^2 (\mathbb{N}), $ $$ \quad
h_k (z) = k^2 z_k + \frac{1}{\sqrt{2}} \sum_{m \in \mathbb{N}}
\left ( \tilde{V}(|k-m|) -\tilde{V} (k+m) \right )
 z_m +C z_k,
$$ and consider the operator $\mathcal{L}_{d} $
  defined by
$$\mathcal{L}_{d} (z) = \left ( h_k (z)  \right )_{k \in
\mathbb{N}} $$ in the domain $$ D(\mathcal{L}_{d} ) = \left \{ z
\in \ell^2 (|k|, \mathbb{N}): \; \;\mathcal{L}_{d} (z) \in \ell^2
(\mathbb{N}) \right \},$$ where $$ \; \ell^2 (|k|,\mathbb{N})=
\left \{z=(z_k)_{k \in \mathbb{N}}: \;\; \sum_{k} |k|^2 |z_k|^2 <
\infty \right \}.$$

\begin{Proposition}
\label{prop003} (Theorem 16 in \cite{DM16}) In the above
notations, we have
\begin{equation}
\label{0133} D(L_{Dir})=\mathcal{F}^{-1} \left ( D(\mathcal{L}_{d}
) \right ),  \qquad
 L_{Dir} = \mathcal{F}^{-1} \circ
 \mathcal{L}_{d} \circ \mathcal{F}.
\end{equation}
\end{Proposition}

Let $L^0$ be the free operator, and let $V$ denotes the operator of
multiplication by $v.$ One of the technical difficulties that arises
for singular potentials is connected with the standard perturbation
type formulae for the resolvent $R_\lambda = (\lambda - L^0
-V)^{-1}.$ In the case where $v \in L^2 ([0,\pi]) $ one can
represent the resolvent in the form $$ R_\lambda = (1-R_\lambda^0
V)^{-1} R_\lambda^0 = \sum_{k=0}^\infty (R_\lambda^0 V)^k
R_\lambda^0 \; \; \mbox{or} \;  R_\lambda = R_\lambda^0 (1-V
R_\lambda^0 )^{-1} = \sum_{k=0}^\infty R_\lambda^0 (V
R_\lambda^0)^k, $$ where $R_\lambda^0 = (1-L^0)^{-1}.$
 The simplest
conditions that guarantee the convergence of these series are $$
\|R_\lambda^0 V\| <1, \quad \mbox{respectively,} \quad \|V
R_\lambda^0\| < 1. $$ Each of these conditions can be easily
verified for large enough $n$ if $ Re \, \lambda \in [n-1, n+1]$
and $ |\lambda - n^2 | \geq C(\|v\|_{L^2}), $ which leads to a
series of results on the spectra, zones of instability and
spectral decompositions.

The situation is more complicated if $v$ is a singular potential.
Then, in general, there are no good estimates for the norms of
 $      R^0_\lambda V  $ and
$ V R^0_\lambda. $  However, one can write $R_\lambda $ in the
form $$ R_\lambda = R^0_\lambda + R^0_\lambda V R^0_\lambda +
R^0_\lambda V R^0_\lambda V R^0_\lambda + \cdots = K^2_\lambda +
\sum_{m=1}^\infty K_\lambda(K_\lambda V K_\lambda)^m K_\lambda, $$
provided $(K_\lambda)^2 = R^0_\lambda.$ We define an operator $K=
K_\lambda $ with this property  by its matrix representation
$$K_{jm} = (\lambda - j^2)^{-1/2} \delta_{jm}, \quad j,m \in
\Gamma_{bc}, $$ where $z^{1/2} = \sqrt{r} e^{i\varphi/2} \quad
\mbox{if} \quad z= re^{i\varphi}, \;\; - \pi \leq \varphi < \pi. $
Then $R_\lambda $ is well--defined if $$  \|K_\lambda V K_\lambda:
\; \ell^2 (\Gamma_{bc})  \to \ell^2 (\Gamma_{bc})\| < 1. $$

By proving good estimates from above of the Hilbert--Schmidt norms
of the operators $K_\lambda V K_\lambda $ for $bc = Per^\pm $ or
$Dir$ we get the following statements.

\begin{Proposition}
\label{prop004} (Theorem 21 in \cite{DM16}) For each periodic
potential $v \in H_{loc}^{-1} (\mathbb{R}), $ the spectrum of the
operator $L_{bc} (v) $  with $bc = Per^\pm, \, Dir $ is discrete.
Moreover, if $bc = Per^\pm $ then, respectively, for each large
enough even number
 $N^+ > 0 $    or odd number $N^-$, we have
$$ Sp \left ( L_{Per^\pm} \right ) \subset R_{N^\pm} \cup
\bigcup_{n \in N^\pm  + 2\mathbb{N}} D_n, $$ where the rectangle
$R_N =\{\lambda=x+iy:  -N < x < N^2 +N, \; |y| <N \}$ contains,
respectively, $2N^+ +1$  or $  2N^-$ eigenvalues, while each disc
$ \;D_n = \{\lambda:\; |\lambda - n^2 | < n/4 \}$ contains two
eigenvalues ( each eigenvalue is counted with its algebraic
multiplicity).

If $bc = Dir$ then, for each large enough
 $N \in \mathbb{N}, $  we have
$$ Sp \left ( L_{Dir} \right ) \subset R_{N} \cup \bigcup_{n
=N+1}^\infty D_n $$ and $$ \# \left ( Sp \left ( L_{Dir} \right )
\cap R_N \right ) = N+1, \qquad \# \left ( Sp \left ( L_{Dir})
\right ) \cap D_n \right ) = 1, \;\; n > N. $$
\end{Proposition}

This localization theorem, i.e., Proposition \ref{prop004}, says
that for each $n>N^*=\max \{N^+, N^-, N \}$ the disc $D_n$
contains exactly one Dirichlet eigenvalue $\mu_n $ and two
periodic (if $n$ is even) or antiperiodic (if $n$ is odd)
eigenvalues $\lambda_n^+$ and $ \lambda_n^-,$  counted with their
algebraic multiplicity, where either $Re \, \lambda_n^+ >Re \,
\lambda_n^- ,$  or $Re \, \lambda_n^+ =Re \, \lambda_n^- $ and $Im
\, \lambda_n^+ \geq Im \, \lambda_n^- .$

 After this observation we define {\em spectral gaps}
 \begin{equation}
 \label{p01}
 \gamma_n =|\lambda_n^+  -\lambda_n^-|, \quad n \geq N^*,
 \end{equation}
 and {\em deviations}
 \begin{equation}
 \label{p02}
 \delta_n = \left | \gamma_n -
 (\lambda_n^+  + \lambda_n^-)/2 \right |, \quad n \geq N^*.
\end{equation}
We characterize ``the rate of decay`` of these sequences by saying
that they are elements of an appropriate weight sequence space $$
\ell^2 (\Omega) = \left \{(x_k)_{k \in \mathbb{N}} : \; \sum
|x_k|^2 (\Omega(k))^2 < \infty  \right \}. $$

The condition (\ref{012}) means that $\ell^2 (\Omega) \subset
\ell^2 (\{n\}).$ Now, with the restriction (\ref{012}) -- which is
weaker than (\ref{022}) -- we permit the weights to decrease to
zero.  For example, the weights  $\Omega_\beta (n) = n^{-\beta},
\; \beta \in (0,1],$ satisfy (\ref{012}). Moreover, if $\beta
>1/2, $ then the sequence $x_k = k^\alpha,  \; 0 < \alpha < \beta
- 1/2 $ goes to $\infty $ as $k \to \infty$ but $(x_k ) \in \ell^2
(\Omega_\beta).$

By the same token, the ``smoothness`` of a potential $v$ with
Fourier coefficients $V(2k) $ given by (\ref{0016}) is
characterized by saying that $v$ belongs to an appropriate weight
space $H(\Omega),$ i.e., $$\sum |V(2k)|^2 (\Omega (k))^2 <
\infty.$$ The inclusion $H(\Omega) \subset L^2 $ is equivalent to
(\ref{022}) while (\ref{012}) is equivalent to the inclusion
$H(\Omega) \subset H^{-1}.$

Let us recall that a sequence of positive numbers, or a weight $A
= (A(n))_{n\in \mathbb{Z}},$ is called {\em sub--multiplicative}
if
\begin{equation}
\label{p03} A (m+n)\leq  A (m) A (n), \quad m,n \in \mathbb{Z}.
\end{equation}
In this paper we often consider weights of the form
\begin{equation}
\label{p04} a (n) = \frac{A (n)}{n} \quad \text{for} \; n \neq 0,
\quad a (0) =1,
\end{equation}
where $A (n)$ is sub--multiplicative, even and monotone increasing
for $n \geq 1. $

Now we are ready to formulate the main result of the present paper
(this is Theorem~\ref{thm33.1} in Section 7).

 {\bf Main Theorem.}  {\em Let $L = L^0 + v(x) $ be the
Hill--Schr\"odinger operator with a $\pi$--periodic potential $v \in
H^{-1}_{loc} (\mathbb{R}).$}

{\em Then, for large enough $n > N(v)$  the operator $L$ has, in a
disc of center $n^2 $ and radius $r_n = n/4, $ exactly two (counted
with their algebraic multiplicity) periodic (for even $n$), or
antiperiodic (for odd $n$)   eigenvalues $\lambda^+_n $ and $
\lambda^-_n, $ and one Dirichlet eigenvalue $\mu_n. $}

{\em Let $$ \Delta_n = |\lambda^+_n - \lambda^-_n | + |\lambda^+_n -
\mu_n|, \quad n > N (v), $$ and let $\Omega = (\Omega (m))_{m\in
\mathbb{Z}}, $ $\; \Omega (m)= \omega (m)/m, \; m\neq 0,  $  where
$\omega$ is a sub-multiplicative weight. Then we have $$ v \in
H(\Omega )  \; \Rightarrow \; (\Delta_n)  \in \ell^2 (\Omega ). $$
Conversely, in the above notations, if  $\omega = (\omega (n))_{n\in
\mathbb{Z}} $ is a sub--multiplicative weight such that $$\frac{\log
\omega (n)}{n} \searrow 0, $$ then $$(\Delta_n) \in \ell^2 (\Omega)
\; \Rightarrow \; v \in H(\Omega ).$$}

{\em If $\omega $ is of exponential type, i.e., $\lim_{n\to \infty}
\frac{\log \omega (n)}{n} >0, $ then $$ (\Delta_n) \in \ell^2
(\Omega) \; \Rightarrow   \; \exists \varepsilon >0: \; v \in
H(e^{\varepsilon |n|} ). $$}

Important tool in the proof of this theorem is the following
statement. Let $E_n$  be the Riesz invariant subspace corresponding
to the (periodic for even $n$, or antiperiodic for odd $n$)
eigenvalues of $L_{Per^\pm}$ lying in the disc $ \{z: \; |z - n^2| <n
\},   $ and let $P_n$ be the corresponding Riesz projection, i.e., $$
P_n = \frac{1}{2\pi i} \int_{C_n} (\lambda - L)^{-1} d\lambda, \qquad
C_n := \{\lambda: \; |\lambda - n^2| = n \}. $$ We denote by $P_n^0$
the Riesz projection that corresponds to the free operator.

{\em In the above notations and under the assumptions of
Proposition}~\ref{prop004}
\begin{equation}
\label{0021} \|P_n - P_n^0 \|_{L^2 \to L^\infty} \to 0 \quad
\mbox{as} \;\; n \to \infty.
\end{equation}

This statement and its stronger version are proven in Appendix;
see Proposition~\ref{prop91} and Theorem~\ref{thm91} there.

\section{Basic equation}

By the localization statement given in Proposition \ref{prop004},
the spectrum of the operator $L_{Per^\pm}$ is discrete, and there
exists $N_* $ such that for each $n \geq N_* $ the disc $D_n =
\{\lambda: |\lambda - n^2| < n/4 \} $ contains exactly two
eigenvalues (counted with their algebraic multiplicity) of
$L_{Per^+}$ (for $n$ even), or
 $L_{Per^-}$ (for $n$ odd).

For each $n \in \mathbb{N},$ let $$ E^0 =E^0_n = Span \{e_{-n} =
e^{-inx},  \; e_n = e^{inx} \} $$ be the eigenspace of the free
operator $L^0 = - d^2/dx^2$ corresponding to its eigenvalue  $n^2
$ (subject to periodic boundary conditions for even $n,$ and
antiperiodic boundary conditions for odd $n$). We  denote by $P^0
=P^0_n $ the orthogonal projection on $E^0,  $  and set $Q^0 =
Q^0_n = 1-P^0_n.  $ Notice that
\begin{equation}
\label{2.1} D_n \cap Sp (Q^0 L^0 Q^0) = \emptyset.
\end{equation}

Consider the operator $\tilde{K}$ given by its matrix
representation
\begin{equation}
\label{2.2} \tilde{K}_{jk} = \begin{cases} \frac{1}{(\lambda -
k^2)^{1/2}} \delta_{jk} & \text{for} \quad k \neq \pm n,\\ 0 &
\text{for}    \quad k = \pm n,
\end{cases}
\end{equation}
where $z^{1/2} := \sqrt{r} e^{i\varphi/2} $ if $ z=re^{i\varphi},
\; -\pi \leq \varphi <\pi,$ and $j,k \in n + 2 \mathbb{Z}.$ One
can easily see that $ \tilde{K} $ acts from $L^2([0,\pi]) $  into
$H^1,$ and from $H^{-1}$ into $L^2([0,\pi]). $

We consider also the operator
\begin{equation}
\label{2.3} T = T(n;\lambda) = \tilde{K} V \tilde{K}.
\end{equation}
By the diagram $$ L^2([0,\pi]) \stackrel{\tilde{K}}{\to} H^1
\stackrel{V}{\to} H^{-1} \stackrel{\tilde{K}}{\to} L^2([0,\pi]) $$
the operator $T$ acts in $L^2([0,\pi]).$ In view of (\ref{0016}),
 (\ref{2.2}) and (\ref{2.3}), the matrix
representation of $T$ is
\begin{equation}
\label{2.4} T_{jk} = \begin{cases}\frac{V(j-k)} {(\lambda -
j^2)^{1/2}{(\lambda - k^2)^{1/2}}}= \frac{i(j-k)q(j-k)} {(\lambda
- j^2)^{1/2}{(\lambda - k^2)^{1/2}}} & \text{if} \quad j,k \neq
\pm n,\\ 0 & \text{otherwise}  \end{cases}
\end{equation}
Let us set
\begin{equation}
\label{2.5a} H_n = \{z: \; (n-1)^2 \leq Re\,z \leq (n+1)^2\}
\end{equation}
and
\begin{equation}
\label{2.5b} \mathcal{E}_{m} (q) = \left ( \sum_{|k|\geq m} |q(m)|^2
\right )^{1/2}.
\end{equation}

\begin{Lemma}
\label{lem2.1} In the above notations, we have
\begin{equation}
\label{2.5} \|T \|_{HS} \leq C \left (  \mathcal{E}_{\sqrt{n}} (q)
+ \|q\|/ \sqrt{n} \right ), \quad \lambda \in H_n,
\end{equation}
where $C$ is an absolute constant.
\end{Lemma}

\begin{proof}
In view of (\ref{2.4}),
\begin{equation}
\label{2.6} \|T\|_{HS}^2 = \sum_{j,k \neq \pm n} \frac{(j-k)^2
|q(j-k)|^2} {|\lambda - j^2| |\lambda - k^2|},
\end{equation}
so we have to estimate from above the sum in (\ref{2.6})  with $
\lambda \in H_n. $

The operator  $\tilde{K} $ is a modification of the operator $K$
defined by $$ K_{jk} = \frac{1}{(\lambda - k^2)^{1/2}}
\delta_{jk}, \quad j,k \in n + 2 \mathbb{Z}, \; \lambda \in
H_n\setminus \{n^2\}. $$ Moreover, for $\lambda \neq n^2, $ we
have $ \tilde{K} = Q^0 K Q^0.
$

Therefore (compare (\ref{2.6}) with Formula (128) in \cite{DM16}),
by repeating the proof of Lemma 19 in \cite{DM16}) with a few
simple changes there one can easily see that (\ref{2.5}) holds.

\end{proof}

\begin{Lemma}
\label{lem2.2} In the above notations, if $ n \in \mathbb{N} $ is
large enough and $\lambda \in D_n,$ then $\lambda $ is an
eigenvalue of the operator $L= L^0 + V $ (considered with periodic
boundary conditions if $n$ is even, or antiperiodic boundary
conditions if $n$ is odd) if and only if $\lambda$ is an
eigenvalue of the operator $P^0 L^0 P^0 + S, $ where
\begin{equation}
\label{2.8} S = S(\lambda;n) = P^0 V P^0 + P^0 V
\tilde{K}(1-T)^{-1} \tilde{K}VP^0 \;: \;\; E^0_n \to E^0_n.
\end{equation}
Moreover,
\begin{equation}
\label{2.9} Lf = \lambda f, \;\; f\neq 0 \; \; \Rightarrow (L^0
+S) f_1 = \lambda  f_1, \;\; f_1 = P^0 f \neq 0,
\end{equation}
\begin{equation}
\label{2.10} (L^0 +S) f_1 = \lambda f_1, \;\; f_1 \in E^0 \; \;
\Rightarrow Lf =\lambda f, \quad f =f_1 + \tilde{K}(1-T)^{-1}
\tilde{K} V f_1.
\end{equation}
\end{Lemma}

\begin{proof}
The equation
\begin{equation}
\label{2.12} (\lambda - L^0 - V)f = g
\end{equation}
is equivalent to the system of two equations
\begin{equation}
\label{2.13} P^0 (\lambda - L^0 - V) (f_1 + f_2) = g_1,
\end{equation}
\begin{equation}
\label{2.14} Q^0 (\lambda - L^0 - V) (f_1 + f_2) = g_2,
\end{equation}
where $ f_1 = P^0 f, \; f_2 = Q^0 f, \; g_1 = P^0 g, \; g_2 = Q^0
g.$

Since the operator $L^0 $ is self--adjoint, the range $Q^0 (H) $
of the projection $Q^0 $ is an invariant subspace of $L^0 $ also.
Therefore, $$ P^0 Q^0 = Q^0 P^0 = 0, \quad P^0 L^0 Q^0 = Q^0 L^0
P^0 = 0,$$ so Equations (\ref{2.13}) and (\ref{2.14}) may be
rewritten as
\begin{equation}
\label{2.15} (\lambda - L^0) f_1 - P^0 V f_1 - P^0 Vf_2 = g_1,
\end{equation}
\begin{equation}
\label{2.16} (\lambda - L^0) f_2 - Q^0 V f_1 - Q^0 Vf_2 = g_2.
\end{equation}
Since $f_2 $ belongs to the range of $Q^0, $ it can be written in
the form
\begin{equation}
\label{2.17}
 f_2 = \tilde{K} \tilde{f}_2.
\end{equation}
Next we substitute this expression for $f_2$ into (\ref{2.16}),
and after that act from the left on the equation by $\tilde{K}. $
As a result we get $$ \tilde{K} (\lambda - L^0) \tilde{K}
\tilde{f}_2  - \tilde{K}Q^0 V f_1 - \tilde{K} V \tilde{K}
\tilde{f}_2 = \tilde{K} g_2. $$ By the definition of $\tilde{K}, $
we have the identity $$ \tilde{K} (\lambda - L^0) \tilde{K}
\tilde{f}_2 =\tilde{f}_2. $$ Therefore, in view of (\ref{2.3}),
the latter equation can be written in the form
\begin{equation}
\label{2.18} (1-T) \tilde{f}_2 = \tilde{K} V f_1 + \tilde{K} g_2.
\end{equation}
By Lemma \ref{lem2.1} the operator $1- T$  is invertible for large
enough $n.$  Thus, (\ref{2.17}) and  (\ref{2.18}) imply, for large
enough $n,$
\begin{equation}
\label{2.19} f_2 = \tilde{K}(1-T)^{-1} \tilde{K} V f_1 + \tilde{K}
(1-T)^{-1} \tilde{K} g_2.
\end{equation}
By inserting this expression for $f_2 $ into (\ref{2.15}) we get
\begin{equation}
\label{2.20} (\lambda -P^0  L^0 P^0 - S) f_1 = g_1 + P^0 V
\tilde{K} (1-T)^{-1} \tilde{K} g_2,
\end{equation}
where the operator $S$ is given by (\ref{2.8}).

If $\lambda $ is an eigenvalue of $L$ and $f\neq 0$ is a
corresponding eigenvector, then (\ref{2.12}) holds with $g=0,$ so
$g_1 =0, \; g_2 =0, $  and (\ref{2.20}) implies (\ref{2.9}), i.e.,
$\lambda $ is an eigenvector of $P^0 L^0 P^0 + S $ and $f_1 = P^0
f$ is a corresponding eigenvector. Then we have $ f_1 \neq 0; $
otherwise (\ref{2.19}) yields $f_2 = 0, $  so  $f = f_1 + f_2 = 0
$ which is a contradiction.

Conversely, let $\lambda $ be an eigenvalue of $P^0 L^0 P^0 + S, $
and let $f_1 $ be a corresponding eigenvector. We set $$ f_2 =
\tilde{K}(1-T)^{-1} \tilde{K} V f_1,  \quad f = f_1 + f_2 $$ and
show that $Lf= \lambda f $ by checking that (\ref{2.15}) and
(\ref{2.16}) hold with $g_1 = 0 $ and $g_2 =0. $ Then (\ref{2.15})
coincides with $(P^0 L^0 P^0 + S)f_1 =0,  $ so it holds. On the
other hand one can easily verify (\ref{2.16}) by using that $$
(1-T)^{-1} = 1 + T(1-T)^{-1}, \quad   T =\tilde{K} V \tilde{K},
\quad (\lambda - L^0) \tilde{K}^2  = Q^0.$$

This completes the proof.

\end{proof}

Notice that (\ref{2.9}) and (\ref{2.10}) imply the following.
\begin{Remark}
\label{rem21.1} Under the assumptions and notations of Lemma
\ref{lem2.2},  for large enough $n,$ the operator $L$ has an
eigenvalue $\lambda \in  D_n $ of geometric multiplicity 2 if and
only if $\lambda $ is an eigenvalue of $ P^0 L^0 P^0 + S(\lambda,
n) $ of geometric multiplicity 2.
\end{Remark}

Of course, in the proof of Lemma \ref{lem2.2} it was enough to
consider the equation (\ref{2.12}) for $g =0 $ (and, respectively,
to set $g_1=0, g_2=0 $ in the following equations). But we
consider the case of arbitrary $g$ in order to explain the
following remark that will be used later.

\begin{Remark}
\label{rem2.1} Under the assumptions and notations of Lemma
\ref{lem2.2} and its proof, Equation (\ref{2.12}), i.e., the
system (\ref{2.13}) and  (\ref{2.14}), implies Equation
(\ref{2.20}).
\end{Remark}

Let
$
\begin{pmatrix}
S^{11} & S^{12} \\ S^{21}  &  S^{22}
\end{pmatrix}
$
be the matrix representation of the two-dimensional operator $S$
with respect to the basis $e^1_n  := e_{-n}, \; e^2_n := e_n; $
then
\begin{equation}
\label{2.27} S^{ij} = \langle  Se^j_n , e^i_n \rangle ,\qquad i,j
\in \{1,2\}.
\end{equation}

Consider the eigenvalue equation for the operator $S:$
\begin{equation}
\label{2.28} \det \left |
\begin{array}{cc}
S^{11} - z & S^{12}\\ S^{21}  &  S^{22} - z
\end{array}
\right| = 0.
\end{equation}
This is the {\em basic equation} that we use to estimate spectral
gaps.

A number $ \lambda = n^2+z \in D_n$ is a periodic or
anti--periodic eigenvalue of $L^0 +V $ if and only if $ z $ is a
solution of (\ref{2.28}).

Next we give explicit formulas for the matrix elements $S^{ij}.$
By Lemma~\ref{lem2.1}, we have $\|T_n \| \leq 1/2 $ for $n \geq
n_0 (v), $ so $ (1-T_n)^{-1} = \sum_{k=1}^\infty T_n^{k-1}, $ and
therefore, $$ S = P^0_n V P^0_n + \sum_{k=1}^\infty P^0_n V
\tilde{K} T_n^{k-1}  \tilde{K}  V  P^0_n. $$ By (\ref{2.8}) and
(\ref{2.27}),
\begin{equation}
\label{2.29} S^{ij} = \sum_{k=0}^\infty  S^{ij}_k, \quad
\text{where} \quad S^{ij}_0 = \langle   P^0 V P^0 e_n^j, e_n^i
\rangle
\end{equation}
and
\begin{equation}
\label{2.29a} S^{ij}_k = \langle V \tilde{K} T_n^{k-1}  \tilde{K}
V  e_n^j, e_n^i \rangle, \quad k= 1,2, \ldots.
\end{equation}

By (\ref{2.3}), it follows that $$
 V \tilde{K} T_n^{k-1}  \tilde{K} V
= (V \tilde{K}^2)^{k} V. $$ Therefore, in view of (\ref{2.29}),
(\ref{2.29a}) and (\ref{2.2}), we have
\begin{equation}
\label{2.30} S^{11}_0 = \langle Ve^1_n, e^1_n \rangle = V(0) = 0,
\quad S^{22}_0 = \langle Ve^2_n, e^2_n \rangle = V(0) = 0,
\end{equation}
and for each $k=1,2, \ldots, $
\begin{equation}
\label{2.31} S^{11}_k =
 \sum_{j_1, \ldots, j_k \neq \pm n}
\frac{V(-n-j_1)V(j_1 - j_2) \cdots V(j_{k-1} -j_k) V(j_k +n)}
{(n^2 -j_1^2 +z) \cdots (n^2 - j_k^2 +z)},
\end{equation}
\begin{equation}
\label{2.32} S^{22}_k =
 \sum_{j_1, \ldots, j_k \neq \pm n}
\frac{V(n-j_1)V(j_1 - j_2) \cdots V(j_{k-1} -j_k) V(j_k -n)} {(n^2
-j_1^2 +z) \cdots (n^2 - j_k^2 +z)}.
\end{equation}
In the same way, we obtain
\begin{equation}
\label{2.33} S^{12}_0 = \langle Ve^2_n, e^1_n \rangle = V(-2n),
\quad S^{21}_0 = \langle Ve^1_n, e^2_n \rangle = V(2n),
\end{equation}
and, for $k=1,2, \ldots,$
\begin{equation}
\label{2.34} S^{12}_k =  \sum_{j_1, \ldots, j_k \neq \pm n}
\frac{V(-n-j_1)V(j_1 - j_2) \cdots V(j_{k-1} -j_k) V(j_k -n)}
{(n^2 -j_1^2 +z) \cdots (n^2 - j_k^2 +z)}
\end{equation}
\begin{equation}
\label{2.35} S^{21}_k = \sum_{j_1, \ldots, j_k \neq \pm n}
\frac{V(n-j_1)V(j_1 - j_2) \cdots V(j_{k-1} -j_k) V(j_k +n)} {(n^2
-j_1^2 +z) \cdots (n^2 - j_k^2 +z)}
\end{equation}
The operator $S$ depends on $v, n\in \mathbb{N}$ and $\lambda $
(or, $ z= \lambda - n^2). $ Of course, the same is true for its
matrix, i.e., $$ S^{ij} = S^{ij} (v;n,z), \quad S^{ij}_k =
S^{ij}_k (v;n,z). $$

\begin{Lemma}
\label{lem2.3} (a) For any (complex--valued) potential $v$
\begin{equation}
\label{2.36}  S^{11} (v;n,z) = S^{22} (v;n,z).
\end{equation}

(b) If $v$ is a real--valued potential, then
\begin{equation}
\label{2.37} S^{12} (v;n,z) = \overline{S^{21}
(v;n,\overline{z})}.
\end{equation}
\end{Lemma}

\begin{proof}
(a) By (\ref{2.31}) and (\ref{2.32}), the change of indices $$ i_s
= - j_{k+1-s}, \quad s=1, \ldots, k, $$ proves that $S^{11}_k
(v;-n,z) = S^{22}_k (v;n,z) $  for $k =1,2, \ldots. $ Thus, in
view of (\ref{2.29}), (\ref{2.36}) holds.

(b) If $v$ is real-valued, we have for its Fourier coefficients $
V(-m) = \overline{V(m)}. $ By (\ref{2.33}), $$
 S^{12}_0 (v;-n,z) = V(-2n) = \overline{V(2n)} =
\overline{S^{21}_0 (v; n, \overline{z})}. $$ By (\ref{2.34}) and
(\ref{2.35}), for each $ k =1,2, \ldots, $ the change of indices
$$ i_s = j_{k+1-s}, \quad s=1, \ldots, k, $$ explains that $
S^{12}_k (-n,z) = \overline{S^{21}_k (n,\overline{z})}. $ Thus, in
view of (\ref{2.29}), (\ref{2.37}) holds. Lemma~\ref{lem2.3} is
proved.

\end{proof}

We set, for $n > n_0 (v), $
\begin{equation}
\label{2.40} \alpha_n (v;z) = S^{11} (v;n,z), \;\; \beta^+_n (v;z) =
S^{21} (v;n,z), \;\; \beta^-_n (v;z) = S^{12} (v;n,z).
\end{equation}

\section{Estimates of $\alpha (v;n,z)$ and $\beta^\pm (v;n,z) $ }

Let $\Omega = (\Omega (m)) $ be a weight of the form
\begin{equation}
\label{} \Omega (m) = \omega (m) /m \quad \text{for} \quad m\neq 0
\qquad \Omega (0) = 1,
\end{equation}
where $\omega $ is a submultiplicative weight. Our main goal in this
section is to estimate the $\ell^2 (\Omega ) $--norm of $(\beta^\pm
(v;n,z))_{n>N_0} $ under the assumption $v \in H (\Omega).$

\begin{Lemma}
\label{lem3.1} If $r \in \ell^2 (\mathbb{Z}), $ then
\begin{equation}
\label{3.0} \sum_{i \neq -n} \left | \frac{n-i}{n+i}  \right |^2
|r(n-i)|^2 \leq Cn^2  \left (  \mathcal{E}_n (r) + \frac{\|r\|}{n}
\right )^2,
\end{equation}
where $C$ is an absolute constant.
\end{Lemma}

\begin{proof}
Since  $|n-i|/|n+i| \leq 1 $ for $i \geq 0,$ and
\begin{equation}
\label{3.01} \left ( \frac{n-i}{n+i} \right )^2 = \left (
\frac{2n}{n+i} - 1 \right )^2 \leq \frac{8n^2}{(n+i)^2} + 2
\quad \text{for} \; i<0,
\end{equation}
the sum in (\ref{3.0}) does not exceed $$ \sum_{i <0, i \neq -n}
\frac{8n^2}{(n+i)^2} |r(n-i)|^2 +
 \sum_{i} 3 |r(n-i)|^2  \leq 8n^2 \left (\mathcal{E}_n (r) \right )^2
 +3 \|r\|^2,
$$ which proves (\ref{3.0}).

\end{proof}

\begin{Lemma}
\label{lem3.2} If $r \in \ell^2 (\mathbb{Z}) $ and $ H(n) = \left
(H_{ij} (n) \right ), \; n \in \mathbb{N}, \;n>N, $ are
Hilbert--Schmidt operators such that
$
H_* := \sup_{n>N}\|H(n) \|_{HS} < \infty,
$
then
\begin{equation}
\label{3.1}
 \sum_{i,j\neq \pm n}
\left | \frac{n-i}{n+i} \right |^{1/2} \left | \frac{n+j}{n-j}
\right |^{1/2} |r(n-i)| |r(n+j)| |H_{ij} (n)| \leq Cn \|r\| \left
( \mathcal{E}_n (r) + \frac{\|r\|}{n}  \right )H_*,
\end{equation}
where $C$ is an absolute constant.
\end{Lemma}

\begin{proof}
By the Cauchy inequality, for each $ n \in \mathbb{N}$ with $n>N,
$ we have
\begin{equation}
\label{3.2} \left ( \sum_{i,j\neq \pm n} \left | \frac{n-i}{n+i}
\right |^{1/2} \left | \frac{n+j}{n-j} \right |^{1/2} |r(n-i)|
|r(n+j)|| H_{ij} (n)| \right )^2 \leq \sigma (n) \|H(n)\|_{HS}^2,
\end{equation}
where
\begin{equation}
\label{3.3} \sigma (n) = \sum_{i,j\neq \pm n} \left |
\frac{n-i}{n+i} \right | \left | \frac{n+j}{n-j} \right |
|r(n-i)|^2 |r(n+j)|^2.
\end{equation}
The index change  $j=-k$ yields
$$ \sigma (n) = \left (
\sum_{i \neq \pm n} \left | \frac{n-i}{n+i} \right | |r(n-i)|^2
\right )^2. $$ Now the Cauchy inequality and (\ref{3.0}) imply
$$ \sigma (n)  \leq  \|r\|^2
\sum_{i \neq \pm n} \left | \frac{n-i}{n+i} \right |^2 |r(n-i)|^2
\leq  C n^2 \|r\|^2 \left (  \mathcal{E}_n (r) + \frac{\|r\|}{n}
\right )^2. $$ By (\ref{3.2}) and (\ref{3.3}), the latter estimate
proves (\ref{3.1}).
\end{proof}

\begin{Lemma}
\label{lem3.2a} If $r \in \ell^2 (\mathbb{Z}) $ and $ H(n) = \left
(H_{ij} (n) \right ), \; n \in \mathbb{N}, \;n>N, $ are
Hilbert--Schmidt operators such that
$
H_* := \sup_{n>N}\|H(n) \|_{HS} < \infty,
$
then
\begin{equation}
\label{3.1a}
 \sum_{n>N}  \frac{1}{n^2}
 \sum_{i,j\neq \pm n} \left ( \left | \frac{n-i}{n+i} \right
|^{1/2} \left | \frac{n+j}{n-j} \right |^{1/2} |r(n-i)|| r(n+j)|
|H_{ij} (n)| \right )^2
\end{equation}
$$ \leq CH_*^2\|r\|^2  \left (  \left (\mathcal{E}_N (r) \right
)^2 +\frac{\|r\|^2}{N} \right ), $$ where $C$ is an absolute
constant.
\end{Lemma}

\begin{proof}
Let $\Sigma $ denote the sum on the left of (\ref{3.1a}). By
(\ref{3.2}) we have
\begin{equation}
\label{3.7a} \Sigma \leq H_*^2  \sum_{n>N} \frac{1}{n^2} \sigma
(n),
\end{equation}
where $\sigma (n)$ is given by (\ref{3.3}).

Changing the summation index by $k=-j$ we get
\begin{equation}
\label{3.8a} \sigma (n) = \left ( \sum_{i \neq \pm n} \left |
\frac{n-i}{n+i} \right | |r(n-i)|^2 \right )^2 \leq 2 \sigma_1 (n)
+ 2 \sigma_2 (n),
\end{equation}
where
\begin{equation}
\label{3.9a}
  \sigma_1 (n) = \left ( \sum_{i<0, i \neq -n} \left |
\frac{n-i}{n+i} \right | |r(n-i)|^2 \right )^2,  \quad \sigma_2 (n)
= \left ( \sum_{ i \geq 0} \cdots \right )^2.
\end{equation}

By the Cauchy inequality and (\ref{3.01}), we have $$ \sum_{n>N}
\frac{1}{n^2} \sigma_1 (n) \leq \|r\|^2 \left ( \sum_{n>N} \sum_{
i<0, i \neq -n} \frac{8}{(n+i)^2} |r(n-i)|^2 + \sum_{n>N}
\frac{2}{n^2} \|r\|^2 \right ). $$ Now, with $\nu = n-i,$  $$
\sum_{n>N} \sum_{  i< 0, i \neq -n} \frac{8}{(n+i)^2} |r(n-i)|^2
\leq \sum_{\nu
>N} |r(\nu)|^2 \sum_{n \neq \nu/2} \frac{8}{(2n- \nu)^2} \leq
\frac{8\pi^2}{3} \left (\mathcal{E}_N (r) \right )^2, $$ so it
follows that
\begin{equation}
\label{3.10a} \sum_{n>N} \frac{1}{n^2} \sigma_1 (n) \leq
 \|r\|^2 \left (
\frac{8\pi^2}{3} \left (\mathcal{E}_N (r) \right )^2
+\frac{2\|r\|^2}{N} \right ).
\end{equation}

On the other hand, since $|n-i|/|n+i| \leq 1$ for $i \geq 0,$ we
have $  \sigma_2 (n) \leq \|r\|^4.$ Thus,
\begin{equation}
\label{3.11a} \sum_{n>N} \frac{1}{n^2} \sigma_2 (n) \leq
\sum_{n>N} \frac{1}{n^2} \|r\|^4  \leq  \frac{1}{N} \|r\|^4.
\end{equation}
Finally, (\ref{3.7a})--(\ref{3.11a}) imply (\ref{3.1a}), which
completes the proof.
\end{proof}

Let us recall that a weight $\omega $ is called {\em slowly
increasing}, if $ \sup_n \frac{\omega (2n)}{\omega (n)} < \infty.
$

\begin{Lemma}
\label{lem3.3} Suppose $\omega = (\omega (m))_{m \in 2\mathbb{Z}}
$ is a slowly increasing weight such that
\begin{equation}
\label{3.4}  \omega (m) \leq C_1 |m|^{1/4}, \quad m\in
2\mathbb{Z},
\end{equation}
$r= (r(m))_{m\in 2\mathbb{Z}}\in \ell^2 (\omega, 2\mathbb{Z}),$ and
$ H(n) = \left (H_{ij} (n) \right )_{i,j \in n +2\mathbb{Z}} , \; n
\geq n_0, $  are Hilbert--Schmidt operators such that $ H_* =
\sup_{n \geq n_0} \|H (n) \|_{HS} < \infty.$ Then, for  $ N>n_0,$
\begin{equation}
\label{3.6}
 \sum_{n>N}  \frac{(\omega (2n))^2}{n^2}
 \sum_{i,j\neq \pm n} \left ( \left | \frac{n-i}{n+i} \right
|^{1/2} \left | \frac{n+j}{n-j} \right |^{1/2} |r(n-i)|| r(n+j)|
|H_{ij} (n)| \right )^2
 \end{equation}
$$
 \leq C H_*^2 \|r\|^4_{\omega} \left
 ( \frac{ 1}{(\omega(N))^2} +
 \frac{1}{\sqrt{N}} \right ),
$$
 where $C= C(\omega). $
\end{Lemma}

\begin{proof}
Let $\Sigma $ denote the sum on the left of (\ref{3.6}). By
(\ref{3.2}) we have
\begin{equation}
\label{3.7} \Sigma \leq H_*^2  \sum_{n>N} \frac{(\omega
(2n))^2}{n^2} \sigma (n),
\end{equation}
where $\sigma (n)$ is given by (\ref{3.3}), so (\ref{3.8a}) holds,
i.e., $$ \sigma (n) \leq  2\sigma_1 (n) + 2\sigma_2 (n), $$ with $
\sigma_1 (n)$ and $\sigma_2 (n)$ defined by (\ref{3.9a}).

Consider the sequence $\bar{r} = (\bar{r}(m))_{m\in 2\mathbb{Z}} $
defined by $$ \bar{r} (m) = \omega (m) r(m), \quad m\in 2\mathbb{Z}.
$$ Then $\bar{r} \in \ell^2 $ and  $\|\bar{r}\| = \|r\|_{\omega}.$

Taking into account that $$ \bar{r} (n-i) = \omega (n-i) r (n-i)
\geq \omega (n)  r (n-i) \quad \text{if} \; i <0, $$ we estimate
 from above: $$ \sigma_1 (n)
\leq \frac{1}{(\omega (n))^4}\left ( \sum_{  i<0, i \neq -n} \left |
\frac{n-i}{n+i} \right | |\bar{r}(n-i)|^2 \right )^2 $$ $$ \leq
\frac{1}{(\omega (n))^4}\left ( \sum_{  i<0, i \neq -n} \left |
\frac{n-i}{n+i} \right |^2 |\bar{r}(n-i)|^2 \right ) \cdot \|
\bar{r}\|^2 $$ (by the Cauchy inequality).

Since $\omega $ is a slowly increasing weight, there is a constant
$C_0 > 0 $ such that $$ \omega (2n) \leq C_0  \omega (n) \quad
\forall n \in \mathbb{N}. $$ Therefore, in view of (\ref{3.01}), we
have $$ \sum_{n>N} \frac{(\omega (2n))^2}{n^2} \sigma_1 (n) \leq
\|\bar{r}\|^2 \sum_{n>N} \frac{C_0^2}{(\omega (n))^2} \sum_{ i<0, i
\neq -n} \left ( \frac{8}{(n+i)^2} +\frac{2}{n^2} \right )
|\bar{r}(n-i)|^2 $$ $$ \leq \frac{C_0^2\|\bar{r}\|^2}{(\omega
(N))^2} \left ( \sum_{n>N}  \sum_{  i<0, i \neq -n}
\frac{8}{(n+i)^2} |\bar{r}(n-i)|^2 + \sum_{n>N} \frac{2}{n^2}
\sum_{i} |\bar{r}(n-i)|^2 \right ). $$ Since we have, with $\nu =
n-i, $
 $$ \sum_{n>N}
\sum_{  i< 0, i \neq -n} \frac{8}{(n+i)^2} |\bar{r}(n-i)|^2 \leq
\sum_{\nu >N}  |\bar{r}(\nu)|^2 \sum_{n \neq \nu/2} \frac{8}{(2n-
\nu)^2} \leq \frac{8\pi^2}{3} \left (\mathcal{E}_N (\bar{r}) \right
)^2, $$ it follows that
\begin{equation}
\label{3.10} \sum_{n>N} \frac{(\omega (2n))^2}{n^2} \sigma_1 (n)
\leq  \frac{C_0^2 \|\bar{r}\|^4}{(\omega (N))^2} \left (
\frac{8\pi^2}{3} +2  \right ).
\end{equation}

On the other hand, since $|n-i|/|n+i| \leq 1$ for $i \geq 0,$ we
have $$  \sigma_2 (n) \leq \|r\|^4.$$ Thus, by (\ref{3.4}),
\begin{equation}
\label{3.11} \sum_{n>N} \frac{(\omega (2n))^2}{n^2} \sigma_2 (n)
\leq  \sum_{n>N} \frac{C^2_1 (2n)^{1/2}}{n^2} \|r\|^4  \leq 2
\sqrt{2}C_1^2 \frac{\|r\|^4}{\sqrt{N}}.
\end{equation}

Now (since $\|r\| \leq \|\bar{r}\|$) and $\|\bar{r}\| =
\|r\|_{\omega})$ (\ref{3.7})--(\ref{3.10}) imply (\ref{3.6}),
which completes the proof.
\end{proof}

We set, for each sequence $\rho =(\rho (m))_{m \in 2\mathbb{Z}}$ such
that $$\rho (0) =0,\quad \sum_{m \neq 0} (\rho (m))^2/m^2 < \infty,$$
\begin{equation}
\label{3.13} A_k (\rho;n) = \sum_{j_1, \ldots, j_k \neq \pm n}
\frac{|\rho (n-j_1) \rho (j_1 -j_2) \cdots \rho (j_k - n)|}{|n^2
-j_1^2| \cdots |n^2  -j_k^2|}, \quad k,n \in \mathbb{N},
\end{equation}
\begin{equation}
\label{3.14} B_k (\rho;\pm n) = \sum_{j_1, \ldots, j_k \neq \pm n}
\frac{|\rho (\pm n-j_1) \rho (j_1 -j_2) \cdots \rho (j_k \pm
n)|}{|n^2 -j_1^2| \cdots |n^2  -j_k^2|}, \quad k,n \in \mathbb{N},
\end{equation}
and
\begin{equation}
\label{3.15} A(\rho;n) = \sum_{k=1}^\infty A_k (\rho;n), \quad
B(\rho;\pm n) = \sum_{k=1}^\infty B_k (\rho; \pm n).
\end{equation}

By the elementary inequality $$ \frac{1}{|n^2 +z-j^2|} \leq
\frac{2}{|n^2 - j^2|}, \quad j \in n+ 2\mathbb{Z}, \;\; j\neq \pm n,
\;\; |z| \leq n/2, $$ (\ref{2.31}),(\ref{2.34}), (\ref{2.35}) and
(\ref{2.36}) imply,  for $ k \in \mathbb{N}, \;  |z| \leq n/2,$ that
\begin{equation}
\label{3.16} | S^{11}_k (v;n,z)|=| S^{22}_k (v;n,z)|  \leq A_k
(2V;n),
\end{equation}
\begin{equation}
\label{3.16a}  |S^{12}_k (v;n,z)| \leq B_k (2V;-n), \quad |S^{21}_k
(v;n,z)| \leq B_k (2V;n),
\end{equation}
where $ V=(V(m))_{m\in 2\mathbb{Z}} $ is the sequence of the Fourier
coefficients of the potential $v(x) = \sum_{m\in 2\mathbb{Z}} V(m)
\exp (imx).  $  Therefore, by (\ref{2.29}), (\ref{2.33}),
(\ref{2.40}) and (\ref{3.15}), we have
\begin{equation}
\label{3.18} |\alpha_n (v;z)|   \leq A(2V;n), \quad | \beta^\pm_n
(v;z)- V(\pm 2n| \leq B(2V; \pm n), \quad |z|<n/2.
\end{equation}

In view of (\ref{3.13})--(\ref{3.15}) we have
\begin{equation}
\label{3.19} A(\rho,n) = \langle \hat{V} \hat{K}^2 \hat{V} e_{n},
e_n \rangle + \langle \hat{V} \hat{K} \hat{T}(1-\hat{T})^{-1}
\hat{K} \hat{V} e_{n}, e_n \rangle,
\end{equation}
\begin{equation}
\label{3.20} B(\rho,\pm n) = \langle \hat{V}\hat{K}^2 \hat{V}
e_{\mp n}, e_{\pm n} \rangle + \langle \hat{V}\hat{K}
\hat{T}(1-\hat{T})^{-1}\hat{K} \hat{V} e_{\mp n}, e_{\pm n}
\rangle,
\end{equation}
where $\hat{V}$ and $ \hat{K} $ denote, respectively, the
operators with matrix representations
\begin{equation}
\label{3.21} \hat{V}_{ij} = |\rho (i-j)|, \quad \hat{K}_{ij} =
\begin{cases}
\frac{1}{|n^2 - j^2|^{1/2}} \delta_{ij}  & \text{if} \; i,j \neq
\pm n,\\ 0 & \text{if} \;i= \pm n, \; \text{or} \; j =\pm n,
\end{cases}
\end{equation}
and
\begin{equation}
\label{3.22} \hat{T} =\hat{T}_n = \hat{K}\hat{V}\hat{K}.
\end{equation}

The operator   $ \hat{T}$ is a slight modification of the operator
$T$ defined by (\ref{2.3}). Therefore, one can easily see that we
have the same estimate for its Hilbert--Schmidt norm, namely
(compare with (\ref{2.3})),
\begin{equation}
\label{3.23} \|\hat{T}\|_{HS} \leq C \left (
\mathcal{E}_{\sqrt{n}}(r) + \frac{\|r\|}{\sqrt{n}} \right ),
\end{equation}
where $r = (r(m))$ is an $\ell^2$--sequence defined by
\begin{equation}
\label{3.24} r(0) = 0, \quad r(m) = \rho (m)/m \quad \text{for}\;
m \neq 0,
\end{equation}
and $C$ is an absolute constant.

\begin{Lemma}
\label{lem3.4} Under the above assumptions and notations, there is
$ n_0 = n_0 (\rho)$ such that, for $n>n_0,$ we have
\begin{equation}
\label{3.25} A(\rho;n) \leq C n\left ( \mathcal{E}_{\sqrt{n}}(r) +
\frac{\|r\|}{\sqrt{n}} \right )\|r\|, \quad B(\rho;\pm n) \leq C
n\left ( \mathcal{E}_{\sqrt{n}}(r) + \frac{\|r\|}{\sqrt{n}} \right
) \|r\|,
\end{equation}
where $C$ is an absolute constant.
\end{Lemma}

\begin{proof}
By (\ref{3.19}), $$ A(\rho,n) = \langle \hat{V} \hat{K}^2 \hat{V}
e_{n}, e_n \rangle + \langle \hat{V} \hat{K} \hat{T}(1-\hat{T})^{-1}
\hat{K} \hat{V} e_{n}, e_n \rangle = \Sigma_1 +\Sigma_2. $$ In view
of (\ref{3.13})--(\ref{3.15}) and (\ref{3.24}), we have $$\Sigma_1
\leq   \sum_{j\neq \pm n} \frac{|\rho (n-j)\rho (j-n)|}{|n^2 - j^2|}
\leq \sum_{j\neq \pm n} \left |   \frac{n-j}{n+j} \right |
|r(n-j)||r(j-n)| $$ $$ \leq \left (
 \sum_{j\neq \pm n} \left |  \frac{n-j}{n+j} \right |^2
|r(n-j)|^2 \right )^{1/2} \cdot \|r\| \leq C n\left (
\mathcal{E}_{\sqrt{n}}(r) + \frac{\|r\|}{\sqrt{n}} \right ) \|r\|
$$ (by the Cauchy inequality and Lemma \ref{lem3.1}). In an
analogous way we get $$  \Sigma_2 \leq   \sum_{i,j \neq \pm n}
\frac{|\rho (n-i) \rho (j-n)|}{|n^2 -i^2|^{1/2}|n^2 -j^2|^{1/2}}
H_{ij}(n)  $$$$ = \sum_{i,j\neq \pm n} \left | \frac{n-i}{n+i}
\right |^{1/2} \left | \frac{n+j}{n-j} \right |^{1/2} |r(n-i)||
r(n+j)| H_{ij} (n), $$ where $ H(n) = \hat{T} (1-\hat{T})^{-1},$ so
$H_{ij} \geq 0.$ By (\ref{3.24}),  $ \|\hat{T}\|\leq 1/2 $ for
$n>n_0, $ which implies $$ \| H(n) \| \leq 2\|\hat{T}\| \leq 1 \quad
\text{for} \; n >n_0. $$ Thus, by Lemma \ref{lem3.2}, we get
$$ \Sigma_2  \leq Cn \|r\| \left ( \mathcal{E}_n (r) +
\frac{\|r\|}{n}  \right ), $$ where $C$  is an absolute constant.

The obtained estimates of $ \Sigma_1  $ and $ \Sigma_2  $ imply the
first inequality in (\ref{3.25}). We omit the proof of the second
inequality there (which gives an upper bound for $B(\rho, n)),$
because it is practically the same.
\end{proof}

\begin{Proposition}
\label{prop3.1} Let $v $ be a $ H^{-1} $--potential of the form
(\ref{002}) with (\ref{003}), and let (\ref{0016}) define the
sequence of its Fourier coefficients $ V= (V(m))_{m \in
2\mathbb{Z}}.$

There exists $n_0 =n_0 (v)$ such that for $n>n_0$ we have
\begin{equation}
\label{3.31} |\alpha_n (v;z)|,|\beta^\pm_n (v;z)- V(\pm 2n)|  \leq
C n\left ( \mathcal{E}_{\sqrt{n}}(q) +
 \frac{\|q\|}{\sqrt{n}} \right ) \|q\|
\quad \text{for} \; |z| \leq n/2,
\end{equation}
and moreover,
\begin{equation}
\label{3.32} \left |\frac{\partial \alpha_n (v;z)}{\partial
z}\right |, \left |\frac{\partial \beta_n (v;z)}{\partial z}\right
|  \leq 2C \left ( \mathcal{E}_{\sqrt{n}}(q) +
\frac{\|q\|}{\sqrt{n}} \right ) \|q\| \quad \text{for}  \; |z|
\leq n/4,
\end{equation}
where $C$ is an absolute constant.
\end{Proposition}

\begin{proof}
In view of (\ref{3.18}), (\ref{3.31}) follows immediately from
Lemma \ref{lem3.4}.

Since $\alpha_n $ and $\beta_n^\pm $ depend analytically on $z $
for $|z| \leq n/2,$ the Cauchy inequality for their derivatives
proves (\ref{3.32}). Proposition~\ref{prop3.1} is proved.
\end{proof}

\begin{Lemma}
\label{lem3.5} Suppose $\omega =(\omega (m))_{m \in 2\mathbb{Z}}$
is a weight such that
\begin{equation}
\label{3.40} \omega (m) = \omega_1 (m)/m, \quad \text{for} \; \; m
\neq 0,
\end{equation}
where $\omega_1 = (\omega_1 (m))_{m \in 2\mathbb{Z}} $ is a slowly
increasing unbounded weight such that
\begin{equation}
\label{3.41} \omega_1 (m) \leq C_1 |m|^{1/4}, \quad m\in
2\mathbb{Z}.
\end{equation}
If $\rho= (\rho(m))_{m\in 2\mathbb{Z}}\in \ell^2 (\omega,
2\mathbb{Z}),\; \rho (0) =0,$ then there exists $n_0 =n_0
(\|\rho\|_{\omega}, \omega_1)$ such that
\begin{equation}
\label{3.42} \sum_{n>N} (\omega (2n))^2 |B(\rho;\pm n)|^2
 \leq C \|\rho\|_{\omega}^4 \left
 ( \frac{ 1}{(\omega_1 (N))^2} +
 \frac{1}{\sqrt{N}} \right ), \quad N>n_0,
\end{equation}
where $C= C(\omega_1). $
\end{Lemma}

\begin{proof}
By changing the summation indices one can easily see that $$ B_k
(\rho;-n) = B_k (\tilde{\rho};n), $$ where the sequence
$\tilde\rho $ is defined by $\tilde\rho (j) = \rho (-j), $ so we
have $$ B (\rho;-n) = B (\tilde\rho;n). $$ Thus, it is enough to
consider only the case of positive $n$ in (\ref{3.42}).

Consider the sequences $r =(r(m)) $ and $\bar{r} =(\bar{r} (m)) $
defined by
\begin{equation}
\label{3.42a} r(0) =0, \quad r(m) =\rho (m)/m  \quad \text{for} \;
m \neq 0, \quad \bar{r} (m) = \omega_1 (m) r(m).
\end{equation}
Then we have $$r \in \ell^2 (\omega_1, 2\mathbb{Z}), \quad \bar{r}
\in \ell^2 (2\mathbb{Z}), \quad \|\bar{r}\| = \|r\|_{\omega_1}=
\|\rho\|_{\omega}.$$

Since $\omega_1 $ is a slowly increasing weight, there exists $C_0
\geq 1$ such that
\begin{equation}
\label{3.43} \omega_1 (2m) \leq C_0 \omega_1 (m), \quad m \in
2\mathbb{Z}.
\end{equation}
One can easily see that the weight $C_0 \omega_1 (m) $ is
sub--multiplicative. Therefore,
\begin{equation}
\label{3.44} \omega_1 (2n) \leq C_0 \omega_1 (n-j) \omega_1 (n+j),
\quad j \in 2\mathbb{Z}.
\end{equation}

By (\ref{3.20}) we have
\begin{equation}
\label{3.45} B(\rho;n) = \langle \hat{V}\hat{K}^2 \hat{V} e_{-n},
e_{n} \rangle + \langle \hat{V}\hat{K}
\hat{T}(1-\hat{T})^{-1}\hat{K} \hat{V} e_{-n}, e_{n} \rangle
=
\Sigma_3 + \Sigma_4.
\end{equation}
By (\ref{3.14}),(\ref{3.15}),  (\ref{3.21}) and (\ref{3.42a}) $$
\Sigma_3 =  \langle \hat{V}\hat{K}^2 \hat{V} e_{-n}, e_{n} \rangle
 =  \sum_{j \neq \pm n} \frac{|\rho (n-j)|| \rho (j+n)|} {|n^2
-j^2|}   \leq  \sum_{j \neq \pm n} |r(n-j)||r(j+n)|. $$ Therefore,
from (\ref{3.42a}) -- (\ref{3.44}) it follows that
\begin{equation}
\label{3.46} \sum_{n>N} \frac{(\omega_1 (2n))^2}{n^2} \left |
\Sigma_3 \right |^2  \leq \sum_{n>N} \frac{1}{n^2}  \left (\sum_{j
\neq \pm n} C_0 |\bar{r}(n-j)||\bar{r}(j+n)| \right )^2 \leq
\frac{1}{N} C^2_0 \|\bar{r}\|^4.
\end{equation}

On the other hand, (\ref{3.14}),(\ref{3.15}), (\ref{3.22}) and
(\ref{3.42a}) imply that $$  \Sigma_4   =  \langle \hat{V}\hat{K}
H(n) \hat{K} \hat{V} e_{-n}, e_{n} \rangle  \leq \sum_{i,j\neq \pm
n} \frac{|\rho (n-i)||\rho (j+n)|} {|n^2 - i^2|^{1/2}|n^2 -
j^2|^{1/2}} H_{ij}(n) $$ $$ = \sum_{i,j\neq \pm n} \left |
\frac{n-i}{n+i} \right |^{1/2} \left | \frac{n+j}{n-j} \right
|^{1/2} |r(n-i)|| r(n+j)| H_{ij} (n), $$ where $ H(n) =\hat{T}
(1-\hat{T})^{-1}.$ By (\ref{3.23}) and (\ref{3.42a}) we have
\begin{equation}
\label{3.47}  \|\hat{T}\|_{HS} \leq C \left (
\mathcal{E}_{\sqrt{n}}(r) + \frac{\|r\|}{\sqrt{n}} \right ),
\end{equation}
which implies, since $\omega_1 $ is unbounded,
\begin{equation}
\label{3.470} \|\hat{T}\|_{HS} \leq C\left (
\frac{\mathcal{E}_{\sqrt{n}}(\bar{r})}{\omega_1 (\sqrt{n})} +
\frac{\|r\|}{\sqrt{n}} \right ) \leq C\left ( \frac{1}{{\omega_1
(\sqrt{n})}} + \frac{1}{\sqrt{n}} \right )\cdot \|\rho\|_{\omega}
\leq \frac{1}{2} \end{equation}
 for  $n>n_0 =n_0
(\|\rho\|_{\omega}, \omega_1).$ Therefore, for $n>n_0,$  it
follows that $$ \|H(n)\|_{HS} = \|\hat{T} (1-\hat{T})^{-1}\|_{HS}
\leq
  2\|\hat{T}\|_{HS} \leq 1.
$$ Hence, by Lemma \ref{lem3.3}, we get
\begin{equation}
\label{3.48} \sum_{n>N} \frac{(\omega_1 (2n))^2}{n^2} \left |
\Sigma_4 \right |^2  \leq C \|\bar{r}\|^4 \left (
\frac{1}{(\omega_1 (N))^2} + \frac{1}{\sqrt{N}} \right ).
\end{equation}
Now   (\ref{3.40}), (\ref{3.46}) and (\ref{3.48}) imply
(\ref{3.42}), may be with another $n_0.$ This completes the proof.
\end{proof}

In Lemma \ref{lem3.5} we assume that the weight $\omega_1 $ is
unbounded, and this assumption is used to get (\ref{3.470}). But
if $\omega_1  $ is bounded, say $\omega_1 (k) \equiv 1, $
(\ref{3.47}) implies $ \|\hat{T}\|_{HS} \leq 1/2 $ for $ n> n_0
(r),$ so we have $ \|H(n)\|_{HS} \leq 1 $ for $n>n_0.$

Hence, by Lemma \ref{lem3.2a}, we get
\begin{equation}
\label{3.48a} \sum_{n>N} \frac{1}{n^2} \left | \Sigma_4 \right |^2
\leq C\|r\|^2  \left (  \left (\mathcal{E}_N (r) \right )^2
+\frac{\|r\|^2}{N} \right ),
\end{equation}
where $C$ is an absolute constant. Since all other estimates in the
proof of Lemma~\ref{lem3.5} hold with $r=\bar{r},$ we get the
following statement.

\begin{Lemma}
\label{lem3.5a} If $r=(r(m))_{m\in 2\mathbb{Z}}\in \ell^2
(2\mathbb{Z}) $ and $\rho= (\rho(m)),\; \rho (m) = m r(m), $ then
\begin{equation}
\label{3.42b} \sum_{n>N} \frac{1}{n^2} |B(\rho;\pm n)|^2
 \leq
 C\|r\|^2  \left (  \left (\mathcal{E}_N (r) \right )^2
+\frac{\|r\|^2}{N} \right ),
\end{equation}
where $C$ is an absolute constant.

\end{Lemma}

If $ \omega = (\omega (m))_{m\in 2\mathbb{Z}} $ is a weight, then
we set $\omega \rho = (\omega (m) \rho (m))_{m\in 2\mathbb{Z}} .$

\begin{Lemma}
\label{lem3.6} In the notations (\ref{3.14})--(\ref{3.15}), if
$\omega $ is a sub--multiplicative weight on $2\mathbb{Z},$ then
\begin{equation}
\label{3.49} B(\rho;\pm n) \omega (2n) \leq B(\omega \rho;\pm n).
\end{equation}
\end{Lemma}

\begin{proof} In view of (\ref{3.15}), it is enough to show that
\begin{equation}
\label{3.50} B_k (\rho;\pm n) \omega (2n) \leq B_k (\omega
\rho;\pm n), \quad k \in \mathbb{N}.
\end{equation}
Since $\omega $ is a sub--multiplicative weight,  for each
$k$-tuple of indices $j_1, \ldots, j_k$ we have $$ \omega(2n) =
\omega(-2n) \leq  \omega(\pm n -j_1) \omega(j_1-j_2) \cdots
\omega(j_{k-1} -j_k) \omega(j_k \pm n). $$ Thus $$ \rho (\pm
n-j_1) \rho (j_1-j_2) \cdots \rho(j_k \pm n) \omega (2n) \leq $$
$$ [\rho(\pm n-j_1) \omega( \pm n-j_1)][\rho(j_1-j_2) \omega(j_1 -
j_2)] \cdots [\rho (j_k  \pm n) \omega (j_k  \pm n)], $$ which
implies (\ref{3.50}).  This proves Lemma~\ref{lem3.6}.

\end{proof}

\begin{Proposition}
\label{prop3.2} Suppose $ \Omega = (\Omega (k))_{k\in \mathbb{Z}}$
is a weight of the form
\begin{equation}
\label{3.51}
 \Omega (k) = \frac{ \Omega_1 (k) \Omega_2 (k)}{k},
\quad k\in \mathbb{Z},
\end{equation}
where $\Omega_1 $ is a slowly increasing weight such that
\begin{equation}
\label{3.52}
  \Omega_1 (k) \leq C_1 |k|^{1/4},
\end{equation}
and $\Omega_2 $ is a sub--multiplicative weight.

If $v \in H (\Omega),$ then, for large enough N, we have
\begin{equation}
\label{3.54} \sum_{n>N} \left ( | \beta^-_n (z)- V(-2n)|^2+|
\beta^+_n (z)- V(2n)|^2 \right ) (\Omega (n))^2
\end{equation}
$$ \leq C \left ( \frac{1}{(\Omega_1 (N))^2} + \frac{1}{\sqrt{N}}
\right ) \|v\|^4_{\Omega}, \quad |z|<n/2,$$ where $C=C(\Omega_1)
\geq 1.$
\end{Proposition}

\begin{proof}
Consider the weight $\omega =(\omega (m))_{m\in 2\mathbb{Z}},$
where $ \omega (m) = \Omega (m/2).$ Then we have $\omega (m)
=2\omega_1 (m) \omega_2 (m)/m, $ where $\omega_1 $ is a slowly
increasing weight such that $\omega_1 (m) = \Omega_1 (m/2), $ and
$\omega_2 $ is a sub--multiplicative weight such that $\omega_2
(m) = \Omega_2 (m/2). $

Consider also the sequence $\rho = (\rho(m))_{m \in 2\mathbb{Z}},$
defined by $$ \rho (m)= 2\max (|V(-m)|,|V(m)|).$$ Then we have
$\rho \in \ell^2 (\omega),$ and moreover,
\begin{equation}
\label{3.55} \frac{1}{4}\| \rho \|_{\omega} \leq \|v\|_{\Omega}
\leq \|\rho\|_{\omega}.
\end{equation}

Since $\rho (-m) =\rho (m),$ we have $B(\rho; - n) = B(\rho;n),$
and therefore, by (\ref{3.18}), $$ | \beta^-_n (z)- V(-2n)|+|
\beta^+_n (z)- V(2n)| \leq 2B(\rho;n).$$ Thus, it is enough to
estimate $ \sum_{n>N} |B(\rho;n)|^2 (\omega (2n))^2.$

By Lemma \ref{lem3.6}, we have $$ B(\rho;n) \omega_2 (2n) \leq
B(\omega_2 \rho;n),$$ and therefore, $$ \sum_{n>N} |B(\rho;n)|^2
(\omega (2n))^2 \leq \sum_{n>N} |B(\omega_2 \rho;n)|^2
\frac{(\omega_1 (2n))^2}{n^2}.$$

In view of (\ref{3.42}) in Lemma \ref{lem3.5} and the identity $
\|\omega_2 \rho\|_{\omega_1/n} = \|\rho \|,$
 the latter sum
does not exceed $$ C \left ( \frac{1}{(\Omega_1 (N))^2} +
\frac{1}{N} \right ) \|\rho\|^4_{\omega}, $$ where $C=C(\omega_1).
$ In view of (\ref{3.55}), this completes the proof of
Proposition~\ref{prop3.2}.

\end{proof}

\section{Estimates for $\gamma_n.$ }

In this section we give estimates of $\gamma_n$ from above and
below in terms of matrix elements (\ref{2.27}) of operators $S$ in
Basic Equation (\ref{2.28}), i.e., in terms of $\alpha_n (v;z), \;
\beta^\pm_n (v;z)$ defined in (\ref{2.40}).

The proofs are essentially the same as in the case of
$L^2$--potentials, provided the necessary a priori estimates of
$\alpha_n (v;z), \; \beta^\pm (v;z) $ and their derivatives are
proved (which is done in Section 4, Proposition \ref{prop3.1}).

By Proposition \ref{prop004}, if the potential $v\in
H^{-1}_{loc}(\mathbb{R})$ is $\pi$--periodic, then the operator
$L_{Per^\pm}$ has exactly two eigenvalues $\lambda^-_n$ and
$\lambda^+_n$ in the disc $D_n = \{\lambda =n^2 + z, \;|z|<n/4 \}$
(counted with their multiplicity, periodic for even $n,$ or
antiperiodic for odd $n$).

By Lemma \ref{lem2.2} and Remark \ref{rem21.1}, the numbers
\begin{equation}
\label{27.2a} z_n^- = \lambda^-_n - n^2 \quad \text{and} \quad
z_n^+ = \lambda^+_n - n^2
\end{equation}
are eigenvalues of the operator $S$  defined in (\ref{2.8}), and
therefore, $ z_n^+ $ and $z^-_n $ are roots of the basic equation
(\ref{2.28}). Let us rewrite (\ref{2.28}) in the form
\begin{equation}
\label{27.3a} (\zeta_n (z))^2 = \beta_n^- (z) \cdot \beta_n^+ (z),
\end{equation}
where
\begin{equation}
\label{27.4} \zeta_n (z) = z -\alpha_n (v;z).
\end{equation}

By  Proposition \ref{prop3.1}, (\ref{3.32}),
\begin{equation}
\label{27.5} \sup_{[z_n^-, z_n^+]} |\partial \alpha_n/\partial z |
\leq \varepsilon_n, \quad  \sup_{[z_n^-, z_n^+]} |\partial
\beta^\pm_n/\partial z | \leq \varepsilon_n,
\end{equation}
where $[z_n^-, z_n^+] $ denotes the segment with end points $z_n^-
$ and $ z_n^+, $ and
\begin{equation}
\label{27.5a}  \varepsilon_n = 2C \left ( \mathcal{E}_{\sqrt{n}}
(q) + \frac{\|q\|}{\sqrt{n}} \right ) \|q\|  \to 0  \quad
\text{as} \;\; n \to \infty.
\end{equation}
Therefore, in view of (\ref{27.4}), $$ |z_n^+ - z_n^-| \leq
|\zeta_n (z^+_n) -\zeta_n (z_n^-)| + |\alpha (z_n^+) - \alpha
(z_n^-)| $$ $$ \leq |\zeta_n (z^+_n) -\zeta_n (z_n^-)|+
\varepsilon_n \cdot |z_n^+ - z_n^-|, $$ which yields
$$(1-\varepsilon_n)|z_n^+ - z_n^-| \leq |\zeta_n (z^+_n) -\zeta_n
(z_n^-)|.$$

On the other hand, in view of (\ref{27.5}), the identity $$
\zeta_n (z^+_n) -\zeta_n (z_n^-) = \int_{z_n^-}^{z_n^+}
(1-\partial \alpha_n/\partial z) dz, $$  implies $$|\zeta_n
(z^+_n) -\zeta_n (z_n^-)| \leq (1+\varepsilon_n)|z_n^+ - z_n^-|.$$
Thus, we have the following two--side estimate:
\begin{equation}
\label{27.6} (1-\varepsilon_n)|z_n^+ - z_n^-| \leq |\zeta_n
(z^+_n) -\zeta_n (z_n^-)| \leq (1+\varepsilon_n)|z_n^+ - z_n^-|.
\end{equation}

\begin{Lemma}
\label{lem4.1} In above notations, for large enough $n,$ we have
 \begin{equation}
\label{4.1} \gamma_n =|\lambda^+_n - \lambda_n^-|\leq (1+\eta_n)
(|\beta_n^- (z^*_n) |+|\beta^+_n (z^*_n)|), \quad z^*_n =
\frac{\lambda^+_n + \lambda_n^-}{2} -n^2,
\end{equation}
with $\eta_n \to 0 $ as $n \to \infty. $
\end{Lemma}

\begin{proof}
By Equation (\ref{27.3a}), $$ |\zeta_n (z)| \leq \frac{1}{2}
(|\beta_n^- (z)| + |\beta_n^+ (z)|) \quad \text{for} \quad z=
z_n^{\pm}.$$ Thus, by (\ref{27.6}), $$ (1-\varepsilon_n)|z_n^+ -
z_n^-| \leq \frac{1}{2}(|\beta_n^-(z_n^-)|+
|\beta_n^-(z_n^+)|+|\beta_n^+(z_n^-)| +|\beta_n^+(z_n^+)|). $$ By
the estimates for $\partial \beta /\partial z $ given in Proposition
\ref{prop3.1}, (\ref{3.32}),
 we have $$
|\beta_n^\pm (z^\pm_n) -\beta_n^\pm (z^*_n)| \leq \varepsilon_n
\cdot |z_n^+ - z_n^-|/2, $$ where one may assume that
$\varepsilon_n $ is the same as in (\ref{27.5}) and (\ref{27.5a}).
Thus we get $$(1-2\varepsilon_n)|z_n^+ - z_n^-| \leq |\beta_n^+
(z^*_n)|+ |\beta_n^- (z^*_n)|,$$ which, in view of (\ref{27.2a}),
implies (\ref{4.1}). Lemma~\ref{lem4.1} is proved.
\end{proof}

\subsection{Estimates of $\gamma_n $ from below}

In the previous Lemma \ref{lem4.1} were obtained estimates of
$\gamma_n $ from above in terms of  $|\beta_n^\pm|.$ The next
statement gives estimates of $\gamma_n $ from below.

\begin{Lemma}
\label{lem32.1} In the notations of Lemma \ref{lem4.1}, there
exists a sequence $\eta_n \downarrow 0 $ such that, for large
enough $|n|, $ if $\gamma_n \neq  0$ and $ \beta_n^- (z_n^+) \cdot
\beta_n^+ (z_n^+) \neq 0,  $ then
\begin{equation}
\label{32.01}  \gamma_n  \geq \left (\frac{2\sqrt{t_n}}{1+t_n}
-\eta_n \right )  \left ( |\beta^-_n (z_n^*)| +|\beta^+_n (z_n^*)|
\right ),
\end{equation}
where
\begin{equation}
\label{32.02} t_n = |\beta_n^+ (z_n^+)| /|\beta_n^- (z_n^+)|.
\end{equation}
\end{Lemma}

\begin{proof}. By Proposition \ref{prop3.1}, (\ref{3.32}),
\begin{equation}
\label{32.1} \sup_{[z_n^-, z_n^+]} |\partial \beta_n^\pm/\partial
z | \leq \varepsilon_n, \qquad   \varepsilon_n \downarrow 0,
\end{equation}
where $[z_n^-, z_n^+] $ denotes the segment determined by $z_n^- $
and $ z_n^+. $ Since $$ \beta_n^\pm (z) - \beta_n^\pm (z_n^*) =
\int_z^{z_n^*} \frac{d}{d z} (\beta_n^\pm (z)  ) d z, $$
(\ref{32.1}) implies that $$ |\beta_n^\pm (z) - \beta_n^\pm
(z_n^*) | \leq \varepsilon_n |z-z_n^*| \leq \varepsilon_n |z_n^+ -
z_n^-| \quad \text{for} \quad z \in [z_n^-, z_n^+]. $$ Thus, for
$z \in [z_n^-, z_n^+], $
\begin{equation}
\label{32.2} |\beta_n^\pm (z_n^*)| - \varepsilon_n |z_n^+ - z_n^-|
\leq
 |\beta_n^\pm (z) | \leq |\beta_n^\pm
(z_n^*)| + \varepsilon_n |z_n^+ - z_n^-|
\end{equation}

By  (\ref{27.4}) and (\ref{27.6}), we have
\begin{equation}
\label{32.3} (1-\varepsilon_n ) |z_n^+ - z_n^-| \leq |\zeta_n^+ -
\zeta_n^-| \leq (1+ \varepsilon_n ) |z_n^+ - z_n^-|,
\end{equation}
where  $$ \zeta_n^+ = \zeta_n (z_n^+), \quad
 \zeta_n^-= \zeta_n (z_n^-). $$
On the other hand, by (\ref{27.3a}) (i.e., by the basic equation
(\ref{2.28})), we have $$ (\zeta_n^+)^2 = \beta_n^+ (z_n^+)
\beta_n^- (z_n^+), \quad (\zeta_n^-)^2 = \beta_n^+ (z_n^-)
\beta_n^- (z_n^-), $$ and therefore,
\begin{equation}
\label{32.4}  (\zeta_n^+)^2 - (\zeta_n^-)^2 = \int_{z_n^-}^{z_n^+}
\frac{d}{dz} [\beta_n^+ (z) \beta_n^- (z)] dz.
\end{equation}
By (\ref{32.1}) and (\ref{32.2}), we have $$ \sup_{[z_n^-, z_n^+]}
\left | \frac{d}{dz} [\beta_n^+ (z) \beta_n^- (z)] \right | \leq
\varepsilon_n \left ( |\beta_n^+ (z_n^*)|+|\beta_n^- (z_n^*)| +
2\varepsilon_n |z_n^+ - z_n^- | \right ). $$ In view of
(\ref{32.3}) and (\ref{32.4}), we get $$ |\zeta_n^+ +\zeta_n^-|
\cdot |\zeta_n^+ - \zeta_n^-| \leq \varepsilon_n \left (
|\beta_n^+ (z_n^*)|+|\beta_n^- (z_n^*)| + 2\varepsilon_n  |z_n^+ -
z_n^- | \right ) |z_n^+ - z_n^- |  $$ $$ \leq \varepsilon_n \left
( |\beta_n^+ (z_n^*)|+|\beta_n^- (z_n^*)| + 2\varepsilon_n |z_n^+
- z_n^- | \right ) \frac{|\zeta_n^+ -
\zeta_n^-|}{1-\varepsilon_n}.$$ Since $\varepsilon_n \to 0, $ we
may assume that $\varepsilon_n < 1/2.$ Then, $1/(1-\varepsilon_n)
\leq 2, $ and the latter inequality implies
\begin{equation}
\label{32.5} |\zeta_n^+ +\zeta_n^-| \leq 2\varepsilon_n  \left (
|\beta_n^+ (z_n^*)|+|\beta_n^- (z_n^*)| \right ) + 2\varepsilon_n
|z_n^+ - z_n^- |.
\end{equation}
By (\ref{32.02}), we have
\begin{equation}
\label{32.6} |\zeta_n^+| = \sqrt{|\beta_n^+ (z_n^+)||\beta_n^-
(z_n^+)|} = \frac{\sqrt{t_n}}{1+t_n} (|\beta_n^+ (z_n^+)| +
|\beta_n^- (z_n^+)|).
\end{equation}
Therefore, by (\ref{32.2}) (since $ \frac{\sqrt{t}}{1+t} \leq 1/2$
for $t\geq 0$) we get
\begin{equation}
\label{32.7} |\zeta_n^+| \geq  \frac{\sqrt{t_n}}{1+t_n}
(|\beta_n^+ (z_n^*)| + |\beta_n^- (z_n^*)|) - \varepsilon_n |z_n^+
- z_n^- |.
\end{equation}
Now, from (\ref{32.5})--(\ref{32.7}) it follows that $$ |\zeta_n^+
- \zeta_n^- | = |2 \zeta_n^+ - (\zeta_n^+ + \zeta_n^-)| \geq 2
|\zeta_n^+ | - |\zeta_n^+ + \zeta_n^- | $$ $$ \geq \left (
\frac{2\sqrt{t_n}}{1+t_n}- 2\varepsilon_n \right ) (|\beta_n^+
(z_n^*)|+|\beta_n^- (z_n^*)|) - 4 \varepsilon_n |z_n^+ - z_n^-|.$$
By (\ref{32.3}), this leads to $$ (1+5 \varepsilon_n)|z_n^+ -
z_n^-| \geq \left (\frac{2\sqrt{t_n}}{1+t_n}- 2\varepsilon_n
\right ) (|\beta_n^+ (z_n^*)|+|\beta_n^- (z_n^*)|). $$ Taking into
account that $ (1+5\varepsilon)^{-1} \geq 1- 5\varepsilon$ and
$\frac{2\sqrt{t_n}}{1+t_n} \leq 1, $ we obtain $$ \gamma_n =|z_n^+
- z_n^-| \geq (1- 5\varepsilon_n) \left (
\frac{2\sqrt{t_n}}{1+t_n}- 2\varepsilon_n \right ) (|\beta_n^+
(z_n^*)|+|\beta_n^- (z_n^*)|) $$ $$ \geq \left (
\frac{2\sqrt{t_n}}{1+t_n}- 7\varepsilon_n \right ) (|\beta_n^+
(z_n^*)|+|\beta_n^- (z_n^*)|). $$ Thus  (\ref{32.01}) holds with
$\eta_n = 7 \varepsilon_n.$ Lemma~\ref{lem32.1} is proved.
\end{proof}

If the potential $v$ is real--valued, then we have the following
two--side estimate of $\gamma_n. $

\begin{Theorem}
\label{thm4.1}  Suppose $v$ is a periodic real--valued $H^{-1}$
potential, $L$ is the corresponding self--adjoint
Hill--Schr\"odinger operator and $(\gamma_n)$ is the gap sequence of
$L.$ Then there exists a sequence $\eta_n \downarrow 0 $ such that,
for $n \geq n_0 (v), $
\begin{equation}
\label{32.0} (1-\eta_n) \left ( |\beta^-_n (z_n^*)| +|\beta^+_n
(z_n^*)| \right ) \leq |\gamma_n | \leq (1+\eta_n) \left (
|\beta^-_n (z_n^*)| +|\beta^+_n (z_n^*)| \right )
\end{equation}
\end{Theorem}

This result is known in the case of $L^2$--potentials (see
Theorem~8 in \cite{DM3}, or Theorem 50 in \cite{DM15}).

\begin{proof} The right inequality in (\ref{32.0})
has been proved in Lemma \ref{lem4.1} for arbitrary
(complex--valued) potentials.

Since $L$ is self--adjoint, we know, by Part (b) of Lemma
\ref{lem2.3}, that $$
 |\beta_n^+ (z_n^+)|=|\beta_n^-(z_n^+)|.
$$

If $|\beta_n^+ (z_n^+)|=|\beta_n^-(z_n^+)|\neq 0 $ and $ \gamma_n
\neq 0, $ then the left inequality in (\ref{32.0}) follows
immediately from Lemma \ref{lem32.1}.

If $|\beta_n^+ (z_n^+)|=|\beta_n^-(z_n^+)| = 0  $ for some $n,$
then $\lambda^+_n = \lambda^0 + z^+_n $  is an eigenvalue of
geometric multiplicity 2 of the operator $P^0L^0 P^0+
S(\lambda^+): E^0 \to E^0.  $ Therefore, by Remark \ref{rem21.1},
$\lambda^+_n $ is an eigenvalue of geometric multiplicity 2 of the
operator $L,$ so $\gamma_n = 0 $ and $ z_n^* = z_n^+.$ Thus
(\ref{32.0}) holds.

If $ \gamma_n = 0 $ for some $n,$ then (since $L$ is
self--adjoint) $\lambda^+_n $ is an eigenvalue of $L$ of geometric
multiplicity 2. Therefore,  by Remark~\ref{rem21.1}, $\lambda^+_n
$ is an eigenvalue of geometric multiplicity 2 of the operator
$P^0L^0 P^0+ S(\lambda^+).$  Then the off--diagonal entries of the
matrix representation $S(\lambda^+_n) $ are zeros, i.e., we have $
\beta_n^+ (z_n^+)=\beta_n^-(z_n^+) = 0 $ and (\ref{32.0}) becomes
trivial because $ z_n^* = z_n^+.$ Theorem~\ref{thm4.1} is proved.
\end{proof}

{\em Remark.} We used to write the Fourier expansion of a potential
$v$ in the form
$
 v \sim   \sum_{m \in 2\mathbb{Z}} V(m) e^{imx}.
$ Now, for convenience, we set
\begin{equation}
\label{00}
 v_k = V(2k), \quad k \in \mathbb{Z},
\end{equation}
and define, for every weight $\Omega = (\Omega (k))_{k \in
\mathbb{Z}},$
\begin{equation}
\label{000} \|v\|^2_{\Omega} = \sum_{k \in \mathbb{Z}} |v_k|^2
(\Omega (k))^2.
\end{equation}

The next theorem generalizes a series of results about asymptotic
behavior of $\gamma_n $ (see and compare Theorem 41 in [13] in the
case of $L^2$-potentials $v$).

\begin{Theorem}
\label{thm4.2} Suppose $L= L^0 + v(x), \; L^0 = -d^2/dx^2, $ is a
periodic Hill--Schr\"odinger operator on $I= [0,\pi] $  with $H^{-1}
(I) $--potential $\displaystyle v(x) = \sum_{k\in \mathbb{Z}} v_k
e^{2kix}.$  Then, for $n> n_0 (v), $ the operator $L$ has in the
disc of center $n^2 $  and radius $n/4 $ exactly two (counted with
their algebraic multiplicity) periodic (for even $n$), or
anti--periodic (for odd $n$) eigenvalues $\lambda^+_n $ and $
\lambda^-_n. $ Moreover, for each  weight $\Omega = (\Omega (k))_{k
\in \mathbb{Z}} $ of the form
\begin{equation}
\label{4.4} \Omega (k) = \tilde{\Omega}(k)/k,  \quad  k \neq 0,
\end{equation}
where  $\tilde{\Omega}$  is a submultiplicative weight, we have
\begin{equation}
\label{4.2} \sum_{k\in \mathbb{Z}} |v_k|^2 (\Omega (k))^2 < \infty
\; \Rightarrow \sum_{n>n_0 (v)} |\gamma_n |^2 (\Omega(n))^2 <
\infty,
\end{equation}
where $ \gamma_n =\lambda_n^+ -\lambda_n^-. $ Moreover,
\begin{equation}
\label{4.5} \sum_{n>n_0 (v)} |\gamma_n |^2 (\Omega(n))^2 \leq  C_1
\|v\|^4_{\Omega} + 4 \|v\|^2_{\Omega}.
\end{equation}
where $C_1 = C_1 (\Omega).$
\end{Theorem}

\begin{proof}
By Lemma \ref{lem4.1}, we have, for $n>n_0 (v),$ $$ |\gamma_n |
\leq 2 \left ( |\beta_n^- (v;z_n^*)| +|\beta_n^+(v;z_n^*)| \right
), $$ where $z_n^* = \frac{\lambda_n^- + \lambda_n^+}{2} - n^2. $

Therefore, by Proposition \ref{prop3.2}, considered with $\Omega_1
\equiv 1,$ we get
\begin{multline}
\label{5.21} \sum_{n>n_0 (v)} |\gamma_n |^2 (\Omega(n))^2  \leq
4\sum_{n>n_0 (v)} \left ( |\beta_n^- (v;z_n^*)|
+|\beta_n^+(v;z_n^*)| \right )^2
 (\Omega(n))^2\\
\leq 8\sum_{n>n_0 (v)} \left ( |\beta_n^- (v;z_n^*)-v_{-n}|
+|\beta_n^+(v;z_n^*)-v_n)| \right )^2
 (\Omega(n))^2\\
 +
2\sum_{n>n_0 (v)}   (|v_{-n}|+|v_n|)^2 (\Omega(n))^2 \leq
32C\|v\|^4_{\Omega} + 4 \|v\|^2_{\Omega} < \infty
\end{multline}
with $C=C(\Omega),$ which completes the proof.
\end{proof}

\section{Main results for real--valued potentials}

In this section we present our main results on the relationship
between spectral gaps rate of decay and potential smoothness for
Hill--Schr\"odinger operators with real--valued periodic singular
potentials. However most of the proofs are carried out for arbitrary
potentials (see formulas (\ref{4.12})--(\ref{4.13a}), Lemma
\ref{lem4.2} and Proposition \ref{prop33.1}).

In Theorems 9 and 10, and Section 5.2 in \cite{DM3}, and Theorem 54
in \cite{DM15}, it is proved that the inverse of the implication
(\ref{4.2}) holds for real--valued $L^2$--potentials $v$ and
(log--concave) submultiplicative weights $\Omega.$ Now we extend
this result to the case of {\em singular} potentials and a wider
class of weights.

By Lemma \ref{lem2.1}, for each periodic potential $$v \in H^{-1}
([0,\pi]),\quad v(x) = \sum_{k \in \mathbb{Z}} v_k e^{2ikx}, $$
there exists $n_0 = n_0 (v)$ such that the constructions of Section
3 work for $n>n_0.$ In particular, the numbers $\beta^\pm_n (v;z) $
are well--defined if $n > n_0 (v) $ and  $ |z| \leq n.$

We set, for $N>n_0 (v), $
\begin{equation}
\label{4.12} \Phi_N (v) = \sum_{n>N} \left ( [\beta^-_n(v;z^*_n
(v)) - v_{-n}] e^{-2inx} + [\beta^+_n(v;z^*_n (v))-v_n ] e^{2inx}
\right ),
\end{equation}
where $ z^*_n (v) = \frac{\lambda_n^- (v) + \lambda_n^+ (v)}{2} -
n^2. $ Consider the mapping
\begin{equation}
\label{4.13}
 A_N (v) = v + \Phi_N (v),
\end{equation}
or
\begin{equation}
\label{4.13.0}
 A_N (v) = H_N (v) + T_N (v),
\end{equation}
with a "head"
\begin{equation}
\label{4.13b}
 H_N (v) =
\sum_{n\leq N} \left ( v_{-n} e^{-2inx} + v_n  e^{2inx} \right )
 \end{equation}
and a "tail"
\begin{equation}
\label{4.13a}
 T_N (v) =
\sum_{n>N} \left ( \beta^-_n(v;z^*_n (v))  e^{-2inx} +
\beta^+_n(v;z^*_n (v)) e^{2inx} \right ).
 \end{equation}
As a finite sum, $ H_N (v) $ is in $H(\Omega)$ for any $\Omega.$ If
$v$ is a real--valued potential in $H^{-1}$ then we have, by
(\ref{4.13}) and Theorem \ref{thm4.1},
\begin{equation}
\label{4.14a} (\gamma_n) \in \ell^2 (\Omega) \Rightarrow
(|\beta^-_n (z^*_n)| +|\beta^+_n (z^*_n)|) \in \ell^2 (\Omega)
\Rightarrow T_N (v) \in H(\Omega),
\end{equation}
and therefore,
\begin{equation}
\label{4.14} (\gamma_n) \in \ell^2 (\Omega) \Rightarrow
 A_N (v) \in H(\Omega),
\end{equation}
for every weight $\Omega.$

Thus, the inverse of the implication (\ref{4.2}) will be proved if
we show that
\begin{equation}
\label{4.15} A_N (v)  \in H(\Omega) \;\;\Rightarrow \;\; v \in
H(\Omega).
\end{equation}

If $v$ is a real--valued potential, then the operator $L= L^0 +v $
is self--adjoint, its periodic and anti--periodic spectra are
real, so the numbers $ z^*_n = \frac{1}{2} ( \lambda^+_n -
\lambda^-_n ) - n^2 $  are real.

Therefore, by (\ref{2.40}) and Lemma \ref{lem2.3} we have $$
\beta^-_n (v;z^*_n) = \overline{\beta^+_n (v;z^*_n)}. $$ Thus, in
view of (\ref{4.12}) and (\ref{4.13}),
\begin{equation}
\label{4.16} v \;\; \mbox{is real--valued}  \quad \Rightarrow
\quad T_N (v), \, \Phi_N (v), \, A_N (v)   \;\;  \mbox{are
real--valued}.
\end{equation}

For each weight $\Omega $ on $\mathbb{Z}$ we denote  by $
\mathcal{B}_r^\Omega $ the ball of complex--valued potentials $$
 \mathcal{B}_r^\Omega =\{v \in H (\Omega) \; :
 \;\; \|v\|_{\Omega} \leq r\}. $$

The following lemma plays a crucial role in the proof of the
inverse  of (\ref{4.2}).

\begin{Lemma}
\label{lem4.2}
 Let  $\Omega_1 $ be a slowly increasing unbounded weight such that
\begin{equation}
\label{4.23}
 \Omega_1 (k) \leq C_1 |k|^{1/4}.
\end{equation}
Then there exist a sequence of positive numbers $(r_N)_{N\in
\mathbb{N}}, \; r_N \nearrow \infty, $ and $N^* = N^* (\Omega_1) \in
\mathbb{N} $ such that the mapping $\Phi_N $  is well defined on the
ball  $ \mathcal{B}^{\Omega_1}_{3r_N}  $ for $N >N^*. $

Moreover,  if $ \Omega = (\Omega (k))_{k\in \mathbb{Z}}$ is a
weight of the form
\begin{equation}
\label{4.24}
 \Omega (k) =  \frac{\Omega_1 (k) \Omega_2 (k)}{k},
\quad k\in \mathbb{Z},
\end{equation}
where $\Omega_2 $ is a sub--multiplicative weight,  then  the
mapping
 $\Phi_N : \mathcal{B}^\Omega_{3r_N}  \to H (\Omega)  $
is well defined for $N > N^* $ and has the following properties:
\begin{equation}
\label{4.25} \|\Phi_N (v_1) - \Phi_N (v_2) \|_{\Omega} \leq
\frac{1}{2} \|v_1 -v_2 \|_\Omega \quad \text{for} \quad v_1, v_2
\in \mathcal{B}_{r_N}^\Omega,
\end{equation}
\begin{equation}
\label{4.26} \frac{1}{2} \|v_1 -v_2 \|_\Omega  \leq  \|A_N (v_1)
-A_N (v_2) \|_\Omega  \leq  \frac{3}{2} \|v_1 -v_2 \|_\Omega \quad
\text{for} \quad v_1, v_2 \in \mathcal{B}_{r_N}^\Omega,
\end{equation}
\begin{equation}
\label{4.27}  A_N \left (\mathcal{B}_{r_N}^\Omega \right ) \supset
\mathcal{B}_{r_N/2}^\Omega.
\end{equation}
\end{Lemma}

\begin{proof}
 By  Proposition \ref{prop004}, for each $H^{-1}$ periodic
potential $v$ there exists $ N^* = N^* (v) $ such that, for $n>N^*,$
the operator $ L = L^0 + v $ has two (counted according to their
algebraic multiplicity, periodic for even $n,$ and antiperiodic for
odd $n$ ) eigenvalues $\lambda_n^- (v) $ and $\lambda_n^+ (v) $ in
the disk $D_n =\{z: \; |z-n^2| < n/4 \}. $ On the other hand, in
view of Lemma \ref{lem2.1}, if $ N^*$ is large enough then all
constructions of Section 3 hold.

Moreover, one can choose  $N^* $ depending only on $\Omega (k)$ and
$\|v\|_{\omega_1},$  where $\omega_1 (k) = \Omega (k)/k, \; k \neq
0. $

Indeed, by the proof of Theorem 21 and Lemma 19, (5.30) in
\cite{DM16}, and by the proof of Lemma \ref{lem2.1}, (\ref{2.5}), it
is enough to choose $N^* $ so that
\begin{equation}
\label{4.28} \kappa_n (v):= C \left ( \mathcal{E}_{\sqrt{n}}(q) +
\|q\|/\sqrt{n} \right ) \leq 1/2 \quad \text{for} \;\;  n \geq N^*,
\end{equation}
where $q=(q(m)) $ is defined in (\ref{003})   and $C$ is an absolute
constant. So,
\begin{equation}
\label{4.29} \|v\|^2_{\omega_1} = \sum \left (|q(-k)|^2+|q(k)|^2
\right ) (\Omega_1 (k))^2,
\end{equation}
and therefore, since $\Omega_1 (k) $ is monotone increasing and
$\Omega_1 (k) \geq 1,$
\begin{equation}
\label{4.30} \|q\| \leq  \|q\|_{\Omega_1} = \|v\|_{\omega_1}
\end{equation}
and
\begin{equation}
\label{4.31} \mathcal{E}^2_m (q) = \sum_{|k| \geq m}  |q(k)|^2
\leq \frac{1}{(\Omega_1 (m))^2} \sum_{|k| \geq m}  |q(k)|^2
(\Omega_1 (k))^2 \leq  \frac{\|v\|^2_{\omega_1}}{(\Omega_1
(m))^2}.
\end{equation}
By (\ref{4.28})--(\ref{4.31}),
\begin{equation}
\label{4.32} \kappa_n (v) \leq C\left
(\frac{\|v\|_{\omega_1}}{\Omega_1 (\sqrt{n})} +
\frac{\|v\|_{\omega_1}}{\sqrt{n}} \right ).
\end{equation}
Therefore, with
\begin{equation}
\label{4.33} r_N = \left ( \frac{1}{\Omega_1 (\sqrt{N})} +
\frac{1}{\sqrt{N}} \right )^{-1/4}
\end{equation}
and a proper choice of $N^*= N^* (\Omega_1,\|v\|_{\omega_1}), $ we
have for $n \geq N \geq N^*$
\begin{equation}
\label{4.34} \kappa_n (v) \leq  C\left ( \frac{1}
{\Omega_1(\sqrt{N})} + \frac{1}{\sqrt{N}} \right )^{1/2}
 \leq \frac{1}{4}\quad \text{if} \quad
\|v\|_{\omega_1}\leq 3r_N.
\end{equation}
But $\|v\|_{\omega_1} \leq \|v\|_{\Omega},$ so $\beta^\pm_n (v,z),
\, n \geq N,$ are well defined if $\|v\|_{\omega_1}\leq 3r_N$ and
$|z|\leq n/4,$  we have the inequality (\ref{3.54}) in Proposition
\ref{prop3.2}, which guarantees that
\begin{equation}
\label{4.35} \|\Phi_N (v)\|^2_{\Omega} \leq C(\Omega_1) \left (
\frac{1}{\Omega_1 (\sqrt{N})} + \frac{1}{\sqrt{N}} \right )
\|v\|^4_{\Omega}, \quad v \in B^\Omega_{3r_N} \subset
B^{\omega_1}_{3r_N},
\end{equation}
where   $C(\Omega_1) \geq  1.$

To explain (\ref{4.25}), i.e., to show that $\Phi_N$ is a
contractive mapping with a coefficient $1/2, $ we estimate its
derivative. Fix $v_1$ and $w$ such that $\|v_1\|_{\Omega}\leq r_N$
and $\|w\|_{\Omega}=1.$ The $H(\Omega)$--valued function $\varphi
(t) = \Phi_N (v_1 + tw)$ is analytic in the disc $|t| \leq 2 r_N.$
Let
\begin{equation}
\label{4.36} b_N = \left [C(\Omega_1) \left ( \frac{1}{\Omega_1
(\sqrt{N})} + \frac{1}{\sqrt{N}} \right )\right ]^{1/2}.
\end{equation}
Then, in view of (\ref{4.35}), we have
\begin{equation}
\label{4.36a}
 \|\Phi_N (v)\|_{\Omega} \leq b_N \|v\|^2_{\Omega} \leq b_N
\cdot 9r_N^2, \quad v \in B^\Omega_{3r_N}.
\end{equation}
The Cauchy inequality and (\ref{4.33}) imply
\begin{equation}
\label{4.36b}  \sup_{|t|\leq r_N} \|\frac{d}{dt} \Phi_N
(v_1+tw)\|_\Omega \leq \frac{1}{r_N}9b_N r_N^2=9 b_N r_N
\end{equation}
$$ \leq  9 \sqrt{C(\Omega_1)} \left ( \frac{1}{\Omega_1 (\sqrt{N})}
+ \frac{1}{\sqrt{N}} \right )^{1/4}  \leq \frac{1}{2}
$$
for $N \geq N^*$ if  $N^*$ is chosen large enough.

Therefore, if $v_1, v_2 \in B^\Omega_{r_N}  $ and $\|v_1 -
v_2\|_\Omega \leq r_N, $ then we obtain with $w= (v_1 -v_2)/\|v_1 -
v_2\|_\Omega$  $$ \|\Phi_N (v_1) -\Phi_N (v_2)\|_\Omega \leq
\sup_{|t|\leq r_N} \|\frac{d}{dt} \Phi_N (v_1+tw)\|_\Omega \cdot
\|v_1 -v_2\|_\Omega \leq
 \frac{1}{2}\|v_1 -v_2\|_\Omega.
$$ If $\|v_1 -v_2\|_\Omega >r_N,$ then (\ref{4.36b})
implies $$ \|\Phi_N (v_1) -\Phi_N (v_1)\|_\Omega \leq\|\Phi_N (v_1)
\|_\Omega +\|\Phi_N (v_1)\|_\Omega \leq 2 b_N r_N^2 \leq \frac{1}{2}
r_N \leq \frac{1}{2} \|v_1 -v_2\|_\Omega. $$

Of course, in view of (\ref{4.13}), (\ref{4.25}) implies
(\ref{4.26}).

Finally, a standard argument shows that (\ref{4.25})--(\ref{4.26})
imply (\ref{4.27}). Namely, for each $u \in B^\Omega_{r_N/2}$ the
mapping
\begin{equation}
\label{4.37} v \to u - \Phi_N (v)
\end{equation}
takes the ball $B^\Omega_{r_N}$ into itself because (in view of
$\Phi (0) =0$)
\begin{equation}
\label{4.38} \|u- \Phi_N (v)\|_\Omega \leq \|u\|_\Omega + \|\Phi_N
(v)- \Phi_N (0)\|_\Omega \leq \frac{1}{2}r_N+\frac{1}{2}r_N =r_N.
\end{equation}
Thus, with (\ref{4.25}), by the contraction mapping principle the
(nonlinear) operator (\ref{4.37}) has a unique fixed point $v_*
\in B^\Omega_{r_N},$ i.e., $v_*=u - \Phi_N (v_*),$ or $A_N (v_*)=
u.$ This completes the proof of Lemma \ref{lem4.2}.
\end{proof}

\begin{Remark}
\label{rem4.1} Lemma \ref{lem4.2} is formulated and proved for
spaces of complex--valued periodic potentials $v \in H^{-1}.$ In
view of Part (b) of Lemma \ref{lem2.3} and (\ref{2.36}), the
formulas (\ref{4.12}) and (\ref{4.12}) show immediately that this
lemma holds for spaces of real--valued potentials as well.
\end{Remark}

\begin{Proposition}
\label{prop33.1}
 Let  $\Omega_1 $ be a slowly increasing unbounded weight such that
\begin{equation}
\label{33.01}
 \Omega_1 (k) \leq C_1 |k|^{1/4}, \quad k\in \mathbb{Z},
\end{equation}
and let the weight $\omega_1 $  be defined by  $\omega_1 (k) =
\Omega_1 (k)/k $ for $k > 0.$ Suppose $ \Omega = (\Omega (k))_{k\in
\mathbb{Z}}$ is a weight of the form
\begin{equation}
\label{33.02}
 \Omega (k) =  \frac{\Omega_1 (k) \Omega_2 (k)}{k},
\quad k\in \mathbb{Z},
\end{equation}
where $\Omega_2 $ is a sub--multiplicative weight.

(a)  If
\begin{equation}
\label{33.1}  \frac{\log \Omega_2 (n)}{n} \searrow 0 \quad \text{as}
\quad n \to \infty,
\end{equation}
then,  for
 $v \in H(\omega_1),$
\begin{equation}
\label{33.2}       \exists N \;\;  A_N (v) \in H( \Omega) \;\;
\;\Rightarrow   \; v \in H(\Omega ).
\end{equation}
 (b)  If $\Omega_2 $ is a sub--multiplicative weight of exponential type,
then
\begin{equation}
\label{33.3} \exists N \;\;  A_N (v) \in H( \Omega)  \; \Rightarrow
\; \exists  \varepsilon >0: \; v \in H(e^{\varepsilon |n|} ).
\end{equation}
\end{Proposition}

\begin{proof} (a) If (\ref{33.1}) holds, then
(see Lemma 47 in \cite{DM15} -- this observation comes from
\cite{P04}) for each $\varepsilon >0 $ the weight
\begin{equation}
\label{33.16}  \Omega^\varepsilon_2 (m) = \min \left (
e^{\varepsilon |m|}, \Omega_2 (m) \right )
\end{equation}
is sub--multiplicative, and obviously for large enough $|m|$ we have
$\Omega^\varepsilon_2 (m) =\Omega_2 (m). $  Let $\Omega^\varepsilon$
be a weight given by
 $$
\Omega^\varepsilon (k) =\frac{\Omega_1 (k) \Omega_2^\varepsilon
(k)}{k}, \quad k \neq 0;
$$
then it follows  $H(\Omega^\varepsilon) = H (\Omega). $

Next we use the constructions and notations of Lemma \ref{lem4.2}.
If $v \in H(\omega_1), $ then $$ \|v\|_{\omega_1} < r_N/8 $$ for
large enough $ N > N^* (\Omega_1, \|v\|_{\omega_1} ). $ We choose $N
$ so that (\ref{33.2}) holds and set
 $$ w:= A_N (v) = v + \Phi_N (v). $$
Then, by (\ref{4.25})  (with $\Omega_2 \equiv 1$ in (\ref{4.24})),
we have $$ \| \Phi_N (v) \|_{\omega_1} \leq \frac{1}{2}
\|v\|_{\omega_1} \leq \frac{r_N}{16},
$$ and therefore,
\begin{equation}
\label{33.17} \| w \|_{\omega_1} = \|A_N (v)\|_{\omega_1} \leq \|v
\|_{\omega_1}  + \| \Phi_N (v) \|_{\omega_1} \leq \frac{r_N}{4}.
\end{equation}
There exists $ \varepsilon > 0 $ such that
$\|w\|_{\Omega^\varepsilon} \leq r_N/2.$ Indeed, let $ w(x) =
\sum_{k \in \mathbb{Z}} w_k \exp (2ikx); $ choose $N_1 \in
\mathbb{N} $ so that
\begin{equation}
\label{33.18} \sum_{|k|>N_1} |w_k|^2 (\Omega (k))^2 <
\frac{r_N^2}{16}.
\end{equation}
After that, choose $\varepsilon >0 $ so that $e^{\varepsilon N_1}
\leq \sqrt{2}. $ Then we have
\begin{equation}
\label{33.19} \Omega_2^{\varepsilon} (m) \leq \sqrt{2} \,\omega_1
(m) \quad \text{for} \quad |m| \leq N_1.
\end{equation}
Now (\ref{33.17})--(\ref{33.19}) imply $$ \|w\|_{
\Omega^\varepsilon}^2 \leq \sum_{|k| \leq N_1} 2|w_k|^2 (\omega_1
(m))^2 +\sum_{|k|
> N_1} |w_k|^2 (\Omega (k))^2 \leq 2\|w\|_{\omega_1}^2 + \frac{r^2_N}{16} \leq
\frac{r^2_N}{8}+\frac{r^2_N}{16},$$ and therefore, $ \|w\|_{ \Omega
^{\varepsilon}} < r_N/2.$

By (\ref{4.27}) in Lemma~\ref{lem4.2}, there exists $\tilde{v} \in
\mathcal{B}^{ \Omega^{\varepsilon}}_{r_N} \subset
\mathcal{B}^{\omega_1}_{r_N} $ such that $$ A_N (\tilde{v}) = w =
A_N (v).$$ On the other hand, by Lemma~\ref{lem4.2}, the restriction
of $A_N $ on the ball $\mathcal{B}^{\omega_1}_{r_N} $ is injective.
Thus $$v = \tilde{v} \in H(\Omega^{\varepsilon}) = H(\Omega).$$

(b) If $\Omega_2 $ is a sub--multiplicative weight of exponential
type, i.e.,
$$ \lim_{n\to \infty} \frac{\log \Omega_2 (n)}{n} > 0, $$ then,  for small
enough $ \varepsilon > 0, $  $$\Omega_2^\varepsilon (m) = \min \left
( e^{\varepsilon |m|}, \Omega_2 (m) \right )\equiv e^{\varepsilon
|m|}, \quad \Omega^{\varepsilon} (m) = \frac{\Omega_1 (m)}{m}
e^{\varepsilon |m|}  .$$

Thus, the same argument as in (a) shows that $ v \in
H(\Omega^\varepsilon) \subset H( e^{(\varepsilon/2) |m|}), $ which
completes the proof of Proposition~\ref{prop33.1}.
\end{proof}

\begin{Proposition}
\label{prop33.2} Suppose $\Omega = (\Omega (m))_{m\in \mathbb{Z}} $
is a weight of the form
\begin{equation}
\label{33.0.1} \Omega (m)= \frac{ \omega (m)}{m}, \quad m\neq 0.
\end{equation}
(a) If $\omega $ a sub--multiplicative weight such that
\begin{equation}
\label{33.0.2}  \frac{\log \omega (n)}{n} \searrow 0 \quad \text{as}
\quad n \to \infty,
\end{equation}
then
\begin{equation}
\label{33.0.3} \exists N: \; A_N (v) \in \ell^2 (\mathbb{N}, \Omega)
\Rightarrow v \in H(\Omega ).
\end{equation}
(b) If $\omega $ is a sub--multiplicative weight of exponential
type, i.e.,
\begin{equation}
\label{33.0.4} \lim_{n\to \infty} \frac{\log \omega (n)}{n} >0,
\end{equation}
 then
\begin{equation}
\label{33.0.5} \exists N: \; A_N (v) \in \ell^2 (\mathbb{N}, \Omega)
\Rightarrow   \exists \varepsilon >0:\;  v \in H(e^{\varepsilon |m|}
).
\end{equation}
\end{Proposition}

\begin{proof}
Let $w:=A_N (v), $ and let $ \sum_{m\in \mathbb{Z}} w_m e^{2imx} $
be the Fourier series of $w.$

If $ w= A_N (v) \in H(\Omega ), $ then $(|w_m|)_{m\in \mathbb{Z}}
\in \ell^2 (\Omega). $ Consider the sequence  $x= (x_n)_{n\in
\mathbb{N}} $ given by
$$ x_n = \left ( \frac{1}{n^2}|v_{-n}|^2  + \frac{1}{n^2}|v_n|^2+
( |w_{-n}|^2 +|w_n|^2) (\Omega (n))^2 \right )^{1/2}.  $$ Since $x
\in \ell^2 (\mathbb{N}),$ there exists a slowly increasing unbounded
weight $ \Omega_1 $ such that $x \in \ell^2 (\mathbb{N}, \Omega_1) $
(see Lemma 48 in \cite{DM15}). We may assume without loss of
generality that $$ \Omega_1 (n) \leq C_1 |n|^{1/4} $$ (otherwise, we
may replace $\Omega_1 $ with $(\Omega_1)^{1/a},$ where $a$ is a
suitable constant).

By the choice of $\Omega_1 $ we have $ A_N (v) \in H (\Omega_1
\Omega) $ and $v \in H(\omega_1),$ where the weight $\omega_1 $ is
given by $\omega_1 (m) =\Omega_1 (m)/m, \, m>0. $ Now, since $H
(\Omega_1  \Omega) \subset H ( \Omega), $ Proposition~\ref{prop33.2}
follows from Proposition~\ref{prop33.1}.
\end{proof}

Now we are ready to complete our analysis in the case of
real--valued potentials $v \in H^{-1}.$

\begin{Theorem}\footnote{In a new preprint (Jan 15, 2009)
V. Mikhailets and V. Molyboga \cite{MM09} prove a special case of
this statement. They assume that the potential $v$ is in $H^{-1+a}$
for any $a
> 0$ (not just in $H^{-1}$) and consider only Sobolev weights $
w_s(k) = (1 + 2|k|)^s , s  \in ( -1, \infty)$ or their slight
variation $w_s(k) = ((1 + 2|k|)^s) g(|k|),$ where $g$ is a slowly
varying function (in Karamata's sense, see \cite{Snt}).}
\label{thm33.1} Let $L = L^0 + v(x) $ be the Hill--Schr\"odinger
operator with a periodic real--valued potential $v \in H_{loc}^{-1}
(\mathbb{R}),\; v(x+\pi) = v(x),$ and let $\gamma = (\gamma_n)$ be
its gap sequence. If $\Omega = (\Omega (m))_{m\in \mathbb{Z}} $ is a
weight of the form
\begin{equation}
\label{33.001} \Omega (m)= \frac{ \omega (m)}{m}, \quad m\neq 0,
\end{equation}
where $\omega $ a sub--multiplicative weight such that
\begin{equation}
\label{33.002}  \frac{\log \omega (n)}{n} \searrow 0 \quad \text{as}
\quad n \to \infty,
\end{equation}
then
\begin{equation}
\label{33.003} \gamma \in \ell^2 (\mathbb{N}, \Omega) \Rightarrow v
\in H(\Omega ).
\end{equation}
If $\omega $ is a sub--multiplicative weight of exponential type,
i.e.,
\begin{equation}
\label{33.004} \lim_{n\to \infty} \frac{\log \omega (n)}{n} >0,
\end{equation}
 then there exists $\varepsilon >0 $ such that
\begin{equation}
\label{33.005} \gamma \in \ell^2 (\mathbb{N}, \Omega) \Rightarrow v
\in H(e^{\varepsilon |m|} ).
\end{equation}
\end{Theorem}

\begin{proof}
 In
view of Theorem \ref{thm4.1},
$$(\gamma_n)_{n>N} \in \ell^2 ( \Omega) \;
\Rightarrow  \;\exists N: \; A_N (v) \in  H (\Omega). $$ Therefore,
Theorem~\ref{thm33.1} follows from Proposition~\ref{prop33.2}.
\end{proof}

\section{Complex--valued $H^{-1}$--potentials}

In this section we extend Theorem \ref{thm33.1} -- with proper
adjustments -- to the case of complex--valued $H^{-1}$ potentials.

In \cite{DM5} we did similar "extension" of our results for real
$L^2$--potentials. We followed the general scheme of
\cite{KM1,KM2,DM3} but added two important technical ingredients.

(a) (elementary observation): A $2\times 2 $ matrix $\begin{pmatrix}
a &K \\k & a \end{pmatrix}$ has {\em two} linearly independent
eigenvectors $u_1$ and $u_2 $  if $kK\neq 0 .$ With the
normalization $kK=1$ and $|k|\leq 1 $ (otherwise the coordinates in
$\mathbb{C}^2$ could be interchanged) the angle $\alpha = \alpha
(u_1,u_2) $ between the vectors $u_1 = \begin{pmatrix} 1\\k
\end{pmatrix}$ and $u_2 =
\begin{pmatrix}  1\\-k \end{pmatrix}$ is equal to
$$\arccos \frac{1-|k|^2}{1 + |k|^2}=\arcsin \frac{2k}{1+|k|^2},$$ so
$\alpha \sim 2|k|$ if $|k| << 1.$

(b) (hard analysis)  The Riesz projections $P_n,\, P_n^0$ on $E_n$
and $E^0_n,$ respectively, are close in the following sense
\begin{equation}
\label{7.1} \|P_n-P_n^0\|_{L^2\to L^2}  \leq  \kappa_n :=
\|P_n-P_n^0\|_{L^2\to L^\infty} \to 0.
\end{equation}

{\em Remark.} In the case of $L^2$ potentials
\begin{equation}
\label{7.2} \kappa_n \leq \frac{C\|v\|}{n}, \quad n \geq n_*
\end{equation}
See Proposition 4 in \cite{DM5} or Proposition 11 in \cite{DM15}.
Close estimates may be found in \cite{VD02,DV05,SV08}.

Even for $v \in H^{-b}, \; b\in [0,1), $  the inequalities
(\ref{7.1}) can be proven with
\begin{equation}
\label{7.3} \kappa_n \leq C(b) \frac{\|v|H^{-b} \|}{n^{1-b}},
\quad n \geq n_*,
\end{equation}
i.e., these estimates are {\em uniform on the balls in} $H^{-b},
\; b\in [0,1). $

If $v \in L^1 $  then
\begin{equation}
\label{vel} \kappa_n \leq C \|v\|_{L^1} \frac{\log n}{n}.
\end{equation}
This is proven, although not explicitly claimed, in
\cite{VD02,DV05,SV08}.

In the case of Dirac operators  (see \cite{Mit04}, Proposition~8.1
and Corollary~8.6, and  \cite{DM15}, Proposition~19, formula
(1.165)), estimates like (\ref{7.1}) do not hold on the balls in
$L^2$--space of potentials. But for individual potentials or on
compacts in $L^2$ an analogue of (\ref{7.1}) holds with
\begin{equation}
\label{7.4} \kappa_n \leq C\|v\| \left (\frac{\|v\|}{\sqrt{n}} +
\mathcal{E}_{|n|/2} (w) \right ),   \quad n \geq n_*,
\end{equation}
where $w$ is an $\ell^2$--sequence.

In Appendix, we give estimates of $\|P-P_n^0\|_{L^2\to L^\infty}$ or
even of $\|P-P_n^0\|_{L^1\to L^\infty}$ in the case of
Hill--Schr\"odinger operators with complex--valued $\pi$--periodic
$H^{-1}$--potentials, subject to $Per^\pm $ or Dirichlet boundary
conditions.  See Proposition~\ref{prop91}, Theorem~\ref{thm91} and
the inequality (\ref{18d}) in Section 9, Appendix.

These facts make possible to preserve the basic structure of the
proof in the case of $L^2$ potentials:  we just need to use
Proposition~\ref{prop91}, or (\ref{18d}). Keeping this in mind we
omit details of the proofs (see \cite{DM15}, Section 4) but
reproduce the steps and the core statements leading to the proof of
the main result.

\begin{Theorem}
\label{thm44.2}  Let $L = L^0 + v(x) $ be the Hill--Schr\"odinger
operator with a $\pi$--periodic potential $v \in H^{-1}_{loc}
(\mathbb{R}).$ Then, for $n > N(v),$  the operator $L$ has in the
disc of center $n^2 $ and radius $r_n = n/4 $ one Dirichlet
eigenvalue $\mu_n $ and
 two (counted with their algebraic multiplicity) periodic (for even $n$),
or antiperiodic (for odd $n$)   eigenvalues $\lambda^+_n $ and $
\lambda^-_n. $

Let
\begin{equation}
\label{44.3} \Delta_n = |\lambda^+_n - \lambda^-_n | + |\lambda^+_n
- \mu_n|, \quad n > N (v);
\end{equation}
then, for each weight $\Omega = (\Omega (m))_{m \in \mathbb{Z}} $ of
the form
$$   \Omega (m)= \omega (m)/m,  \quad m \neq 0,  $$
where $\omega $ is a sub-multiplicative weight,  we have
\begin{equation}
\label{44.4} v \in H(\Omega )  \; \Rightarrow \;  (\Delta_n)  \in
\ell^2 (\Omega ).
\end{equation}

Conversely,  if  $ \; \omega = (\omega (m))_{m\in \mathbb{Z}} $ is a
sub--multiplicative weight such that
\begin{equation}
\label{44.5}  \frac{\log \omega (n)}{n} \searrow 0 \quad \mbox{as}
\quad n \to \infty,
\end{equation}
then
\begin{equation}
\label{44.6}  (\Delta_n) \in \ell^2 (\Omega) \;  \Rightarrow  \; v
\in H(\Omega ).
\end{equation}

If $\omega $ is a sub--multiplicative weight of exponential type,
i.e.,
\begin{equation}
\label{44.7} \lim_{n\to \infty}  \frac{\log \omega (n)}{n} >0
\end{equation}
then
\begin{equation}
\label{44.8} (\Delta_n) \in \ell^2 (\Omega) \; \Rightarrow   \;
\exists  \varepsilon >0: \; v \in H(e^{\varepsilon |m|} ).
\end{equation}
\end{Theorem}

\begin{proof}
By Proposition \ref{prop004}, for $n > N(v),$  the operator $L$ has
in the disc of center $n^2 $ and radius $r_n = n/4$
 one Dirichlet eigenvalue
$\mu_n $ and two (counted with their algebraic multiplicity) periodic
(for even $n$), or antiperiodic (for odd $n$) eigenvalues
$\lambda^+_n $ and $ \lambda^-_n. $

Let $E =E_n $ and $E^0 =E_n^0 $ denote the corresponding
2--dimensional invariant subspace  of $L$ and the free operator $L^0$
(subject to periodic or antiperiodic bc), and let  $P= P_n $ and $P^0
= P^0_n $ denote the Cauchy--Riesz projections on $E_n $ and $E_n^0,
$ respectively. In what follows, we fix an $n \in \mathbb{N}$  and
consider the corresponding objects like $E=E_n, P= P_n $ etc,
suppressing  $n$ in the notations.  The subspace $E^0 = E^0_n $ has
the following standard basis of eigenvectors of $L^0 $ (corresponding
to the eigenvalue $n^2$):
\begin{equation}
\label{41.2} e^1 (x)= e^{-inx}, \quad e^2 (x)= e^{inx}, \quad n \in
\mathbb{N}.
\end{equation}

If the restriction of $L$ on $E$ has two distinct eigenvalues, we
denote them by $ \lambda^+ $ and $ \lambda^-,$ where  $\lambda^+ $
is the eigenvalue which real part is larger, or which imaginary part
is larger if the real parts are equal, and set $ \gamma = \lambda^+
- \lambda^-.$

Step 1 (analogue of Lemma 59 in \cite{DM15}).

\begin{Lemma}
\label{lem41.1} In the above notations, for large enough $n,$ there
exists a pair of vectors $f, \varphi \in E = E_n $ such that
\begin{enumerate}
\item[(a)]  $ \|f\|=1, \; \|\varphi \|= 1, \; \langle f, \varphi \rangle = 0; $
\item[(b)]  $Lf = \lambda^+ f;$
\item[(c)]  $ L \varphi = \lambda^+ \varphi - \gamma \varphi + \xi f. $
\end{enumerate}
Moreover,  with $\varphi^0 = P^0 \varphi,$ we have
\begin{equation}
\label{41.01} |\xi| \leq     4 |\gamma| +  2 \| \left (z^+ - S
(\lambda^+) \right )\varphi^0 \|
\end{equation}
and
\begin{equation}
\label{41.02} \| \left (z^+ - S (\lambda^+) \right )\varphi^0 \|
\leq  2( |\xi| + |\gamma |),
\end{equation}
where $ z^+ = \lambda^+ - n^2 $ and $ S(\lambda^+): E^0 \to E^0 $ is
the operator (\ref{2.8}) constructed in Lemma \ref{lem2.2}.
\end{Lemma}

Proof is given in \cite{DM15}, Lemma 59.  However, there in all
inequalities on pp. 735--736 after (4.9) till the lines 9--10, p. 736
we need to use (\ref{18d}) to guarantee that $\kappa_n \to 0.$
\vspace{3mm}

Step 2. In what follows we use Lemma \ref{lem41.1} and its
notations. Let $f, \varphi \in E$ be the orthonormal pair of
vectors constructed in Lemma~\ref{lem41.1}, and let $f^0_1, f^0_2
$ and $\varphi^0_1,\varphi^0_1$ be the coordinates of $f^0 =P^0 f
$ and $\varphi^0 = P^0 \varphi $ with respect to the basis $\{e^1,
e^2 \},$ i.e.,
\begin{equation}
\label{41.5} f^0 (x) = f^0_1  e^1 (x) +   f^0_2 e^2 (x), \quad
\varphi^0 (x) = \varphi^0_1  e^1 (x) +   \varphi^0_2 e^2 (x).
\end{equation}

Then $Lf = \lambda^+ f $ and, by Lemma \ref{lem2.2}, the vector $f^0
= P^0 f $ is an eigenvector of the operator $  L^0 + S(\lambda^+) :
E^0 \to E^0 $ with eigenvalue $\lambda^+.$ This leads to
\begin{equation}
\label{41.4}
\begin{pmatrix} \zeta^+ & B^- \\ B^+ & \zeta^+
\end{pmatrix}
\begin{pmatrix} f^0_1 \\ f^0_2 \end{pmatrix} =
\begin{pmatrix} \zeta^+ f^0_1 + B^- f^0_2\\ B^+ f^0_1 + \zeta^+ f^0_2
\end{pmatrix} =
\begin{pmatrix} 0 \\ 0 \end{pmatrix}, \quad (\zeta^+)^2 = B^+ B^-,
\end{equation}
where $ \begin{pmatrix} \zeta^+ & B^- \\ B^+ & \zeta^+
\end{pmatrix} $ is the matrix representation of the operator $z^+
- S(\lambda^+), $ so, in view of (\ref{2.28}) and (\ref{2.40}),
\begin{equation}
\label{41.4a} \zeta^+ = z^+ - \alpha_n (z^+), \quad B^\pm =
\beta^\pm_n (z^+), \quad z^+ = \lambda^+ - n^2.
\end{equation}

\begin{Lemma}
\label{lem41.2} In the above notations, for large enough $n,$
\begin{equation}
\label{41.31} \frac{1}{2} \left (|B^+|+|B^-| \right ) \leq \|
\left (z^+ - S (\lambda^+) \right )\varphi^0 \| \leq \left
(|B^+|+|B^-| \right )
\end{equation}
\end{Lemma}

Proof is really given in \cite{DM15}, Lemma 60. However, there in
inequalities (4.21), (4.22) and between we need to use (\ref{18d})
to guarantee that $\kappa \equiv \kappa_n \leq 1/2$ for $n \geq
n(v).$\vspace{3mm}

Step 3. Upper bounds for deviations $ |\mu - \lambda^+|. $

Now we construct a Dirichlet function $G \in E = E_n $ and use it to
estimate $|\mu - \lambda^+|$ in terms of $|B^-|$ and $|B^+|.$

Let $g$ be a unit Dirichlet eigenvector that corresponds to $\mu,
$ i.e.,
\begin{equation}
\label{42.1} L_{dir} g = \mu g, \qquad \|g\| =1,
\end{equation}
and let $P_{dir} $ be the Cauchy--Riesz projection on the
corresponding one--dimensional eigenspace.

\begin{Lemma}
\label{lem42.1}  Under the assumptions of Lemma \ref{lem41.1},

(a) there is a vector $G \in E=E_n $  of the form
\begin{equation}
\label{42.2} G = af + b \varphi, \quad  \|G\|^2 = |a|^2 + |b|^2 =
1,
\end{equation}
such that
\begin{equation}
\label{42.3} G(0) = 0, \quad  G (\pi) = 0,
\end{equation}
 i.e., $G $ is in the domain of $L_{dir}.$

 (b) Moreover, we have
 \begin{equation}
 \label{42.4}
 \tau (\mu - \lambda^+) = b \xi \langle P_{dir}f,g \rangle -
 b \gamma \langle P_{dir}\varphi, g \rangle,
 \end{equation}
 where $\tau = \tau_n, $
 \begin{equation}
 \label{42.5}
     1/2 \leq  |\tau| \leq 2.
 \end{equation}
 \end{Lemma}

Proof repeats the proof of Lemma 61 (its Hill--Schr\"odinger part)
in \cite{DM15} but on pp. 739--740 in inequalities (4.33--39) and
(4.40--42) again we use (\ref{18d}) to guarantee that $\kappa_n
\to 0.$

Steps 1--3 lead us to the following

\begin{Proposition}
\label{prop42.1} Under the above assumptions and notations, for large
enough $n,$ we have
\begin{equation}
\label{42.20}   |\mu - \lambda^+|  \leq   18 |\gamma| + 8\left (
|B^+| +|B^-| \right )
\end{equation}
\end{Proposition}

Proof is given in \cite{DM15}, Proposition 62, where it is explained
how the statements of steps 1--3 imply the inequality
(\ref{42.20}).\vspace{3mm}

Step 4.
\begin{Proposition}
\label{prop43.1} For large enough $n, $ if
\begin{equation}
\label{43.01} \frac{1}{4} |B^- | \leq |B^+| \leq 4|B^-|,
\end{equation}
then
\begin{equation}
\label{43.02}  |\beta^-_n (z_n^*)| +|\beta^+_n (z_n^*)|  \leq 2
|\gamma_n | ,
\end{equation}
where $ B^\pm = \beta^\pm_n (z^+) $ and  $ z^*_n = \frac{\lambda^+_n
+ \lambda_n^-}{2} - n^2 $ in the case of simple eigenvalues, $z^*_n=
\lambda^+_n - n^2 $ otherwise.
\end{Proposition}
\begin{proof} If (\ref{43.01}) holds, then either $B^+ = B^- = 0,$ or $B^+ B^-
\neq 0. $ If we have $B^+ = 0 $ and $B^- = 0,$ then $\lambda^+ $
is an eigenvalue of geometric multiplicity 2 of the operator $P^0
L^0 P^0 + S(\lambda^+). $ That may happen, by Remark
\ref{rem21.1}, if and only if $\lambda^+ $ is an eigenvalue of
geometric multiplicity 2 of the operator $L$ also. But then
$\gamma_n =0, \;z^*_n =z^+_n, $ and therefore,  (\ref{43.02})
holds.

If $B^+ B^- \neq 0 $ and $\gamma_n \neq 0, $ then the claim
follows from Lemma \ref{lem32.1} because (\ref{43.01}) implies $t
= |B^+|/|B^-| \in [1/4, 4],$ and therefore, $2\sqrt{t}/(1+t) \geq
4/5.$

Finally, let us consider the case, where $B^+B^- \neq 0 $ but $
\gamma_n = 0. $ By Lemma \ref{lem2.2}, if  $z\in  D = \{w: \,|w|<n/4
\},$ then the point $ \lambda = n^2 + z $ is an eigenvalue of $L$ if
and only if $z$ is a root of the basic equation
$$ h(z) :=(\zeta (z))^2 - \beta^+ (z) \beta^- (z)= 0, \quad \zeta
(z) = z- \alpha (z). $$ Therefore, if $ \gamma_n = 0, $ then $ z^+ =
\lambda^+ - n^2 $ is the only root of the equation $h(z) = 0$ on the
disc $D.$

Moreover, the root $z^+$ is of multiplicity 2. Indeed, consider the
two equations $ z^2 = 0 $ and $ h(z) = 0$ on the disk $D.$ In view
of Proposition~\ref{prop3.1}, the maximum values of $|\alpha_n (z)|$
and $|\beta^\pm_n (z)| $ on the circle $\partial D =\{z: |z|= n/4\}$
do not exceed $ n\varepsilon_n, $  where  $ \varepsilon_n \to 0  $
as $n \to \infty. $ Since $$ h(z) - z^2 = -2\alpha z +\alpha^2 -
\beta^+ \beta^-, $$ we have $$ \sup_{\partial D} |h(z) - z^2| \leq
\sup_{\partial D}(2|z||\alpha| +|\alpha^2| +|\beta^+ \beta^-|) \leq
\left ( \frac{n}{4} \right )^2 \eta_n, $$ where $\eta_n = 8
\varepsilon_n (1+4\varepsilon_n) \to 0. $

Therefore, for large enough $n,$ we have $$ \sup_{\partial D} |h(z)
-z^2| < \sup_{\partial D} |z^2|, $$ so the Rouche Theorem implies
that $z^+ $ is a double root of the equation $h(z) = 0. $ Thus, the
derivative of $h$ vanishes at $z^+$, i.e., $$ 2 \zeta (z^+) \cdot
\left (1- \frac{d\alpha}{dz} (z^+) \right) = \frac{d \beta^+}{dz}
(z^+) \cdot \beta^- (z^+) + \beta^+ (z^+) \cdot \frac{d \beta^-}{dz}
(z^+) .$$

By Proposition \ref{prop3.1}, formula (\ref{3.32}), we have, for
large enough $n,$ $$ \left | \frac{d\alpha}{dz} (z^+) \right | \leq
\frac{1}{5} , \qquad \left | \frac{d\beta^\pm}{dz} (z^+) \right |
\leq \frac{1}{5}. $$ Therefore, by the triangle inequality, we get
$$ 2|\zeta^+| (1-1/5 ) \leq  \frac{1}{5} (|B^+| +|B^-|),$$ where $
|\zeta^+| = |\zeta (z^+)| = \sqrt{|B^+B^-|}.$ Thus, the latter
 inequality implies, in view of (\ref{43.01}),
 $$  8 \leq \sqrt{|B^+/B^-|} + \sqrt{|B^-/B^+|} \leq  4, $$
 which is impossible. This completes the proof.
 \end{proof}

Step 5. Next we consider the case that is complementary to
(\ref{43.01}), i.e.,
\begin{equation}
\label{43.1} (a) \quad  4|B^+| < |B^-| \qquad   \text{or} \;\; (b)
\quad 4|B^-| < |B^+|.
\end{equation}
We begin with the following technical statement.

\begin{Lemma}
\label{lem43.1} If $n $ is large enough and (\ref{43.1}) holds, then
\begin{equation}
\label{43.2} \frac{1}{4} \leq \frac{|f(0)|}{|\varphi (0)|} \leq 4.
\end{equation}
\end{Lemma}

Proof is essentially given in \cite{DM15}, pp. 743--744, but again
to justify the analogs of the inequalities (4.49) to (4.56) and
between  in \cite{DM15}, we use Proposition~\ref{prop91} (or
Theorem~\ref{thm91}) to claim that $\kappa_n \to 0.$ \vspace{3mm}

Step 6.

\begin{Proposition}
 \label{prop43.2}
 If (\ref{43.1}) holds,
 then we have, for large enough $n,$
  \begin{equation}
  \label{43.26}
  |B^+ |+ |B^-| \leq   36 |\gamma | + 144  |\mu - \lambda^+|.
  \end{equation}
 \end{Proposition}

 Proof with the same disclaimer as above is given in
\cite{DM15}, pp. 744--745. \vspace{3mm}

Step 7.

\begin{Theorem}
\label{thm44.1} Let $L = L^0 + v(x) $ be the Hill--Schr\"odinger
operator with a $\pi$--periodic potential $v \in H^{-1}_{loc}
(\mathbb{R}).$  For large enough $n,$ if $ \lambda^+_n ,
\lambda^-_n $ is the $n$-th couple of periodic (for even $n$) or
antiperiodic (for odd $n$) eigenvalues of $L,$ $\gamma_n =
\lambda^+_n - \lambda^-_n ,$ and $\mu_n $ is the $n$-th Dirichlet
eigenvalue of $L,$ then
\begin{equation}
\label{44.1}
 \frac{1}{72} \left ( |\beta^-_n (z_n^*)| +|\beta^+_n
(z_n^*)| \right ) \leq |\gamma_n | +|\mu_n - \lambda^+_n |  \leq 58
\left ( |\beta^-_n (z_n^*)| +|\beta^+_n (z_n^*)| \right ).
\end{equation}
\end{Theorem}

\begin{proof}  By Proposition \ref{prop004} (localization of spectra)
and Proposition~\ref{prop3.1}, $$ \sup_{[z_n^-, z_n^+]} |\partial
\alpha_n/\partial z | \leq \varepsilon_n, \quad \sup_{[z_n^-,
z_n^+]} |\partial \beta^\pm_n/\partial z | \leq \varepsilon_n, \quad
\varepsilon_n \downarrow 0, $$ where $[z_n^-, z_n^+] $ denotes the
segment with end points $z_n^- $ and $ z_n^+. $ Therefore, since $
|z^+_n - z^*_n| \leq |\gamma_n | =|z^+_n - z^-_n| ,$ we have $$
|\beta_n^\pm (z^+_n) -\beta_n^\pm (z^*_n)| \leq \varepsilon_n \cdot
|\gamma_n|. $$ By the triangle inequality, it follows, for large
enough $n,$ that
\begin{equation}
\label{44.2}
  |B^\pm| - \frac{1}{2}|\gamma_n| \leq  |\beta^\pm_n (z_n^*)|
  \leq  |B^\pm| + \frac{1}{2} |\gamma_n|,
\end{equation}
where $ B^\pm = \beta^\pm_n (z_n^+).$

In view of (\ref{44.2}), Propositions~\ref{prop43.1} and
\ref{prop43.2} imply the left inequality in (\ref{44.1}). On the
other hand, Lemma~\ref{lem4.1} yields that, for large enough $n,$
$$   |\gamma_n | \leq 2   \left ( |\beta^-_n (z_n^*)| +|\beta^+_n
(z_n^*)| \right ).$$ Therefore, Proposition~\ref{prop42.1} and the
inequality (\ref{44.2}) imply the right inequality in (\ref{44.1}).
Theorem~\ref{thm44.1} is proved.
\end{proof} \vspace{2mm}

Step 8. Let $N $ be so large that Theorem \ref{thm44.1} holds for
$n\geq N,$ and let  $\Omega = (\Omega (m))_{m \in \mathbb{Z}} $ be a
weight of the form
$$   \Omega (m)= \omega (m)/m,  \quad m \neq 0,  $$
where $\omega $ is a sub-multiplicative weight. By the right
inequality in (\ref{44.1}), Theorem \ref{thm44.1}, we have, for
$n>N,$ $$ \Delta_n =|\gamma_n | +|\mu_n - \lambda^+_n |\leq 58 \left
( |\beta^-_n (z_n^*)| +|\beta^+_n (z_n^*)| \right )$$
$$ \leq
58\left (|v_{-n}| +|v_n|+ |\beta^-_n (z_n^*) - v_{-n} | + |\beta^+_n
(z_n^*)- v_n| \right ).
$$

Therefore, in view of (\ref{3.54}) in Proposition \ref{prop3.2}, or
more precisely as in its specification (\ref{5.21}), we obtain $$
\sum_{n> N} \Delta_n^2 (\Omega (n))^2 \leq C \sum_{n> N} (|v_{-n}|^2
+ |v_n|^2)(\Omega (n))^2  + C (\|v\|^4_{\Omega} +\|v\|^2_{\Omega}) <
\infty, $$  which proves (\ref{44.4}).

Conversely, suppose $(\Delta_n)_{n\geq N} \in \ell^2 (\Omega) $ with
$ \Omega (m)= \omega (m)/m,  \quad m \neq 0,  $ where $\omega $ is a
sub--multiplicative weight $\omega $ having the property
(\ref{44.5}) (or, respectively, (\ref{44.7})). Then, by the left
inequality in (\ref{44.1}), Theorem \ref{thm44.1}, we have $$
(\Delta_n)_{n\geq N} \in \ell^2 (\Omega) \; \Rightarrow \; \left (
|\beta^-_n (z_n^*)| +|\beta^+_n (z_n^*)| \right )_{n\geq N} \in
\ell^2 (\Omega). $$ This yields, in view of the definition of the
mapping $A_N $ (see (\ref{4.13})--(\ref{4.13a})), that $ A_N (v) \in
H(\Omega ) $ because, for $n> N,$ the numbers $\beta^\pm_n (z_n^*) $
are, respectively, the $\pm n$-th Fourier coefficients of $A_N (v).
$ Now, by Proposition \ref{prop33.2}, we get $v \in H(\Omega ) $
(or, respectively, $ v \in H(e^{\varepsilon |m|} ) $), which
completes the proof of Theorem~\ref{thm44.2}.
\end{proof}

\section{Comments}

1. In his preprint \cite{P04} J. P\"oschel presented results of
\cite{KM1,KM2,DM3,DM4,DM5} and made attempts to improve the
technical exposition and to ease some assumptions, for example on
weight sequences $\Omega.$ But its starting point (at least in the
case of complex--valued $L^2$--potentials) is the family of ``{\em
alternate gap lengths}'' which would mimic the properties of
Dirichlet eigenvalues. He mentions that Sansuc and Tkachenko
(presumably in \cite{ST}) considered the quantities $\delta_n =
\mu_n - \tau_n,$ where $\mu_n $ are the Dirichlet eigenvalues and
$ \tau_n = (\lambda_n^+ + \lambda_n^-)/2 $ are the midpoints of
the spectral gaps,
 and then claims:
{\em ''More generally, one may consider a family of continuously
differentiable alternate gap lengths $\delta_n : \mathcal{H}^0 \to
\mathbb{C},$ characterized by the properties that}

-- $\delta_n$ {\em vanishes whenever $\lambda_n^+ = \lambda_n^-$
has also geometric multiplicity 2, and}

-- {\em there are real numbers $\xi_n $ such that its gradient
satisfy $$  d\delta_n = t_n + O(1/n), \quad t_n = \cos 2\pi n
(x+\xi_n), $$ uniformly on bounded subsets of $\mathcal{H}^0.$
That is, $$ \| d_q \delta_n -t_n \|_0 \leq C_q (\|q\|_0 )/n $$
with $C_\delta $ depending only on} $\|q\|_0 :=
\|q\|_{\mathcal{H}^0}``.$\\ (In \cite{P04}, line 10 on page 3,
$\mathcal{H}^0$ is defined as the $L^2$--space of complex--valued
functions on $[0,\pi].$)

But such {\em entire} functions $\delta_n $ do not exist. (It
means that many constructions of \cite{P04} manipulate with an
empty set.) This was well understood by J. Meixner and F. W.
Sch\"afke in the early 1950's. They explained \cite{MS54,MSW80}
that the $n$--th Dirichlet eigenvalue $E_n (z)$ of the Mathieu
operator $$ L(z) y= -y^{\prime \prime} + z \cos 2x \, y, \quad
y(0)=y(\pi) =0, \;\; E_n (0)=n^2, $$ has a finite radius of
analyticity.

This phenomenon is very important in understanding and
construction of analytic functions used in the papers
\cite{KM1,KM2,DM3,DM4,DM5}

Now, for completeness of our analysis and presentation, we give a
proof of the following statement which is a generalization of the
Meixner-Sch\"afke result (\cite{MS54}, Thm 8, Section 1.5).

\begin{Proposition}
\label{prop22} Let
\begin{equation}
\label{s0} v(x) = \sum_{k=1}^\infty v_k \sqrt{2} \cos 2kx,
\end{equation}
where $(v_k) $ is a real sequence such that
\begin{equation}
\label{s1} \sum_k |v_k|= \sigma  < \infty
\end{equation}
and
\begin{equation}
\label{s2}   k v_k \to 0 \quad \text{as} \;\; k \to \infty.
\end{equation}
Then the $n$--th Dirichlet eigenvalue $E_n (z)$ of the operator $$
L(z) y= -y^{\prime \prime} + z v(x)y, \qquad y (0)=y(\pi) =0, \;\;
E_n (0)=n^2 $$ is analytic in a neighborhood of $z=0,$ and the
radius of convergence $R_n $ of its Taylor series about $z=0$
satisfies, for large enough $n,$
\begin{equation}
\label{s3}  R_n \leq C n^2,\quad C=
\frac{32\sqrt{2}\sigma}{\|v\|^2}.
\end{equation}
\end{Proposition}

\begin{proof}
It is well known (see \cite{Kato}, Sections 7.2 and 7.3) that the
function $E_n (z) $ is analytic in a neighborhood of $0.$

Let
\begin{equation}
\label{s5} E_n (z) = n^2 + \sum_1^\infty a_k (z) z^k
\end{equation}
be its Taylor series expansion about $z=0,$ and let $R_n $ be the
radius of convergence of (\ref{s5}).

Proof by the Meixner--Sch\"afke scheme has two independent (to some
extent) parts:

(A) the estimates from above of the Taylor  coefficient $a_2 (n)$
of $E_n (z), $ or, more generally, $a_k (n), \; k\geq 3, $ in
terms of the radius $R_n;$

(B) the estimates from below of $a_2 (n), $ or, more generally,
$a_k (n), \; k\geq 3. $

{\em Part A.} Under the condition (\ref{s1}) the potential $v$ is
a bounded (continuous) function, so the multiplier operator $V: f
\to vf$ is bounded in $L^2 ([0,\pi]) $ and
\begin{equation}
\label{s6} \|V\| \leq \sigma.
\end{equation}

Let us assume that $E_n (z) $ is well defined analytic function of
$z$ in a disc $D_r = \{z: \; |z| \leq r \}.$ Then we have, with $
f \in Dom (L)$ depending on $z,$
\begin{equation}
\label{s7} (L^0 +zV) f = E_n (z) f, \quad \|f\|=1,
\end{equation}
and therefore,
\begin{equation}
\label{s8} E_n (z) = \langle L^0 f,f \rangle + z \langle V f,f
\rangle.
\end{equation}
Since  the operator $L^0  $  is self--adjoint, we have
\begin{equation}
\label{s9} Im \,E_n (z)  = Im \, ( z \langle V f,f \rangle).
\end{equation}
From (\ref{s6}) and (\ref{s9}) it follows that
\begin{equation}
\label{s10} \left |Im \,(E_n (z)-n^2) \right |=\left |Im \,E_n (z)
\right |  \leq \sigma r.
\end{equation}
Since $E_n(0) =n^2, $ in view of  (\ref{s10}) we have
$$\frac{1}{2\pi}\int_0^{2\pi} \left | Re \, (E_n (re^{it})-n^2)
\right |^2 dt= \frac{1}{2\pi}\int_0^{2\pi} \left | Im \, (E_n
(re^{it})-n^2) \right |^2 dt \leq  \sigma^2 r^2,$$ which implies,
by the Cauchy inequality,
\begin{equation}
\label{s11} \frac{1}{2\pi}\int_0^{2\pi} \left | E_n (re^{it})-n^2
\right | dt \leq  \left (\frac{1}{2\pi}\int_0^{2\pi} \left | E_n
(re^{it})-n^2 \right |^2 dt \right )^{1/2}  \leq \sqrt{2} \sigma
r.
\end{equation}
Therefore, the Cauchy formula yields $$ |a_2 (n) | = \left |
\frac{1}{ 2 \pi i} \int_{\partial D_r} \frac{E_n (z)-n^2}{z^3} dz
\right | \leq \frac{1}{ 2 \pi r^2} \int_0^{2\pi} \left |  (E_n
(re^{it})-n^2) \right | dt \leq \frac{\sqrt{2}\sigma}{r}. $$ Since
this inequality holds for every $ r< R_n, $ we get
\begin{equation}
\label{s12} |a_2 (n) |\leq \frac{\sqrt{2}\sigma}{R_n},
\end{equation}
so whenever $a_2 (n) \neq 0 $ it implies that $R_n $ is finite and
\begin{equation}
\label{s15} R_n \leq \frac{\sqrt{2}\sigma}{|a_2 (n)|}.
\end{equation}

{\em Remark.} Of course, an analogue of (\ref{s15}) could be
derived for any $k>2.$  The Cauchy formula and (\ref{s11}) imply,
for each $r< R_n,$  $$|a_k (n) | = \left | \frac{1}{ 2 \pi i}
\int_{\partial D_r} \frac{E_n (z)-n^2}{z^{k+1}} dz \right | \leq
\frac{1}{ 2 \pi r^k} \int_0^{2\pi} \left |  (E_n (re^{it})-n^2)
\right | dt \leq \frac{\sqrt{2}\sigma}{r^{k-1}}, $$ so
\begin{equation}
\label{s16} R_n \leq \left (\frac{\sqrt{2}\sigma}{|a_k (n)|}
\right )^{1/(k-1)} \qquad \text{if} \quad a_k (n) \neq 0.
\end{equation}

{\em Part B.}   To make the inequality (\ref{s15}) meaningful as a
tool to evaluate $R_n$ we need to estimate $|a_2 (n)|$ from below.
Let us follow the Raleigh--Schr\"odinger procedure.

It is well known (e.g., see \cite{Kato}, Section 7.3, in
particular (7.3.3)) that there exists an analytic family of
eigenvectors
\begin{equation}
\label{s17} \varphi (z) = \varphi_0 + z \varphi_1  + z^2 \varphi_2
+ \cdots, \quad |z| < \rho<<1,
\end{equation}
with
\begin{equation}
\label{s17a}
 \varphi_0 = \sqrt{2} \sin nx.
 \end{equation}
Therefore, by (\ref{s5}),
\begin{equation}
\label{s18}
 (L^0 + zV) ( \varphi_0 + z \varphi_1  + z^2 \varphi_2 + \cdots)=
 (a_0 + a_1 z + a_2 z^2 + \cdots)( \varphi_0 + z \varphi_1  + z^2 \varphi_2 + \cdots),
 \end{equation}
so we have
\begin{equation}
\label{s18a}
 L^0 \varphi_0 = a_0 \varphi_0, \quad a_0 = n^2,
 \end{equation}
\begin{equation}
\label{s18b}
 L^0 \varphi_1 +V \varphi_0 = a_0 \varphi_1 + a_1  \varphi_0,
 \end{equation}
\begin{equation}
\label{s18c}
 L^0 \varphi_2 +V \varphi_1 = a_0 \varphi_2 + a_1  \varphi_1+
 a_2  \varphi_0.
 \end{equation}
Let us notice that if $ g \in \;  Dom (L^0) = Dom (L^0 +zV), $
then we have
\begin{equation}
\label{s19}
 \langle (L^0 -a_0)g, \varphi_0 \rangle = \langle g, (L^0 -a_0)\varphi_0 \rangle =0.
 \end{equation}
Therefore,  taking the scalar product of both sides of
(\ref{s18b}) with $\varphi_0 ,$ we get
\begin{equation}
\label{s20} a_1 = \langle V \varphi_0,\varphi_0 \rangle =
\frac{1}{\pi} \int_0^\pi v(x) 2 \sin^2 nx \, dx = \frac{1}{\pi}
\int_0^\pi v(x) (1-\cos 2nx)dx = -\frac{v_n}{\sqrt{2}}.
 \end{equation}
We rewrite (\ref{s18b}) as
\begin{equation}
\label{s21} (a_0 - L^0) \varphi_1  = (V-a_1) \varphi_0;
 \end{equation}
this implies
\begin{equation}
\label{s22} \varphi_1 =b_1 \varphi_0 + \hat{R} (a_0) (V-a_1)
\varphi_0,
 \end{equation}
where $b_1 $ is an unknown constant and
\begin{equation}
\label{s23} \hat{R} (a_0) (\sin kx ) = \begin{cases} 0,  &  k=n,\\
\frac{1}{n^2 - k^2} \sin kx, & k \neq n.
\end{cases}
 \end{equation}
The next step will give us the value of $a_2 (n).$ In view of
(\ref{s19}), a multiplication of both sides of (\ref{s18c}) by
$\varphi_0 $ leads to
\begin{equation}
\label{s24} a_2 (n) = \langle (V-a_1)\varphi_1, \varphi_0 \rangle.
 \end{equation}

The first term $b_1 \varphi_0 $ in (\ref{s22}) is not known but
(\ref{s20}) yields $$  \langle (V-a_1)b_1\varphi_0, \varphi_0
\rangle= b_1\langle (V-a_1)\varphi_0, \varphi_0 \rangle =0. $$
Therefore,
\begin{equation}
\label{s26} a_2=  \langle (V-a_1)\hat{R} (a_0) (V-a_1)
\varphi_0,\varphi_0 \rangle= \sum_{k\neq n} \frac{1}{n^2-k^2}
(g_k)^2,
 \end{equation}
where $g_k $ are the Fourier coefficients of the function
$g=(V-a_1)\varphi_0,$ i.e., $$ g_k = \frac{1}{\pi} \int_0^\pi g(x)
\sqrt{2} \sin kx dx, \quad k\in \mathbb{N}. $$

By (\ref{s0}),
\begin{equation}
\label{s28} V\varphi_0 = \sum_k v_k \sqrt{2}\cos 2kx \sqrt{2} \sin
nx=\sum_k v_k \left ( \sin(n+2k)x + \sin(n-2k)x \right ).
\end{equation}
Therefore, by (\ref{s20}), if $n $ is even, $n=2m,$ we have
\begin{equation}
\label{s29} (V-a_1)\varphi_0 = \frac{1}{\sqrt{2}}\sum_{p\in
\mathbb{N}\setminus \{m\}} (v_{|p-m|} - v_{p+m})\sqrt{2}\sin 2px,
\end{equation}
and if $n $ is odd, $n=2m-1,$
\begin{equation}
\label{s30} (V-a_1)\varphi_0 = \frac{1}{\sqrt{2}}\sum_{p\in
\mathbb{N}\setminus \{m\}} (v_{|p-m|} - v_{p+m-1})\sqrt{2}\sin
(2p-1)x.
\end{equation}
Now  (\ref{s26}) implies
\begin{equation}
\label{s31} a_2 (n)= \frac{1}{8}\sum_{p\in \mathbb{N}\setminus
\{m\}} \frac{1}{m^2 -p^2}(v_{|p-m|} - v_{p+m})^2, \quad n=2m,
\end{equation}
and
\begin{equation}
\label{s32} a_2 (n)= \frac{1}{8}\sum_{p\in \mathbb{N}\setminus
\{m\}} \frac{1}{(m-p)(m+p-1)}(v_{|p-m|} - v_{p+m-1})^2, \quad
n=2m-1.
\end{equation}
Of course, these formulas could be useful for different purposes,
so let us state the following.
\begin{Proposition}
\label{prop28} Assume that the potential   $v(x) = \sum_k v_k \cos
2kx $ is continuous on $[0,\pi].$ Then the operator $L= -d^2/dx^2
+ zV $ subject to Dirichlet boundary conditions has, for $|z| <
\rho, $ simple eigenvalues $E_n (z) $ which are analytic
functions, where $\rho>0$ does not depend of $n.$ Moreover
\begin{equation}
\label{s33} E_n (0) = n^2, \quad E^\prime (0) = - v_n/\sqrt{2},
\quad E^{\prime \prime} (0)= a_2 (n)/2,
\end{equation}
where $a_2 (n) $ is given by (\ref{s31}) and (\ref{s32}).
\end{Proposition}

Now we use (\ref{s31}) and (\ref{s32}) to estimate $|a_2 (n)|$
from below. We present details for the case of even $n=2m$ only.
For odd $n$ technical details are the same.

Analysis of quadratic forms given by the sums in (\ref{s31}) and
(\ref{s32}) could give many examples of potentials (sequences
$(v_k)$) with specific properties. To prove our main Proposition
\ref{prop22} we will use the following.

\begin{Lemma}
\label{lem33} Let $(v_k), \, k \in \mathbb{N},$ be a real
$\ell^2$-sequence such that, for sufficiently large $k,$
\begin{equation}
\label{s34} |v_k| \leq \delta/k, \quad 0<\delta \leq \|v\|/15.
\end{equation}
Then we have, for sufficiently large $n,$
\begin{equation}
\label{s36} |a_2 (n)| \geq \frac{\|v\|^2}{32n^2}.
\end{equation}
\end{Lemma}

\begin{proof}
Consider the case $n=2m.$ By (\ref{s31}), with $k=|m-p| ,$ we have
\begin{equation}
\label{s37} 8 a_2 (n) = \sum_{k=1}^{m-1} \frac{1}{k(2m-k)}(v_k -
v_{2m-k})^2 -\sum_{k=1}^\infty \frac{1}{k(2m+k)}(v_k -
v_{2m+k})^2.
\end{equation}
Therefore,  $$ \sum_{k=m}^\infty \frac{1}{k(2m+k)}(v_k -
v_{2m+k})^2 \leq \frac{1}{3m^2} \sum_{k=m}^\infty(2|v_k|^2 +
2|v_{2m+k}|^2)\leq \frac{4}{3m^2}(\mathcal{E}_m (v))^2,$$ where
$$\mathcal{E}_m (v):= \left ( \sum_{k=m+1}^\infty |v_k|^2 \right
)^{1/2} \to 0 \quad \;\; m \to \infty.$$

 On
the other hand, $$ \sum_{k=1}^{m-1} \frac{1}{k(2m-k)}(v_k -
v_{2m-k})^2 -\sum_{k=1}^{m-1} \frac{1}{k(2m+k)}(v_k - v_{2m+k})^2
= \sum_{j=1}^5 A_j, $$ where $$ A_1 =\sum_{k=1}^{m-1}\frac{1}{k}
\left (\frac{1}{2m-k}-\frac{1}{2m+k} \right ) (v_k)^2
=\sum_{k=1}^{m-1}\frac{2}{(2m-k)(2m+k)} (v_k)^2 $$ $$\geq
\frac{2}{3m^2} \sum_{k=1}^{m-1} (v_k)^2 \geq \frac{2}{3m^2}\left
(\|v\|^2 - (\mathcal{E}_m (v))^2 \right );$$

$$ |A_2| =\left | \sum_{k=1}^{m-1}\frac{1}{k(2m-k)}(-2v_k
v_{2m-k}) \right | \leq \frac{2\delta}{m^2} \sum_{k=1}^{m-1}
\frac{|v_k|}{k} \leq \frac{3\|v\|\delta}{m^2}$$ by (\ref{s34}),
and $\sum |v_k|/k \leq (\sum |v_k|^2)^{1/2} (\sum 1/k^2)^{1/2}
\leq \|v\| \cdot \frac{\pi}{\sqrt{6}}\leq \frac{3}{2}\|v\|;$

$$ |A_3| =\left | \sum_{k=1}^{m-1}\frac{1}{k(2m+k)}(2v_k v_{2m+k})
\right | \leq \frac{2\delta}{4m^2}\sum_{k=1}^{m-1} \frac{|v_k|}{k}
\leq \frac{3\|v\|\delta}{4m^2}$$ by the same argument as above;

$$ A_4 = \sum_{k=1}^{m-1}\frac{1}{k(2m-k)}(v_{2m-k})^2 \leq
\frac{\delta}{m^2}  \sum_{k=1}^{m-1} \frac{|v_{2m-k}|}{k}
\leq\frac{\delta}{m^2} \cdot \mathcal{E}_m (v)
\cdot\frac{\pi}{\sqrt{6}} $$ by (\ref{s34}) and the Cauchy
inequality, as in the estimate of $A_2$;

$$ |A_5| = \left | - \sum_{k=1}^{m-1}\frac{1}{k(2m+k)}(v_{2m+k})^2
\right | \leq \frac{1}{4m^2}  \sum_{k=1}^{m-1}
\frac{|v_{2m+k}|}{k} \leq\frac{\delta}{4m^2} \cdot \mathcal{E}_m
(v) \cdot\frac{\pi}{\sqrt{6}}$$ by the same argument as above.

Since $\mathcal{E}_m (v) \to 0 $ as $m \to \infty, $ (\ref{s37}),
(\ref{s34}) and  the above inequalities imply, for large enough
$m,$ $$ 8 a_2 (2m) \geq \frac{1}{m^2}\left (\frac{1}{2} \|v\|^2 -
\frac{15 \|v\| \delta}{4m^2} \right )\geq \frac{\|v\|^2}{4m^2} .$$
This completes the proof of (\ref{s36}).
\end{proof}

Now we can complete the proof of Proposition \ref{prop22}. In view
of (\ref{s2}), the condition (\ref{s34}) holds for large enough
$k.$ Therefore, the inequalities (\ref{s15}) and (\ref{s36}) hold,
so $$ R_n \leq \frac{\sqrt{2} \sigma}{|a_2 (n)|}\leq Cn^2, \quad
C= \frac{32\sqrt{2}\sigma}{\|v\|^2}. $$
\end{proof}

2. {\em More comments.} In this subsection we will make a few
comments on applications of the Banach--Cacciopoli contraction
principle (see Section~6) to complement the references in
\cite{DM15}, Section 3.5.

In an unpublished 2000 manuscript \cite{Mit00} B. Mityagin used
the Banach--Cacciopoli contraction principle to prove density of
(complex--valued) finite--zone potentials of Hill operators in
$H(\Omega) $--spaces when $\Omega$ is a submultiplicative weight
(see Theorem 69 in \cite{DM15} for a precise statement). His
analysis dealt with ``tails`` (\ref{4.13a}); the ``head`` was not
important but the choice of sufficiently large $N$ to guarantee
that $H^N $ is contractive has been.

Following the same scheme,  P. Djakov and B. Mityagin (see
announcement in \cite{Mit03}), and independently (but with extra
conditions on $L^2$-potentials) B. Grebert and T. Kappeler
\cite{GK3}, proved density of finite--zone potentials of 1D Dirac
operator (see for an accurate statement \cite{DM15}, Theorem 70).
In particular, B. Grebert and T. Kappeler write in \cite{GK3}
(their paper appeared in January 2003 issue of {\em Asymptotic
Analysis}):

``{\em To prove Theorem 1.1 ... we follow the approach used in
\cite{Mit00}: as a set--up we take the Fourier block decomposition
introduced first for the Hill operator in \cite{KM1,KM2} and used
out subsequently for the Zaharov--Shabat operators in
\cite{GKM,GK1}. Unlike in \cite{Mit00} where a contraction mapping
argument was used to obtain the density results for the Hill
operator, we get a short proof of Theorem 1.1 by applying the
inverse function theorem in a straightforward way. As in
\cite{Mit00}, the main feature of the present proof is that it
does not involve any results from the inverse spectral theory.}``
(This is a word--by--word quote from \cite{GK3} but we changed its
reference numbers to fit the reference list of the present paper.)

According to his 2004 preprint \cite{P04} J. P\"oschel spent a few
months in 2003/04 in University of Z\"urich and had long
discussions with T. Kappeler. In \cite{P04} he combined ``tails``
and ``heads `` into the operator $A_N = H_N + T_N \in
(\ref{4.13.0}).$  Of course, this does not change the analytic
core of the proofs, i.e., the necessity of inequalities which
guarantee that $T_N,$ or $A_N$ are contractive, or -- B. Grebert
and T. Kappeler \cite{GK3} and J. P\"oschel \cite{P04} believe
that this is a simplification -- a version of Implicit Function
Theorem could be used instead. T. Kappeler and J. P\"oschel
\cite{KP1} claim in 2008 that to use Implicit Function Theorem in
the context of spectral gaps was ``{\em new functional analytic
approach}`` invented by J. P\"oschel. We will not argue with this
opinion but mention which elements  in \cite{P04} we have found
really useful for application of the Banach--Cacciopoli
contraction principle. Firstly, let us mention the following
statement\footnote{One of the authors (B.M.) thanks Professor Petr
Zabreiko (Belorussian State University, Minsk, Belarus), who
reminded him lovely discussions and wonderful atmosphere in the
led by Mark Krasnoselski seminar on non--linear functional
analysis and differential equations in Voronezh State University,
Voronezh, Russia, in 1962--1967.}.

\begin{Trick}
\label{prop44} (folklore of the 1950's) Let $X_1 \subset X_0 $ be
two Banach spaces with norms $\|x\|_0 \leq \|x\|_1. $ Suppose $F $
is an operator acting in both spaces, $F: X_j \to X_j, \quad
j=0,1,$ so that
\begin{equation}
\label{ss2} \|F(x)- F(y) \|_j \leq \frac{1}{2} \|x-y\|_j, \quad
x,y \in X_j, \;\; j=0,1.
\end{equation}
If $a \in X_0 $ and
\begin{equation}
\label{ss3} c=a+ F(a) \in X_1,
\end{equation}
then $a \in X_1. $
\end{Trick}

\begin{proof}
By (\ref{ss2}), the operator $y \to c-F(y) $ is a contraction in
$X_1.$  Therefore, by the Banach--Cacciopoli contraction
principle, it has a fixed point $p \in X_1, $ i.e.,
\begin{equation}
\label{ss5} p=c- F(p).
\end{equation}
By (\ref{ss3}) and (\ref{ss5}), $$a-p = F(p) - F(a). $$ On the
other hand, by (\ref{ss2}) with $j=0, $ $$ \|a-p\|_0 \leq
\frac{1}{2} \|a-p\|_0, $$ which yields $a=p, $ so $ a\in X_1. $
\end{proof}

We have explained this trick in the case when the operator is
defined on the entire spaces $X_j, \, j=0,1. $ Of course, as
usually, the Banach--Cacciopoli contraction principle is used in
Lemma~\ref{lem4.2}  when the operator acts on balls. We follow J.
P\"oschel \cite{P04}, when we introduce the weight (\ref{33.16})
and apply Trick~\ref{prop44}. But this ``soft`` analysis does not
help to avoid ``hard analysis`` of proving that the operators
involved are contractive -- by 2003, i.e., prior to either
\cite{GK3} or \cite{P04}, it has been done in \cite{KM2,DM3,DM5}
for Hill operators with $L^2$ complex--valued potentials. No
surprise, neither \cite{P04}, no \cite{P08} make any claims about
1D periodic Dirac operators (see ``hard analysis`` in \cite{DM6}
and \cite{DM15})
 or Hill operators with $H^{-1}$--potentials -- the latter
case is analyzed and done in the present paper. \vspace{5mm}

3. In Introduction -- see (\ref{sob}),(\ref{sob1}) -- "Sobolev"
 spaces of functions or weighted sequence spaces are defined
by weights $\Omega.$  In this paper we consider weights of the form
\begin{equation}
\label{d1} \Omega (m) = \frac{\omega (m)}{m} \quad \text{for} \; \;
m \neq 0, \quad \Omega (0) =1,
\end{equation}
where $\omega (m)$ is a sub--multiplicative weight such that $\log
(\omega (n))/n $  is monotone decreasing as $n \to \infty.$

Of course, classical examples of such weights are
\begin{equation}
\label{d1.10} \Omega (m) = |m|^a, \quad a \geq -1,
\end{equation}
\begin{equation}
\label{d1.11} \Omega (m) = |m|^s \exp (c|m|^b) , \quad 0<b<1,
\end{equation}
which give us Sobolev spaces $H^a$ or Gevrey spaces $G(b;s,c)$
correspondingly.

More generally, if $\varphi (x), \, x \geq 0,  $  is a concave
function such that $\varphi (0) = 0,$ then the weight
\begin{equation}
\label{d2} \omega (m) = \exp (\varphi |m|), \quad m \in \mathbb{Z},
\end{equation}
is sub--multiplicative,  $\log (\omega (n))/n $  is monotone
decreasing as $n \to \infty,$  so we can consider the corresponding
Sobolev space $H(\Omega)$ with $\Omega (m) = \omega (m)/m, \, m \neq
0.$  In this general setting a weight $\Omega$ could be chosen ''to
oscillate'' so that the space $H(\Omega)$ does not contain all
$C^\infty$ (or even Gevrey) functions and at the same time
$H(\Omega)$ is not contained in $H^a$  for every $ a> -1.$

Let us make this remark more formal and precise.

\begin{Lemma}
\label{lem8.20} Let the functions $a(x), b(x) \in C^2 ([0,\infty))$
satisfy the following conditions:

(i)  $a(0) = b(0) =0, \quad a(x)< b(x) \; \; \text{for} \; \; x>0,
\quad b(x) - a(x) \to \infty  \;\; \text{as} \; x \to \infty; $

(ii)  $a^\prime (x), \;b^\prime (x) > 0, \quad a^\prime (x),
\;b^\prime (x) \to 0 \;\; \text{as} \; x \to \infty;$

(iii) $a^{\prime \prime} (x), \;b^{\prime \prime} (x) < 0,  \quad
x\geq 0. $\\
Then, there is a concave function $g(x) $ and a sequence $c_k
\uparrow \infty $ such that
\begin{equation}
\label{d31} a(x) \leq g(x) \leq b(x)
\end{equation}
and
\begin{equation}
\label{d32} g(c_{2k-1})=b(c_{2k-1}), \quad g(c_{2k})=a(c_{2k}),
\quad k=1,2, \ldots.
\end{equation}
\end{Lemma}

\begin{proof}
We construct inductively  $c_k$ and $g(x)$ so that (\ref{d32}) holds
and  $g(x) $  is linear on the interval $[c_{2k-1},c_{2k+1}],\; k
\geq 1. $

Choose $c_1 >0$ large enough to guarantee that
\begin{equation}
\label{d40} a^\prime (x), \;b^\prime (x) \leq 1/2  \quad \text{for}
\;\; x \geq c_1
\end{equation}
and
\begin{equation}
\label{d50} b(x) - c(x) \geq 1 \quad \text{for} \;\; x \geq c_1.
\end{equation}
We set $g(x) = b(x)$ for $  0 \leq x \leq c_1,$ and
\begin{equation}
\label{d30} m_1 = \inf \{m: \; g(c_1) + m (x-c_1) \geq a(x) \quad
\text{for} \;\; x \geq c_1\}.
\end{equation}
Concavity of $a(x)$  and the initial condition
$$
g(c_1 ) = b(c_1) > a(c_1 )
$$
guarantee that $m_1$ is well--defined by (\ref{d30}) and there are
uniquely determined points $c_2, c_3$  such that
$$
m_1 (c_2 - c_1) + g(c_1) = a(c_2), \quad m_1 (c_3 - c_1) + g(c_1) =
b(c_3).
$$
Therefore, with $g(x) = m_1 (x - c_1) + g(c_1) $ for $x\in [c_1,
c_3]$ the condition (\ref{d32}) holds for $k=1.$

We continue by induction. Assuming that $c_1, \ldots, c_{2p-1}$ and
$g(x), \; x \leq c_{2p-1},$  are constructed, we set
$$
m_p = \inf \{m: \; g(c_{2p-1}) + m(x-c_{2p-1}) \geq a(x) \;\;
\text{for} \; x \geq c_{2p-1} \}.
$$
Then there are  uniquely determined points $c_{2p},\, c_{2p+1}$ such
that
$$
m_p (c_{2p} - c_{2p-1}) + g(c_{2p-1}) = a(c_{2p}), \quad m_p
(c_{2p+1}-c_{2p-1}) + g(c_{2p-1}) = b(c_{2p+1}).
$$
Therefore, with $g(x) = m_p (x - c_{2p-1}) + g(c_{2p-1}) $ for $x\in
[c_{2p-1},c_{2p+1}] $ the condition (\ref{d32}) holds for $k=p.$

Since $a(c_{2p}) \geq b(c_{2p-1})$,  (\ref{d40}) and (\ref{d50})
imply
$$ 1 \leq b(c_{2p})- a(c_{2p}) \leq b(c_{2p+1})-b(c_{2p-1}) \leq
(c_{2p+1}-c_{2p-1})/2. $$ Therefore, $c_{2p+1}-c_{2p-1} \geq 2,$ so
$c_k \to \infty.$

\end{proof}

Let $a(x), b(x), g(x)$ be the functions from Lemma \ref{lem8.20}. We
are going to define weight sequences using the values of  $a(x),
b(x), g(x)$ at integer points. Let
$$
n_k = [c_k], \quad \text{so} \;\; n_k \leq c_k < n_k +1,
$$
where $(c_k) $ is the sequence constructed in the proof of Lemma
\ref{lem8.20}. By its construction, in view of (\ref{d40}), the
function $g(x)$ is  piecewise linear for  $ x \geq c_1$ with
positive slopes $ m_p \leq 1/2.$ Therefore, the Mean Value Theorem
implies
\begin{equation}
\label{d61} b(n_k) -g(n_k) \geq b(c_k) -a(c_k) - 1/2  \quad
\text{for even} \; k
\end{equation}
and
\begin{equation}
\label{d62} g(n_k) -a(n_k) \geq b(c_k) -a(c_k) - 1/2  \quad
\text{for odd} \; k.
\end{equation}

Consider the weights $$A(m)= \frac{1}{|m|}e^{a(|m|)}, \quad B(m)=
\frac{1}{|m|}e^{b(|m|)},\quad G(m)= \frac{1}{|m|} e^{g(|m|)}, \quad
m \neq 0.$$ By (i) in Lemma \ref{lem8.20}, $b(x) - a(x) \to \infty $
as $x \to \infty. $  Therefore, (\ref{d61}) and (\ref{d62}) imply
\begin{equation}
\label{d71} \sup_n
\frac{B(n)}{G(n)} \geq \sup_{k} \frac{B(n_{2k})}{G(n_{2k})} =
\infty, \quad \sup_n \frac{G(n)}{A(n)} \geq \sup_{k}
\frac{G(n_{2k-1})}{A(n_{2k-1})} = \infty.
\end{equation}
We have $A(m) \leq G(m) \leq B(m),$  so in view of (\ref{d71}),
\begin{equation}
\label{d74} H(B) \varsubsetneqq H(G)\varsubsetneqq  H(A).
\end{equation}

\begin{Lemma}
\label{lem8.30} In the notations of Lemma \ref{lem8.20}, suppose
$f(x)\in C^2 ([0,\infty))$ is a function which satisfy (ii) and
(iii) and
\begin{equation}
\label{d81} a(x) \leq f(x) \leq b(x), \quad x\geq 0,
\end{equation}
\begin{equation}
\label{d82} f(x) - a(x) \to \infty, \quad b(x) - f(x) \to \infty
\quad \text{as} \;\; x\to\infty.
\end{equation}
Let $F$ be the corresponding the weight sequence $F(m) = |m|^{-1}
\exp f(|m|)$ for $m \neq 0, \, F(0) =1.$  Then
\begin{equation}
\label{d83} H(B) \subset H(F) \subset H(A)
\end{equation}
and
\begin{equation}
\label{d84} H(B) \varsubsetneqq H(G) \varsubsetneqq H(A),
\end{equation}
but
\begin{equation}
\label{d85} H(F) \not\subset H(G) \quad \text{and} \quad H(G)
\not\subset H(F).
\end{equation}
\end{Lemma}

\begin{proof}
Inequalities (\ref{d81}) and (\ref{d31}) imply the inclusions
(\ref{d83}), and (\ref{d84}) is explained in (\ref{d74}). The same
argument proves (\ref{d85}) because (\ref{d82}) implies
$$
\sup_n \frac{G(n)}{F(n)} \geq \sup_{k}
\frac{G(n_{2k-1})}{F(n_{2k-1})} \geq \sup_{k} \exp [-1/2
+b(c_{2k-1})-f(c_{2k-1})] =\infty
$$
and, in a similar way, $\sup_n  F(n)/G(n) = \infty. $
\end{proof}

Lemmas \ref{lem8.20} and \ref{lem8.30} give a variety of options to
construct weights with prescribed imbedding properties of related
Sobolev spaces.

{\em Example 1.}  For $-1 < \alpha <\beta, $ let
$$
a(x) = (\alpha +1) \log (x+e), \quad b(x) = (\beta +1) \log (x+e),
$$
and let $g(x), \, G $ be the corresponding function from Lemma
\ref{lem8.20} and the weight sequence $G(m) = |m|^{-1} \exp g(|m|)$
for $m \neq 0, \, G(0) =1.$  Then $ H^\beta  \subset H(G) \subset
H^\alpha $ but we have $H^\gamma \not\subset H(G)$ and  $H(G)
\not\subset H^\gamma $ for any $\gamma \in (\alpha, \beta).$

{\em Example 2.} Let $$a(x) = \log \log (x+e),\quad b(x) = x/\log
(x+e),$$ and let $g(x), \, G $ be the corresponding function from
Lemma \ref{lem8.20} and the weight sequence $G(m) = |m|^{-1} \exp
g(|m|)$ for $m \neq 0, \, G(0) =1.$

Then, for any weight $\Omega$  in (\ref{d1.10}) with $a>-1$ or in
(\ref{d1.11}), we have $H(\Omega) \not\subset H(G)$ and $H(G) \not
\subset H(\Omega).$  At the same time $H(G) \subset H^{-1}$ and all
analytic functions are in $H(G).$

\section{Appendix:Deviations of Riesz projections of Hill operators with singular
potentials}

It is shown that the deviations $P_n -P_n^0$ of Riesz projections
$$
 P_n = \frac{1}{2\pi i} \int_{C_n}
(z-L)^{-1} dz, \quad C_n=\{|z-n^2|= n\}, $$ of Hill operators $L y
= - y^{\prime \prime} + v(x) y, \; x \in [0,\pi],$ with zero and
$H^{-1}$ periodic potentials go to zero as $n \to \infty $ even if
we consider $P_n -P_n^0$ as operators from $L^1$ to $L^\infty. $
This implies that all $L^p$-norms are uniformly equivalent on the
Riesz subspaces $Ran \,P_n. $

\subsection{Preliminaries}

Now, in the case of singular potentials, we want to compare the
Riesz projections $P_n$ of the operator $L_{bc}, $ defined for
large enough $n$ by the formula
\begin{equation}
\label{9.03}
 P_n = \frac{1}{2\pi i} \int_{C_n}
(z-L_{bc})^{-1} dz, \quad C_n=\{|z-n^2|= n\},
\end{equation}
with the corresponding Riesz projections $P_n^0$ of the free
operator $L_{bc}^0$  (although $E_n^0 = Ran (P_n^0)$ maybe have no
common nonzero vectors with the domain of $L_{bc}).$

The main result is Theorem~\ref{thm91} (see Subsection 2 below),
which claims that
\begin{equation}
\label{9.05} \tilde{\tau}_n = \|P_n -P_n^0 \|_{L^1 \to L^\infty}
\to 0 \quad \text{as} \;\; n \to \infty.
\end{equation}

This implies  a sort of {quantum chaos}, namely all $L^p$--norms
on the Riesz subspaces $E_n = Ran P_n, $ for bc = $Per^\pm $ or
$Dir, $ are uniformly equivalent (see Theorem~\ref{thm97} in
Section 9.5).

In our analysis (see \cite{DM15}) of the relationship between
smoothness of a potential $v $ and the rate of decay of spectral
gaps and spectral triangles  a statement similar to (\ref{9.05})
\begin{equation}
\label{9.06} \tau_n = \|P_n -P_n^0 \|_{L^2 \to L^\infty} \to
0\quad \text{as} \;\; n \to \infty.
\end{equation}
was crucial when we used the deviations of Dirichlet eigenvalues
from periodic or anti--periodic eigenvalues to estimate the
Fourier coefficients of the potentials $v.$ But if $v \in L^2 $ it
was ''easy'' (see \cite{DM5}, Section 3, Prop.4, or \cite{DM15},
Prop.11). Moreover, those are strong estimates: for $n \geq
N(\|v\|_{L^2})$
\begin{equation}
\label{9.07} \tau_n \leq \frac{C}{n} \|v\|_{L^2},
\end{equation}
where $C$ is an absolute constant. Therefore, in (\ref{9.07}) only
the $L^2$--norm is important, so $\tau_n \leq CR/n$ holds {\em for
every $v$ in an $L^2$--ball of radius } $R.$

Just for comparison let us mention the same type of question in
the case of 1D periodic Dirac operators $$ MF = i \begin{pmatrix}
1 &0\\0  & -1
\end{pmatrix}
\frac{dF}{dx} + V  F,  \quad    0 \leq x\leq \pi, $$ where $V=
\begin{pmatrix} 0 &p\\ q  & 0
\end{pmatrix}, $
$p$ and $q$ are $L^2$--functions, and $ F= \begin{pmatrix}
f_1\\f_2
\end{pmatrix}.$

The boundary conditions under consideration are $Per^\pm $ and
$Dir,$ where $$Per^\pm : \;\; F(\pi) = \pm F(0), \qquad Dir:  \;
\; f_1 (0)=f_2 (0), \;\; f_1 (\pi) = f_2 (\pi). $$ Then (see
\cite{Mit04} or \cite{DM15}, Section 1.1) $$ E_n^0 = \left \{
\begin{pmatrix} a e^{-inx} \\be^{inx}
\end{pmatrix} :\;\; a,b \in \mathbb{C} \right \},\quad n \in
\mathbb{Z}, $$ where $n$ is even if $bc = Per^+$ and $n$ is odd if
$bc = Per^-,$ and $$ E_n^0 = \{c \sin nx, \; c \in \mathbb{C} \},
\quad n \in \mathbb{N} $$ if $ bc =Dir.$ Then for $$ Q_n =
\frac{1}{2\pi i} \int_{C_n} (\lambda - L)^{-1} d\lambda, \quad C_n
= \{\lambda: \; |\lambda - n| = 1/4 \},$$ we have $$ \rho_n (V)
:=\| Q_n -Q_n^0\|_{L^2 \to L^\infty } \to 0 \quad \text{as} \;\;
n\to \infty; $$ moreover, for any compact set $K \subset L^2$ and
$ V\in K,$ i.e., $p,q \in K $ one can construct a sequence
$\varepsilon_n (K) \to 0 $ such that $\rho_n (V) \leq
\varepsilon_n (K), \; V \in K. $ This has been proven in
\cite{Mit04}, Prop.8.1 and Cor.8.6; see Prop. 19 in \cite{DM15} as
well.

Of course, the norms $\tau_n$ in (\ref{9.06}) are larger than the
norms of these operators in $L^2$ $$ t_n = \| P_n -P_n^0\|_{L^2
\to L^2 } \leq \tau_n $$ and better (smaller) estimates for $t_n $
are possible. For example, A. Savchuk and A. Shkalikov proved
(\cite{SS03}, Sect.2.4) that $\sum t^2_n < \infty. $ This implies
(by Bari--Markus theorem -- see \cite{GK}, Ch.6, Sect.5.3, Theorem
5.2) that the spectral decompositions $$ f = f_N + \sum_{n>N} P_n
f $$ converge unconditionally. For Dirac operators the
Bari--Markus condition is $$ \sum_{n\in \mathbb{Z},|n|>N} \|Q_n -
Q_n^0\|^2 < \infty. $$ This fact (together with the completeness
and minimality of the system of Riesz subspaces $Ran \,Q_n$) imply
unconditional convergence of the spectral decompositions. This has
been proved in \cite{Mit03,Mit04} under the assumption that the
potential $V$ is in the Sobolev space $H^\alpha, \; \alpha>1/2 $
(see \cite{Mit04}, Theorem~8.8 for more precise statement). See
further comments in Section 9.5 below as well.

The proof of Theorem \ref{thm91}, or the estimates of norms
(\ref{9.05}), are based on the perturbation theory, which gives
the representation
\begin{equation}
\label{9.011}
 P_n -P_n^0=
\frac{1}{2\pi i} \int_{C_n} \left ( R(\lambda) -R^0 (\lambda)
\right ) d\lambda,
\end{equation}
where $R(\lambda) = (\lambda - L_{bc})^{-1} $ and $R^0 (\lambda )$
are the resolvents of $L_{bc}$ and of the free operator
$L^0_{bc},$ respectively. Often -- and certainly in the above
mentioned examples where $v \in L^2 $   --  one can get reasonable
estimates for the norms $ \| R(\lambda ) - R^0 (\lambda )\|$ on
the contour $C_n,$ and then by integration for $\|P_n -P_n^0 \|.$
But now, with  $v \in H^{-1}, $ we succeed to get good estimates
for the norms $\|P_n -P_n^0 \|$ {\em after} having integrated term
by term the series representation
\begin{equation}
\label{9.012} R-R^0 = R^0VR^0  +R^0VR^0 V R^0  + \cdots .
 \end{equation}
This integration kills or makes more manageable many terms, maybe
in their matrix representation. Only then we go to the norm
estimates. Technical details of this procedure (Subsection 9.3) is
the core of the proof of Theorem \ref{thm91}.

\subsection{Main result on the deviations $P_n -P_n^0$}

 By Proposition \ref{prop004}
 (i.e., our Theorem 21 in \cite{DM16} about spectra localization), the operator
$L_{Per\pm}$ has, for large enough $n,$ exactly two eigenvalues
(counted with their algebraic multiplicity) inside the disc of
radius $n$ about $n^2$ (periodic for even $n$ or antiperiodic for
odd $n$). The operator $L_{Dir}$ has one eigenvalue in these discs
for all large enough $n.$

Let $E_n$  be the corresponding Riesz invariant subspace, and let
$P_n$  be the corresponding Riesz projection, i.e., $$ P_n =
\frac{1}{2\pi i} \int_{C_n} (\lambda - L)^{-1} d\lambda, $$ where
$C_n = \{\lambda: \; |\lambda - n^2| =n \}.   $ We denote by $P_n^0$
the Riesz projection that corresponds to the free operator.

\begin{Proposition}
\label{prop91} In the above notations, for boundary conditions $bc
=Per^\pm $ or $Dir,$
\begin{equation}
\label{p20} \|P_n - P_n^0 \|_{L^2 \to L^\infty} \to 0 \quad
\text{as} \;\; n \to \infty.
\end{equation}
\end{Proposition}

As a matter of fact we will prove a stronger statement.

\begin{Theorem}
\label{thm91} In the above notations, for boundary conditions $bc
=Per^\pm $ or $Dir,$
\begin{equation}
\label{p21} \|P_n - P_n^0 \|_{L^1 \to L^\infty} \to 0 \quad
\text{as} \;\; n \to \infty.
\end{equation}
\end{Theorem}

\begin{proof}
We give a complete proof in the case $bc =Per^\pm.$ If $bc =Dir$ the
proof is the same, and only minor changes are necessary due to the
fact that in this case the orthonormal system of eigenfunctions of
$L^0$ is $\{ \sqrt{2} \sin nx, \; n\in \mathbb{N}\}$ (while it is
$\{\exp(imx), \; m\in 2 \mathbb{Z}\}$ for $bc = Per^+, $ and
$\{\exp(imx), \; m\in 1+ 2 \mathbb{Z}\}$ for $bc = Per^- $). So,
roughly speaking, the only difference is that when working with $bc =
Per^\pm $ the summation indexes in our formulas below run,
respectively, in $2 \mathbb{Z}$ and $1+ 2 \mathbb{Z},$ while for $bc
=Dir$ the summation indexes have to run in $\mathbb{N}.$ Therefore,
we consider in detail only $bc= Per^\pm, $ and provide some formulas
for the case $bc=Dir.$

Let
\begin{equation}
\label{p210} B_{km}(n):= \langle (P_n - P_n^0) e_m, e_k \rangle.
\end{equation}
We are going to prove that
\begin{equation}
\label{p22} \sum_{k,m} |B_{km}(n)|    \to 0 \quad \text{as} \;\; n
\to \infty.
\end{equation}
Of course, the convergence of the series in (\ref{p22}) means that
the operator with the matrix $B_{km} (n)$ acts from $\ell^\infty $
into $\ell^1.$

The Fourier coefficients of an $L^1$-function form an $\ell^\infty
$-sequence. On the other hand,
\begin{equation}
\label{p220} D= \sup_{x,n} |e_n (x)|  < \infty.
\end{equation}
 Therefore, the operators $P_n -P_n^0 $ act
from $L^1$ into $L^\infty $ (even into $C$) and
\begin{equation}
\label{p23} \|P_n - P_n^0 \|_{L^1 \to L^\infty} \leq D^2
\sum_{k,m} |B_{km}(n)|.
\end{equation}
Indeed, if  $\|f\|_{L^1} =1$ and  $f = \sum f_m e_m, $ then
$|f_m|\leq D $ and $$ (P_n - P_n^0 )f = \sum_k \left ( \sum_m
B_{km} f_m  \right ) e_k. $$

Taking into account (\ref{p220}), we get $$ \|(P_n - P_n^0) f
\|_{L^\infty} \leq D\sum_k \left | \sum_m B_{km} f_m \right | \leq
D^2\sum_k \sum_m |B_{km}|, $$ which proves (\ref{p23}).

In \cite{DM16}, Section 5, we gave a detailed analysis of the
representation $$ R_\lambda - R_\lambda^0 = \sum_{s=0}^\infty
K_\lambda (K_\lambda V  K_\lambda )^{s+1} K_\lambda, $$ where
$K_\lambda = \sqrt{R^0_\lambda} $ -- see \cite{DM16}, (5.13-14)
and what follows there. By (\ref{9.011}), $$ P_n - P_n^0 =
\frac{1}{2\pi i} \int_{C_n} \sum_{s=0}^\infty K_\lambda (K_\lambda
V  K_\lambda )^{s+1} K_\lambda d\lambda. $$ if the series on the
right converges. Thus
\begin{equation}
\label{p24} \langle (P_n - P_n^0) e_m, e_k \rangle =
   \sum_{s=0}^\infty
\frac{1}{2\pi i} \int_{C_n} \langle K_\lambda  (K_\lambda V
K_\lambda )^{s+1} K_\lambda e_m, e_k \rangle d\lambda,
\end{equation}
so we have
\begin{equation}
\label{p24b} \sum_{k,m} |\langle (P_n - P_n^0) e_m, e_k \rangle|
\leq \sum_{s=0}^\infty A(n,s),
\end{equation}
where
\begin{equation}
\label{p24c} A(n,s) =\sum_{k,m} \left | \frac{1}{2\pi i}
\int_{C_n} \langle K_\lambda (K_\lambda V K_\lambda )^{s+1}
K_\lambda e_m, e_k \rangle d\lambda \right |.
\end{equation}

By the matrix representation of the operators $K_\lambda $ and $V$
(see more details in \cite{DM16}, (5.15-22)) it follows that
\begin{equation}
\label{p24a} \langle K_\lambda (K_\lambda V  K_\lambda )K_\lambda
e_m, e_k \rangle = \frac{V(k-m)}{(\lambda -k^2)(\lambda -m^2)},
\quad k,m  \in n+2\mathbb{Z},
\end{equation}
for $bc = Per^\pm, $ and
\begin{equation}
\label{p25a} \langle K_\lambda (K_\lambda V  K_\lambda )K_\lambda
e_m, e_k \rangle =
\frac{|k-m|\tilde{q}(|k-m|)-(k+m)\tilde{q}(k+m)}{\sqrt{2}(\lambda
-k^2)(\lambda -m^2)}, \quad k,m \in \mathbb{N},
\end{equation}
for $bc = Dir. $ Let us remind that $\tilde{q}(m)$ are the sine
Fourier coefficients of the function $Q(x),$ i.e., $$ Q(x) =
\sum_{m=1}^\infty \tilde{q}(m) \sqrt{2} \sin mx. $$ The matrix
representations of $K_\lambda (K_\lambda V  K_\lambda )K_\lambda$
in (\ref{p24a}) and (\ref{p24b}) are the ''building blocks'' for
the matrices of the products of the form $K_\lambda (K_\lambda V
K_\lambda )^s K_\lambda$ that we have to estimate below. For
convenience, we set
\begin{equation}
\label{p26} V(m) = m w(m), \quad w \in \ell^2 (2\mathbb{Z}), \quad
r(m)= max(|w(m)|,|w(-m)|)
\end{equation}
if $bc = Per^\pm, $ and
\begin{equation}
\label{p260} \tilde{q}(0)=0, \quad r(m) = \tilde{q}(|m|), \quad m
\in \mathbb{Z}.
\end{equation}
if $bc = Dir.$ We use the notations (\ref{p26}) in the estimates
related to $bc =Per^\pm $ below, and if one would use in a similar
way (\ref{p260}) in the Dirichlet case, then the corresponding
computations becomes practically identical (the only difference
will be that in the Dirichlet case the summation will run over
$\mathbb{Z}$). So, further we consider only  the case $bc
=Per^\pm.
$

Let us calculate the first term on the right--hand side of
(\ref{p24}) (i.e., the term coming for $s=0$). We have
\begin{equation}
\label{p25} \frac{1}{2\pi i} \int_{C_n} \frac{V(k-m)}{(\lambda
-k^2)(\lambda -m^2)}d\lambda = \begin{cases} \frac{V(k \mp
n)}{(n^2 -k^2)}   &  m= \pm n, \;\; k \neq \pm n, \\ \frac{V(\pm n
-m)}{(n^2 -m^2)}   &  k= \pm n, \;\; m \neq \pm n, \\ 0 &
\text{otherwise}.
\end{cases}
\end{equation}
Thus $$ A(n,0) =\sum_{k,m} \left | \frac{1}{2\pi i} \int_{C_n}
\langle K_\lambda (K_\lambda V  K_\lambda )K_\lambda e_m, e_k
\rangle \right | $$ $$ = \sum_{k \neq \pm n} \frac{|V(k-n)|}{|n^2
-k^2|} + \sum_{k \neq \pm n} \frac{|V(k+n)|}{|n^2 -k^2|} + \sum_{m
\neq \pm n} \frac{|V(-n+m)|}{|n^2 -m^2|} + \sum_{m \neq \pm n}
\frac{|V(n-m)|}{|n^2 -m^2|}. $$

By the Cauchy inequality, we estimate the first sum on the
right--hand side:
\begin{equation}
\label{p26a} \sum_{k \neq \pm n} \frac{|V(k-n)|}{|n^2 -k^2|}
=\sum_{k \neq \pm n} \frac{ |k-n| |w(k-n)|}{|n^2 -k^2|}
\end{equation}
$$ \leq \sum_{k \neq -n} \frac{r(k-n)}{|n+k|}  \leq \sum_{k>0}
\cdots + \sum_{k\leq 0, k \neq  -n} \cdots $$

$$ \leq \left ( \sum_{k>0} \frac{1}{|n+k|^2} \right )^{1/2} \cdot
\|r\|+ \left ( \sum_{k\leq 0, k \neq  -n} \frac{1}{|n+k|^2} \right
)^{1/2}  \left ( \sum_{k\leq 0} (r(n-k))^2 \right )^{1/2} $$ $$
\leq \frac{\|r\|}{\sqrt{n}} + \mathcal{E}_n (r). $$ Since each of
the other three sums could be estimated in the same way, we get
\begin{equation}
\label{p27} A(n,0) \leq \sum_{k,m}  \left | \frac{1}{2\pi i}
\int_{C_n} \langle K_\lambda (K_\lambda V  K_\lambda )K_\lambda
e_m, e_k \rangle   d \lambda \right |  \leq
\frac{4\|r\|}{\sqrt{n}} + 4 \mathcal{E}_n (r).
\end{equation}

Next we estimate $A(n,s), s\geq 1.$ By the matrix representation
of $K_\lambda $ and $V$ -- see (\ref{p24a}) -- we have
\begin{equation}
\label{p27a} \langle K_\lambda  (K_\lambda V  K_\lambda )^{s+1}
K_\lambda e_m, e_k \rangle = \frac{\Sigma
(\lambda;s,k,m)}{(\lambda - k^2)(\lambda - m^2)}
\end{equation}
where
\begin{equation}
\label{p28} \Sigma (\lambda;s,k,m) =\sum_{j_1, \ldots, j_s}
\frac{V(k-j_1) V(j_1 - j_2) \cdots V(j_{s-1} -j_s)V(j_s -m)}
{(\lambda -j_1^2) (\lambda -j_2^2)\cdots
 (\lambda -j_s^2) },
\end{equation}
$ k, m, j_1, \ldots, j_s \in n+ 2\mathbb{Z}. $ For convenience, we
set also
\begin{equation}
\label{p28a} \Sigma (\lambda;0,k,m) = V(k-m).
\end{equation}

In view of (\ref{p24c}), we have
\begin{equation}
\label{p29} A(n,s) = \sum_{k,m} \left |
 \frac{1}{2\pi i} \int_{C_n}
\frac{\Sigma (\lambda;s,k,m)}{(\lambda - k^2)(\lambda - m^2)}
d\lambda \right |.
\end{equation}

Let us consider the following sub--sums of sum $\Sigma
(\lambda;s,k,m)$ defined in (\ref{p28}):
\begin{equation}
\label{p30} \Sigma^0 (\lambda;s,k,m) =\sum_{j_1, \ldots, j_s \neq
\pm n} \cdots \quad \text{for} \;s \geq 1, \quad \Sigma^0
(\lambda;0,k,m) := V(k-m);
\end{equation}
\begin{equation}
\label{p30a} \Sigma^1 (\lambda;s,k,m) = \sum_{\exists \;
\text{one} \; j_\nu = \pm n}  \cdots \quad  \text{for} \;s \geq 1;
\end{equation}
\begin{equation}
\label{p31} \Sigma^* (\lambda;s,k,m) = \sum_{\exists j_\nu = \pm
n}  \cdots, \quad \Sigma^{**} (\lambda;s,k,m) = \sum_{\exists
j_\nu, j_\mu = \pm n}  \cdots, \quad s \geq 2
\end{equation}
(i.e.,  $\Sigma^0 $ is the sub--sum of the sum $\Sigma $ in
(\ref{p28}) over those indices $j_1, \ldots, j_s$ that are
different from $\pm n ,$ in $\Sigma^1 $ exactly one summation
index is equal to $\pm n,$ in $\Sigma^* $ at least one summation
index is equal to $\pm n,$ and in $\Sigma^{**} $ at least two
summation indices are equal to $\pm n).$  Notice that $$ \Sigma
(\lambda;s,k,m)=\Sigma^0 (\lambda;s,k,m) + \Sigma^*
(\lambda;s,k,m), \quad s\geq 1, $$ and $$  \Sigma (\lambda;s,k,m)
= \Sigma^0 (\lambda;s,k,m) +\Sigma^1 (\lambda;s,k,m) + \Sigma^{**}
(\lambda;s,k,m), \quad s\geq 2. $$

In these notations we have
\begin{equation}
\label{p32} \sum_{m,k \neq \pm n} \left |
 \frac{1}{2\pi i} \int_{C_n}
\frac{\Sigma^0 (\lambda;s,k,m)}{(\lambda - k^2)(\lambda - m^2)}
d\lambda \right | = 0
\end{equation}
because, for $m,k \neq \pm n,$ the integrand is an analytic
function of $\lambda $ in the disc $\{\lambda: \; |\lambda - n^2|
\leq n/4 \}.$

Therefore, $A(n,s)$ could be estimated as follows:
\begin{equation}
\label{p33a} A(n,1) \leq  \sum_{i=1}^5 A_i (n,1),
\end{equation}
and
\begin{equation}
\label{p33} A(n,s) \leq  \sum_{i=1}^7 A_i (n,s), \quad s \geq 2,
\end{equation}
where
\begin{equation}
\label{p331} A_1 (n,s) = \sum_{k,m = \pm n} n \cdot \sup_{\lambda
\in C_n } \left | \frac{\Sigma (\lambda;s,k,m)}{(\lambda -
k^2)(\lambda - m^2)} \right |,
\end{equation}
\begin{equation} \label{p332}
A_2 (n,s) = \sum_{k= \pm n, m \neq \pm n} n \cdot \sup_{\lambda
\in C_n } \left | \frac{\Sigma^0 (\lambda;s,k,m)}{(\lambda -
k^2)(\lambda - m^2)} \right |,
\end{equation}
\begin{equation} \label{p333}
A_3 (n,s) = \sum_{k= \pm n, m \neq \pm n} n \cdot \sup_{\lambda
\in C_n } \left | \frac{\Sigma^* (\lambda;s,k,m)}{(\lambda -
k^2)(\lambda - m^2)} \right |,
\end{equation}
\begin{equation} \label{p334}
A_4 (n,s) = \sum_{k \neq \pm n, m = \pm n} n \cdot \sup_{\lambda
\in C_n } \left | \frac{\Sigma^0 (\lambda;s,k,m)}{(\lambda -
k^2)(\lambda - m^2)} \right |,
\end{equation}
\begin{equation} \label{p335}
A_5 (n,s) = \sum_{k \neq \pm n, m = \pm n} n \cdot \sup_{\lambda
\in C_n } \left | \frac{\Sigma^* (\lambda;s,k,m)}{(\lambda -
k^2)(\lambda - m^2)} \right |,
\end{equation}
\begin{equation} \label{p336}
A_6 (n,s) = \sum_{k,m \neq \pm n} n \cdot \sup_{\lambda \in C_n }
\left | \frac{\Sigma^1 (\lambda;s,k,m)}{(\lambda - k^2)(\lambda -
m^2)}\right |,
\end{equation}
\begin{equation} \label{p337}
 A_7 (n,s) = \sum_{k,m \neq \pm n} n \cdot
\sup_{\lambda \in C_n } \left | \frac{\Sigma^{**}
(\lambda;s,k,m)}{(\lambda - k^2)(\lambda - m^2)} \right |.
\end{equation}
First we estimate $A_1 (n,s). $ By (\ref{p24a}) and \cite{DM16},
Lemma 19 (inequalities (5.30),(5.31)),
\begin{equation}
\label{p34} \sup_{\lambda \in C_n} \|K_\lambda \|
=\frac{2}{\sqrt{n}}, \quad \sup_{\lambda \in C_n} \| K_\lambda V
K_\lambda \| \leq \rho_n := C \left (\frac{\|r\|}{\sqrt{n}} +
\mathcal{E}_{\sqrt{n}} (r)\right  ),
\end{equation}
where $r = (r(m))$ is defined by the relations (\ref{p26}) and $C$
is an absolute constant.

\begin{Lemma}
\label{lemp0} In the above notations
\begin{equation}
\label{p36} \sup_{\lambda \in C_n } \left | \frac{\Sigma
(\lambda;s,k,m)}{(\lambda - k^2)(\lambda - m^2)} \right | \leq
\frac{1}{n} \rho_n^{s+1}.
\end{equation}
\end{Lemma}

\begin{proof}
Indeed, in view of (\ref{p28}) and  (\ref{p34}), we have $$ \left
| \frac{\Sigma (\lambda;s,k,m)}{(\lambda - k^2)(\lambda - m^2)}
\right |
=
|\ K_\lambda  (K_\lambda V  K_\lambda )^{s+1} K_\lambda e_k, e_m
\rangle | $$ $$ \leq \| K_\lambda  (K_\lambda V  K_\lambda )^{s+1}
K_\lambda \| \leq  \| K_\lambda \| \cdot
 \| K_\lambda V  K_\lambda \|^{s+1} \cdot  \| K_\lambda \| \leq
\frac{1}{n} \rho_n^{s+1}, $$ which proves (\ref{p36}).
\end{proof}

Now we estimate $A_1 (n,s).$ By (\ref{p36}),
\begin{equation}
\label{p37} A_1 (n,s) = \sum_{m,k = \pm n} n \cdot \sup_{\lambda
\in C_n } \left | \frac{\Sigma (\lambda;s,k,m)}{(\lambda -
k^2)(\lambda - m^2)} \right | \leq  4 \rho_n^{s+1}.
\end{equation}

To estimate $A_2 (n,s),$ we consider $\Sigma^0 (\lambda; s,k,m)$
for $k = \pm n.$  From the elementary inequality
\begin{equation}
\label{p38} \frac{1}{|\lambda - j^2|} \leq \frac{2}{|n^2 - j^2|}
\quad \text{for} \quad  \lambda \in C_n, \; j \in n + 2
\mathbb{Z}, \; j \neq \pm n,
\end{equation}
it follows, for  $m \neq \pm n,$
\begin{equation}
\label{p39} \sup_{\lambda \in C_n } \left | \frac{\Sigma^0
(\lambda;s,\pm n,m)}{(\lambda - n^2)(\lambda - m^2)} \right | \leq
\frac{1}{n} \cdot 2^{s+1} \times
\end{equation}
$$ \times
  \sum_{j_1,\ldots, j_s \neq \pm n} \frac{|V(\pm n -j_1 )V(j_1
-j_2) \cdots V(j_{s-1}- j_s)V(j_s -m)|} {|n^2 - j_1^2 | |n^2-
j_2^2 |\cdots |n^2- j_s^2||n^2 -m^2 |}. $$ Thus, taking the sum of
both sides of (\ref{p39}) over $m \neq \pm n, $ we get

\begin{equation}
\label{p40} A_2 (n,s) \leq 2^{s+1} \left [ L(s+1,n) + L(s+1,-n)
\right ],
\end{equation}
where
\begin{equation}
\label{p41} L(p,d) :=
  \sum_{i_1,\ldots, i_p \neq \pm n} \frac{|V(d-i_1 )|}{|n^2 - i_1^2 |}
\cdot \frac{|V(i_1-i_2 )|}{|n^2 - i_2^2 |} \cdots
\frac{|V(i_{p-1}-i_p )|}{|n^2 - i_p^2 |}.
\end{equation}

The roles of $k$ and $m$ in $A_2 (n,s) $ and $A_4 (n,s) $ are
symmetric, so $A_4 (n,s) $ could be estimated in an analogous way.
Indeed, for $k \neq \pm n,$ we have
\begin{equation}
\label{p39a} \sup_{\lambda \in C_n } \left | \frac{\Sigma^0
(\lambda;s,k,\pm n)}{(\lambda - k^2)(\lambda - n^2)} \right | \leq
\frac{1}{n} \cdot 2^{s+1} \times
\end{equation}
$$ \times
  \sum_{j_1,\ldots, j_s \neq \pm n} \frac{|V(k -j_1 )V(j_1
-j_2) \cdots V(j_{s-1}- j_s)V(j_s -\pm n)|} {|n^2 - k^2 ||n^2 -
j_1^2 | |n^2- j_2^2 |\cdots |n^2- j_s^2|}. $$ Thus, taking the sum
of both sides of (\ref{p39a}) over $k \neq \pm n, $ we get
\begin{equation}
\label{p42} A_4 (n,s) \leq 2^{s+1} \left [ R(s+1,n) + R(s+1,-n)
\right ],
\end{equation}
where
\begin{equation}
\label{p43} R(p,d) :=
  \sum_{i_1,\ldots, i_p \neq \pm n}
 \frac{|V(i_1-i_2 )|}{|n^2 - i_1^2 |} \cdots
\frac{|V(i_{p-1}-i_p )|}{|n^2 - i_{p-1}^2 |} \cdot \frac{|V(i_p -
d )|}{|n^2 - i_p^2 |}.
\end{equation}

Below (see Lemma \ref{lemp1} and its proof in Subsection 9.3)  we
estimate the sums $ L(p,\pm n) $ and $ R(p,\pm n). $ But now we
are going to show that $A_i (n,s), \; i = 3, 5, 6, 7, $ could be
estimated in terms of $L$ and $R$ from (\ref{p41}), (\ref{p43}) as
well.

To estimate $ A_6 (n,s)$ we write the expression $\frac{\Sigma^1
(\lambda;s,k,m)} {(\lambda - k^2)(\lambda - m^2)} $ in the form $$
\sum_{\nu =1}^s \sum_{d =\pm n} \frac{1}{\lambda- k^2}\Sigma^0
(\lambda; \nu -1, k, d) \frac{1}{\lambda- n^2}\Sigma^0 (\lambda;s
- \nu, d, m) \frac{1}{\lambda- m^2} $$ By (\ref{p38}), the
absolute values of the terms of this double sum do not exceed:

(a) for $\nu = 1 $ $$ 2^{s+1} \cdot \frac{|V(k- \pm n)|}{|n^2 -
k^2|} \cdot \frac{1}{n} \cdot \sum_{i_1, \ldots,  i_{s-1}\neq \pm
n} \frac{|V(\pm n - i_1)|  |V(i_1 -i_2) | \cdots  |V(i_{s-1} -m)
|} {|n^2 -i_1^2| \cdots  |n^2 -i_{s-1}^2|  |n^2 -m^2|}. $$

(b) for $\nu = s $ $$ 2^{s+1} \cdot   \left ( \sum_{i_1, \ldots,
i_{s-1}\neq \pm n} \frac{|V(k - i_1)| |V(i_1 - i_2)| \cdots
|V(i_{s-1} -\pm n)|} {|n^2 -k^2||n^2 -i_1^2| |n^2 -i_2^2| \cdots
|n^2 -i_{s-1}^2| } \right ) \cdot \frac{1}{n} \cdot \frac{|V(\pm n
-m)|}{|n^2 - m^2|} $$

(c) for $ 1 < \nu < s $ $$ 2^{s+1} \cdot   \left ( \sum_{i_1,
\ldots,  i_{\nu-1}\neq \pm n} \frac{|V(k - i_1)| |V(i_1 - i_2)|
\cdots  |V(i_{\nu -1} -\pm n)|} {|n^2 -k^2| |n^2 -i_1^2| |n^2
-i_2^2| \cdots  |n^2 -i_{\nu -1}^2| } \right ) \cdot \frac{1}{n}
$$ $$    \times \; \sum_{i_1, \ldots,  i_{s-\nu} \neq \pm n}
\frac{|V(\pm n - i_1)|  |V(i_1 -i_2) | \cdots  |V(i_{s-\nu} -m) |}
{|n^2 -i_1^2| \cdots  |n^2 -i_{s-\nu}^2|  |n^2 -m^2|}. $$

Therefore, taking the sum over $m,k \neq \pm n, $ we get
\begin{equation}
\label{p44} A_6 (n,s) \leq 2^{s+1} \cdot \sum_{\nu =1}^s \sum_{d =
\pm n} R(\nu,d) \cdot L(s+1-\nu ,d).
\end{equation}

One could estimate  $A_3(n,s), A_5 (n,s) $ and $A_7 (n,s) $ in an
analogous way. We will write the core formulas but omit some
details.

To estimate $A_3 (n,s), $ we use the identity $$ \frac{\Sigma
(\lambda;s,k, \pm n )}{(\lambda -k^2)(\lambda - n^2)} =\sum_{\nu
=1}^s \sum_{d=\pm n} \frac{1}{\lambda - k^2} \Sigma^0 (\lambda;\nu
-1, k, d) \frac{1}{\lambda - n^2}  \frac{\Sigma (\lambda;s- \nu, d,
\pm n)}{\lambda - n^2}. $$ In view of (\ref{p36}), (\ref{p38}) and
(\ref{p43}), from here it follows that
\begin{equation}
\label{p45} A_3 (n,s) \leq 2^{s+1} \cdot \sum_{\nu =1}^s \sum_{d =
\pm n}
 R(\nu ,d) \cdot \rho_n^{s-\nu+1}.
\end{equation}

We estimate $A_5 (n,s) $ by using the identity $$ \frac{\Sigma
(\lambda;s,\pm n, m)}{(\lambda -n^2)(\lambda - m^2)} =\sum_{\nu =1}^s
\sum_{d=\pm n} \frac{1}{\lambda - n^2} \Sigma (\lambda;\nu -1, \pm n,
d) \frac{1}{\lambda - n^2}  \frac{\Sigma^0 (\lambda;s- \nu, d,
m)}{\lambda - m^2}. $$ In view of (\ref{p36}), (\ref{p38}) and
(\ref{p41}), from here it follows that
\begin{equation}
\label{p46} A_5 (n,s) \leq 2^{s+1} \cdot \sum_{\nu =1}^s \sum_{d =
\pm n}
 \rho_n^{\nu} \cdot L(s-\nu+1,d).
\end{equation}

Finally, to estimate $A_7 (n,s) $ we use the identity
$$\frac{\Sigma (\lambda;s,k, m)}{(\lambda -k^2)(\lambda - m^2)} =
\sum_{1 \leq \nu < \mu \leq s }^s \sum_{d_1, d_2=\pm n}
\frac{1}{\lambda - k^2} \Sigma^0 (\lambda;\nu -1, k, d_1)  \times
$$$$ \times \; \frac{1}{\lambda - n^2}  \Sigma (\lambda; \mu- \nu
-1, d_1, d_2) \frac{1}{\lambda - n^2} \Sigma^0 (\lambda;s -\mu,
d_2, m) \frac{1}{\lambda - m^2} $$ In view of (\ref{p36}),
(\ref{p38}), (\ref{p41}) and (\ref{p43}), from here it follows
that
\begin{equation}
\label{p47} A_7 (n,s) \leq 2^{s} \cdot \sum_{1 \leq \nu < \mu \leq
s } \sum_{d_1, d_2=\pm n}  R(\nu,d_1) \cdot \rho_n^{\mu-\nu} \cdot
L(s-\mu+1,d_2).
\end{equation}

Next we estimate $L(p,\pm n) $ and $R(p,\pm n).$ Changing the
indices in (\ref{p43}) by $$ j_\nu = -i_{p+1-\nu}, \quad  1\leq
\nu \leq p, $$ we get
\begin{equation}
\label{p51} R(p,d) = L(p, -d).
\end{equation}
Therefore, it is enough to estimate only $L(p, \pm n).$

\begin{Lemma}
\label{lemp1} In the above notations, there exists a sequence of
positive numbers $\varepsilon_n \to 0 $ such that, for large
enough $n,$
\begin{equation}
\label{p52}  L(s, \pm n) \leq (\varepsilon_n)^s, \quad \forall
\,s\in \mathbb{N}.
\end{equation}

\end{Lemma}

The proof of this lemma is technical. It is given in detail in
Section 3. Then in Section 4 we complete the proof of Theorem
\ref{thm91}. With (\ref{p51}) and (\ref{p52}), in Section 4 we
will use Lemma \ref{lemp1} in the following form.

\begin{Corollary}
\label{cor1}
 In the above notations, there exists a sequence of
positive numbers $\varepsilon_n \to 0 $ such that, for large
enough $n,$
\begin{equation}
\label{p520} \max\{ L(s, \pm n), R(s, \pm n) \} \leq
(\varepsilon_n)^s, \quad \forall \,s\in \mathbb{N}.
\end{equation}
\end{Corollary}

\subsection{Technical inequalities and their proofs}

We follow the notations from Section 2. Now we prove Lemma
\ref{lemp1}.

\begin{proof}
First we show that
\begin{equation}
\label{p53}  L(s, \pm n) \leq \sigma (n,s),\quad  s \geq 1,
\end{equation}
where
\begin{equation}
\label{p11a}  \sigma (n,1) =\sum_{j_1 \neq \pm n}
\frac{r(n+j_1)}{|n^2 - j_1^2|},
\end{equation}
for $s\geq 2 $
\begin{equation}
\label{p11} \sigma (n,s) := \sum_{j_1,\ldots, j_s \neq \pm n}
\left (\frac{1}{|n-j_1|} +\frac{1}{|n+j_2|} \right )  \cdots \left
(\frac{1}{|n-j_{s-1}|} +\frac{1}{|n+j_s|} \right )
\frac{1}{|n-j_s|}
\end{equation}
$$ \times \;\;  r(n+j_1)r(j_1 +j_2) \cdots r(j_{s-1}+j_s), $$ and
the sequence $r= (r(m)) $ is defined by (\ref{p26}).

For $s=1$ we have, with $i_1 = -j_1,$ $$ L(1,n) = \sum_{j_1 \neq \pm
n} \frac{|V(n-j_1)|}{|n^2 - j_1^2|} =\sum_{i_1 \neq \pm n}
\frac{|V(n+i_1)|}{|n^2 - i_1^2|}$$$$ =\sum_{i_1 \neq \pm n}
\frac{|w(n+i_1)|}{|n - i_1|} \leq \sum_{i_1 \neq \pm n}
\frac{r(n+i_1)}{|n - i_1|} $$ (where (\ref{p26}) is used). In an
analogous way we get $$ L(1,-n) = \sum_{j_1 \neq \pm n}
\frac{|V(-n-j_1)|}{|n^2 - j_1^2|} =\sum_{i_1 \neq \pm n}
\frac{|w(-n-i_1)|}{|n - i_1|} \leq \sum_{i_1 \neq \pm n}
\frac{r(n+i_1)}{|n - i_1|}, $$ so, (\ref{p53}) holds for $s=1.$

Let $s \geq 2.$ Changing the indices of summation in (\ref{p41})
(considered with $p=s $ and $d=n)$ by $ j_\nu = (-1)^\nu i_\nu, $ we
get $$  L(s,n) = \sum_{j_1,\ldots, j_s \neq \pm n}
\frac{|V(n+j_1)|}{|n^2 -j_1^2 |} \frac{|V(-j_1 -j_2)|}{|n^2 -j_1^2 |}
\cdots \frac{|V[ (-1)^{s-1}(j_{s-1}+j_s)]|}{|j_s^2 -n^2|} $$ $$
=\sum_{j_1,\ldots, j_s \neq \pm n} \frac{|n+ j_1||j_2 + j_1| \cdots
|j_s +j_{s-1}|}{|j_1^2 - n^2| |j_2^2 - n^2|\cdots |j_s^2 -n^2|} |w(n+
j_1) \cdots w[ (-1)^{s-1}(j_{s-1}+j_s)]|
$$ $$ \leq
 \sum_{j_1,\ldots, j_s \neq \pm n} \frac{|j_2
+j_1| \cdots |j_s +j_{s-1}|}{|n -j_1| |n^2 -j_2^2 |\cdots |n^2 -
j_s^2 |} r(n+j_1)r(j_1 +j_2) \cdots r(j_{s-1}+j_s). $$

By the identity $$ \frac{i+k}{(n-i)(n+k)} = \frac{1}{n-i}
-\frac{1}{n+k}, $$ we get that the latter sum does not exceed

$$ \sum_{j_1,\ldots, j_s \neq \pm n} \left |\frac{1}{n-j_1}
-\frac{1}{n+j_2} \right |  \cdots \left |\frac{1}{n-j_{s-1}}
-\frac{1}{n+j_s} \right | \frac{1}{|n-j_s|} $$ $$
 \times  \;r(n+j_1)r(j_1 +j_2) \cdots r(j_{s-1}+j_s)
 \leq \sigma (n,s). $$

Changing the indices of summation in (\ref{p41}) (considered with
$p=s $ and $d=-n)$ by $ j_\nu = (-1)^{\nu +1} i_\nu, $ one can
show that $L(s,-n) \leq \sigma (n,s). $ Since the proof is the
same we omit the details. This completes the proof of (\ref{p53}).

In view of (\ref{p53}), Lemma \ref{lemp1} will be proved if we
show that there exists a sequence of positive numbers
$\varepsilon_n \to 0$ such that, for large enough $n,$
\begin{equation}
\label{p55} \sigma (n,s) \leq  (\varepsilon_n)^s, \quad \forall
\,s\in \mathbb{N}.
\end{equation}

In order to prove (\ref{p55}) we need the following statements.

\begin{Lemma}
\label{lemp2} Let $r = (r(k)) \in \ell^2 (2\mathbb{Z}), \; r(k)
\geq 0, $ and let
\begin{equation}
\label{p0} \sigma_1 (n,s;m) = \sum_{j_1, \ldots, j_s \neq   n}
\frac{r(m+j_1)}{|n-j_1|}\frac{r(j_1+j_2)}{|n-j_2|} \cdots
\frac{r(j_{s-1}+j_s)}{|n-j_s|}, \quad n,s \in \mathbb{N},
\end{equation}
where $ m, j_1, \ldots, j_s \in n+ 2\mathbb{Z}. $ Then, with
\begin{equation}
\label{p1} \tilde{\rho}_n := \mathcal{E}_n (r) + 2\|r\|/\sqrt{n},
\end{equation}
we have, for $ n \geq 4,$
\begin{equation}
\label{p2} \sigma_1 (n,1;m) \leq \begin{cases} \tilde{\rho}_n
\quad & \text{if} \;\; |m-n| \leq n/2,\\ \|r\| \quad & \text{for
arbitrary} \; \; m \in n+2\mathbb{Z},
\end{cases}
\end{equation}
\begin{equation}
\label{p3} \sigma_1 (n,2p;m) \leq  (2\|r\| \tilde{\rho}_n)^p,
\quad \sigma_1 (n,2p + 1;m) \leq \|r\| \cdot (2\|r\|
\tilde{\rho}_n)^p.
\end{equation}

\end{Lemma}

\begin{proof}  Let us recall that
$$\sum_{k=1}^\infty \frac{1}{k^2} = \pi^2/6, \quad
\sum_{k=n+1}^\infty \frac{1}{k^2} <\sum_{k=n+1}^\infty \left
(\frac{1}{k-1} - \frac{1}{k} \right ) = \frac{1}{n}. $$ Therefore,
one can easily see that
 $$
\sum_{i\in 2\mathbb{Z}, i\neq 0} \frac{1}{i^2} = \pi^2/12 < 1,
\quad \sum_{i\in 2\mathbb{Z}, |i|> n/2} \frac{1}{i^2} < 4/n, \quad
n \geq 4.$$

By the Cauchy inequality, $$ \sigma_1 (n,1;m) = \sum_{j_1 \neq
n}\frac{r(m+j_1)}{|n-j_1|} \leq  \left ( \sum_{j_1 \neq \pm n}
|n-j_1 |^{-2}  \right  )^{1/2} \cdot \|r\| \leq \|r\|, $$ which
proves the second case in (\ref{p2}).

If $|m-n| \leq  n/2 $ then we have $ n/2 \leq m \leq 3n/2. $ Let
us write $\sigma_1 (n,1;m)$ in the form $$\sigma_1 (n,1;m) =
\sum_{ 0<|j_1-n|\leq n/2} \frac{r(m+j_1)}{|n-j_1|} +
\sum_{|j_1-n|>n/2} \frac{r(m+j_1)}{|n-j_1|} $$ and apply the
Cauchy inequality to each of the above sums. In the first sum  $
n/2 \leq j \leq 3n/2, $ so $ j+m \geq n, $ and therefore, we get
$$ \sigma_1 (n,1;m) \leq \left ( \sum_{ i\geq n} |r(i)|^2 \right
)^{1/2} \cdot 1 + \|r\| \cdot \left ( \sum_{|n-j_1|>n/2} |j_1 -
n|^{-2}   \right )^{1/2}. $$ Thus $$ \sigma_1 (n,1;m) \leq
\mathcal{E}_n (r) + \frac{2\|r\|}{\sqrt{n}} = \tilde{\rho}_n \quad
\text{if} \quad |n-m| \leq n/2. $$ This completes the proof of
(\ref{p2}).

Next we estimate  $\sigma_1 (n,2;m).$ We have $$ \sigma_1 (n,2;m)
= \sum_{j_1 \neq n} \frac{r(m+j_1)}{|n-j_1 |} \sum_{j_2 \neq n}
\frac{r(j_1+j_2)}{|n-j_2|} $$ $$ = \sum_{0<|j_1- n| \leq  n/2}
\frac{r(m+j_1)}{|n-j_1|} \cdot \sigma_1 (n,1;j_1) +
\sum_{|j_1-n|>n/2} \frac{r(m+j_1)}{|n-j_1|} \cdot \sigma_1
(n,1;j_1) $$  By the Cauchy inequality and (\ref{p2}), we get $$
\sum_{0<|j_1- n| \leq  n/2} \frac{r(m+j_1)}{|n-j_1|} \sigma_1
(n,1;j_1) \leq  \|r\|\cdot \sup_{0<|j_1- n| \leq  n/2} \sigma_1
(n,1;j_1) \leq \|r\| \tilde{\rho}_n, $$ and $$ \sum_{|j_1-n|>n/2}
\frac{r(m+j_1)}{|n-j_1|} \sigma_1 (n,1;j_1) \leq
\sum_{|j_1-n|>n/2} \frac{r(m+j_1)}{|n-j_1|} \cdot \|r\| \leq
\frac{2\|r\|}{\sqrt{n}} \cdot \|r\|. $$ Thus, in view of
(\ref{p1}), we have
\begin{equation}
\label{p6} \sigma_1 (n,2;m) \leq 2\|r\|\cdot \tilde{\rho}_n.
\end{equation}

On the other hand, for every  $s \in \mathbb{N},$ we have  $$
\sigma_1 (n,s+2;m) = \sum_{j_1,\ldots, j_s  \neq n}
\frac{r(m+j_1)}{|n-j_1|}\cdots \frac{r(j_{s-1}+j_s)}{|n-j_s|}$$$$
\times \;\; \sum_{j_{s+1},j_{s+2} \neq n} \frac{r(j_s
+j_{s+1})}{|n-j_{s+1}|} \frac{r(j_{s+1}+j_{s+2})}{|n-j_{s+2}|} =
\sigma_1 (n,s;m) \cdot \sup_{j_s} \sigma_1 (n,2;j_s).$$ Thus, by
(\ref{p6}),
\begin{equation}
\label{p7} \sigma_1 (n,s+2;m)  \leq \sigma_1 (n,s;m) \cdot 2\|r\|
\tilde{\rho}_n.
\end{equation}
Now it is easy to see, by induction in $p,$ that  (\ref{p2}),
(\ref{p6}) and (\ref{p7}) imply (\ref{p3}).
\end{proof}

\begin{Lemma}
\label{lemp3} Let $r = (r(k)) \in \ell^2 (2\mathbb{Z}) $ be the
sequence defined by (\ref{p26}), and let
\begin{equation}
\label{p15} \sigma_2 (n,s;m) = \sum_{j_1, \ldots, j_s \neq   n}
r(m+j_1)\frac{r(j_1+j_2)}{|n+j_2|} \cdots
\frac{r(j_{s-2}+j_{s-1})}{|n+j_{s-1}|}
\frac{r(j_{s-1}+j_s)}{|n^2-j^2_s|},
\end{equation}
where $n \in \mathbb{N}, \; s\geq 2$ and $ m, j_1, \ldots, j_s \in n+
2\mathbb{Z}. $ Then we have
\begin{equation}
\label{p16} \sigma_2 (n,2;m) \leq \|r\|^2 \cdot \frac{2\log \,
6n}{n}
\end{equation}
and
\begin{equation}
\label{p17} \sigma_2 (n,s;m) \leq \|r\|^2 \cdot \frac{2 \log \,
6n}{n} \cdot \sup_k \sigma_1 (n,s-2;k), \quad s \geq 3.
\end{equation}
\end{Lemma}

\begin{proof}
We have
\begin{equation}
\label{p18} \sigma_2 (n,2,m) = \sum_{j_2 \neq \pm n}
\frac{1}{|n^2-j_2^2 |} \sum_{j_1 \neq \pm n} r(m+j_1)r(j_1+j_2).
\end{equation}
By the Cauchy inequality, the sum over $ j_1 \neq \pm n  $ does
not exceed $\|r\|^2. $  Let us notice that
\begin{equation}
\label{p19} \sum_{j \neq \pm n} \frac{1}{n^2 - j^2} = \frac{2}{n}
\sum_1^{2n} \frac{1}{k} - \frac{1}{2n^2} < \frac{2 \log 6n}{n}.
\end{equation}
Therefore, (\ref{p18}) and (\ref{p19}) imply (\ref{p16}).

If $s\geq 3$ then  the sum $\sigma_2 (n,s;m) $ can be written as a
product of three sums:
$$ \sum_{j_s \neq \pm n} \frac{1}{|n^2-j^2_s|}
\sum_{j_2, \ldots, j_{s-1}\neq \pm n} \frac{r(j_2+j_3)}{|n+j_2|}
\cdots \frac{r(j_{s-1}+j_s)}{|n+j_{s-1}|} \sum_{j_1 \neq \pm n}
r(m+j_1)r(j_1+j_2).$$

Changing the sign of all indices, one can easily see that the
middle sum (over $j_2, \ldots, j_{s-1} $) equals $ \sigma_1
(n;s-2,j_s).$ Thus, we have

$$ \sigma_2 (n,s;m) \leq \sum_{j_s \neq \pm n}
\frac{1}{|n^2-j^2_s|} \sigma_1 (n;s-2,j_s) \cdot \sup_{j_2}
\sum_{j_1 \neq \pm n} r(m+j_1)r(j_1+j_2).$$

By the Cauchy inequality, the sum over $ j_1 \neq \pm n  $ does
not exceed $\|r\|^2. $

Therefore, by (\ref{p19}), we get (\ref{p17}).
\end{proof}

{\em Proof of Lemma \ref{lemp1}.} We set
\begin{equation}
\label{p60} \varepsilon_n = M \cdot \left [ \left ( \frac{2 \log
6n}{n} \right )^{1/4} + (\tilde{\rho}_n)^{1/2} \right ],
\end{equation}
where $M = 4(1 + \|r\|)$  is chosen so that for large enough $n$
\begin{equation}
\label{p61} \sup_m \sigma_1 (n,2p,m) \leq (\varepsilon_n/2)^{2p},
\quad \sup_m \sigma_1 (n,2p+1,m) \leq \|r\|
(\varepsilon_n/2)^{2p}.
\end{equation}
Then, for large enough $n,$ we have
\begin{equation}
\label{p62} \sup_m \sigma_2 (n,s,m) \leq \frac{1}{M}
(\varepsilon_n/2)^{s+1}.
\end{equation}

Indeed, by the choice of $M,$ we have
\begin{equation}
\label{p62a}
 \|r\|^2 \cdot \frac{2\log \, 6n}{n}
  \leq  \|r\|^2 (\varepsilon_n/M)^4 \leq
\frac{1}{M^2} (\varepsilon_n/2)^4.
\end{equation}
Since $\varepsilon_n \to 0, $ there is $n_0 $ such that $
\varepsilon_n <1 \quad \text{for} \; n \geq n_0.$ Therefore, if
$n\geq n_0, $ then
 (\ref{p16}) and (\ref{p62a})
yields
 (\ref{p62}) for $s=2.$
If $s=2p $ with $p>1,$ then (\ref{p17}), (\ref{p62a}) and
(\ref{p61}) imply, for $n \geq n_0,$ $$\sup_m \sigma_2 (n,2p,m)
\leq \frac{1}{M^2} (\varepsilon_n/2)^4 \cdot
(\varepsilon_n/2)^{2p-2} \leq
\frac{1}{M}(\varepsilon_n/2)^{2p+1}.$$

In an analogous way, for $n \geq n_0,$  we get $$\sup_m \sigma_2
(n,2p+1,m) \leq \frac{1}{M^2}(\varepsilon_n/2)^4 \cdot
\|r\|(\varepsilon_n/2)^{2(p-1)} \leq
\frac{1}{M}(\varepsilon_n/2)^{2p+2}, $$ which completes the proof
of (\ref{p62}).

Next we estimate $\sigma (n,s)$ by induction in $s.$ By
(\ref{p2}), we have for $n\geq n_0,$
\begin{equation}
\label{p63}
 \sigma (n,1) = \sum_{j_1 \neq \pm n} \frac{r(n-j_1)}{|n- j_1|} =
\sigma_1 (n,1;n) \leq \tilde{\rho}_n \leq (\varepsilon_n/2)^2 \leq
\varepsilon_n.
\end{equation}

For $s=2$ we get, in view of (\ref{p61}) and (\ref{p61}):

\begin{equation}
\label{p64}  \sigma (n,2) = \sum_{j_1, j_2 \neq \pm n} \left (
\frac{1}{|n- j_1|} + \frac{1}{|n+ j_2|} \right ) \frac{1}{|n-
j_2|} r(n+j_1)r(j_1 + j_2)
\end{equation}
$$ \leq \sum_{j_1, j_2 \neq \pm n}  \frac{r(n+j_1)r(j_1 + j_2)}{|n- j_1||n- j_2|}
 + \sum_{j_1, j_2 \neq \pm n}
\frac{1}{|n+ j_2|} \cdot \frac{1}{|n- j_2|} r(n+j_1)r(j_1 + j_2) $$
$$ = \sigma_1 (n,2,n) + \sigma_2 (n,2,n) \leq (\varepsilon_n/2)^2
+(\varepsilon_n/2)^2 \leq (\varepsilon_n)^2.
$$

 Next we estimate $ \sigma (n,s), \;s\geq 2, $
Recall that $\sigma (n,s) $ is the sum of terms of the form $$ \Pi
(j_1, \ldots, j_s) r(n+j_1) r(j_1 +j_2 ) \cdots r(j_{s-1}+j_s), $$
where
\begin{equation}
\label{p65} \Pi (j_1, \ldots, j_s) = \left ( \frac{1}{|n- j_1|} +
\frac{1}{|n+ j_2|} \right ) \cdots \left ( \frac{1}{|n- j_{s-1}|}
+ \frac{1}{|n+ j_s|} \right ) \frac{1}{|n- j_s|}.
\end{equation}
By opening the parentheses we get
\begin{equation}
\label{p66} \Pi (j_1, \ldots, j_s) = \sum_{\delta_1,
\ldots,\delta_{s-1}=\pm 1} \left ( \prod_{\nu=1}^{s-1}
\frac{1}{|n+\delta_\nu j_{\nu +\tilde{\delta}_\nu}|} \right )
\frac{1}{|n-j_s|}, \quad \tilde{\delta}_\nu = \frac{1+\delta_\nu
}{2}.
\end{equation}
Therefore,
\begin{equation}
\label{p67} \sigma (n,s) = \sum_{\delta_1, \ldots,\delta_{s-1}=\pm
1} \tilde{\sigma}(\delta_1, \ldots,\delta_{s-1}),
\end{equation}
where
\begin{equation}
\label{p68} \tilde{\sigma}(\delta_1, \ldots,\delta_{s-1})= \sum_{j_1,
\ldots, j_s \neq \pm n} \left ( \prod_{\nu=1}^{s-1}
\frac{1}{|n+\delta_\nu j_{\nu +\tilde{\delta}_\nu}|} \right )
\frac{1}{|n-j_s|} r(n+j_1)\cdots r(j_{s-1}+j_s).
\end{equation}

In view of (\ref{p52}), (\ref{p53}) and (\ref{p67}), Lemma
\ref{lemp1} will be proved if we show that
\begin{equation}
\label{p70} \tilde{\sigma}(\delta_1, \ldots,\delta_{s-1}) \leq
(\varepsilon_n/2)^s, \quad s\geq 2.
\end{equation}
We prove (\ref{p70}) by induction in $s.$

If $s=2$ then $$\tilde{\sigma}(-1)= \sigma_1 (n,2,n) \leq
(\varepsilon_n/2)^2, $$ and $$\tilde{\sigma}(+1)= \sigma_2 (n,2,n)
\leq (\varepsilon_n/2)^2. $$

If $s=3$ then there are four cases: $$\tilde{\sigma}(-1,-1)=
\sigma_1 (n,3,n) \leq (\varepsilon_n/2)^3; \quad
\tilde{\sigma}(+1,+1)= \sigma_2 (n,3,n) \leq (\varepsilon_n/2)^3;
$$ $$\tilde{\sigma}(-1,+1)= \sum_{j_1 \neq \pm n}
\frac{r(n+j_1)}{|n-j_1|} \sum_{j_2, j_3 \neq \pm n}
\frac{r(j_1+j_2)}{|n+j_3|}\frac{r(j_2+j_3)}{|n-j_3|}$$ $$=
\sum_{j_1 \neq \pm n} \frac{r(n+j_1)}{|n-j_1|} \sigma_2
(n,2,j_1)$$ $$ \leq \sigma_1 (n,1,n) \cdot \sup_m \sigma_2 (n,2,m)
\leq \|r\| \frac{1}{K} (\varepsilon_n/2)^3 \leq
(\varepsilon_n/2)^3;$$ $$\tilde{\sigma}(+1,-1)= \sum_{j_1,j_2 \neq
\pm n} \frac{r(n+j_1)r(j_1 +j_2)}{|n^2-j^2_2|} \sum_{j_3 \neq \pm
n} \frac{r(j_2 +j_3)}{|n-j_3|} $$ $$ \leq \sigma_2 (n,2,n) \cdot
\sup_m \sigma_1 (n,1,m) \leq \frac{1}{K}(\varepsilon_n/2)^3 \|r\|
\leq (\varepsilon_n/2)^3.$$

Next we prove that if (\ref{p70}) hold for some $s,$ then it holds
for $s+2.$ Indeed, let us consider the following cases:

(i)    $\delta_s =     \delta_{s+1} = -1;         $ then we have
$$ \tilde{\sigma}(\delta_1, \ldots,\delta_{s-1},-1,-1)
=\sum_{j_1,\ldots,j_s \neq \pm n} \left ( \prod_{\nu=1}^{s-1}
\frac{1}{|n+\delta_\nu j_{\nu +\tilde{\delta}_\nu}|} \right )
\frac{1}{|n-j_s|}$$ $$ \times \;  r(n+j_1)r(j_1 +j_2)\cdots
r(j_{s-1} +j_s) \sum_{j_{s+1}, j_{s+2} \neq \pm n} \frac{r(j_s
+j_{s+1})}{|n-j_{s+1}|}\frac{r(j_{s+1} +j_{s+2})}{|n-j_{s+2}|}$$
$$ =\sum_{j_1,\ldots,j_s \neq \pm n} \left ( \prod_{\nu=1}^{s-1}
\frac{1}{|n+\delta_\nu j_{\nu +\tilde{\delta}_\nu}|} \right )
\frac{1}{|n-j_s|} r(n+j_1)\cdots r(j_{s-1} +j_s) \sigma_1
(n,2,j_s) $$ $$ \leq \tilde{\sigma}(\delta_1, \ldots,\delta_{s-1})
\cdot \sup_m  \sigma_1 (n,2,m) \leq (\varepsilon_n/2)^s \cdot
(\varepsilon_n/2)^2 = (\varepsilon_n/2)^{s+2}. $$

(ii)   $\delta_s =-1,   \;     \delta_{s+1} = + 1;$ then we have
$$ \tilde{\sigma}(\delta_1, \ldots,\delta_{s-1},-1,+1)
=\sum_{j_1,\ldots,j_s \neq \pm n} \left ( \prod_{\nu=1}^{s-1}
\frac{1}{|n+\delta_\nu j_{\nu +\tilde{\delta}_\nu}|} \right )
\frac{1}{|n-j_s|}$$ $$ \times \;  r(n+j_1)r(j_1 +j_2)\cdots
r(j_{s-1} +j_s) \sum_{j_{s+1}, j_{s+2} \neq \pm n} \frac{r(j_s
+j_{s+1})r(j_{s+1} +j_{s+2})}{|n^2-j^2_{s+2}|}$$
 $$ \leq
\tilde{\sigma}(\delta_1, \ldots,\delta_{s-1}) \cdot \sup_m
\sigma_2 (n,2,m) \leq (\varepsilon_n/2)^s \cdot
(\varepsilon_n/2)^2 = (\varepsilon_n/2)^{s+2}. $$

(iii) $\delta_s = \delta_{s+1} = +1;$ then, if $\delta_1 = \cdots
= \delta_{s-1}= +1, $ we have $$ \tilde{\sigma}(\delta_1,
\ldots,\delta_{s+1})= \sigma_2 (n,s+2,n) \leq
(\varepsilon_n/2)^{s+2}. $$

Otherwise, let $\mu < s$  be the largest index such that
$\delta_\mu = -1.$ Then we have

$$ \tilde{\sigma}(\delta_1, \ldots,\delta_{s-1},+1,+1)
=\sum_{j_1,\ldots,j_\mu \neq \pm n} \left ( \prod_{\nu=1}^{\mu-1}
\frac{1}{|n+\delta_\nu j_{\nu +\tilde{\delta}_\nu}|} \right )
\frac{1}{|n-j_\mu|}$$ $$ \times \;  r(n+j_1)r(j_1 +j_2)\cdots
r(j_{\mu-1} +j_\mu) \sigma_2 (n,s+2-\mu,j_\mu)             $$ $$
\leq \tilde{\sigma}(\delta_1, \ldots,\delta_{\mu-1}) \cdot \sup_m
\sigma_2 (n,s+2-\mu,j_\mu) \leq (\varepsilon_n/2)^{\mu} \cdot
(\varepsilon_n/2)^{s+2-\mu} = (\varepsilon_n/2)^{s+2}. $$

(iv) $\delta_s =+1, \; \delta_{s+1} = -1;$ then, if $\delta_1 =
\cdots = \delta_{s-1}= +1, $ we have $$ \tilde{\sigma}(\delta_1,
\ldots,\delta_{s+1},-1)= \tilde{\sigma}(+1,\ldots, +1,-1, -1)=
$$$$ =\sum_{j_1,\ldots,j_{s+1}\neq \pm n} \left ( \prod_{\nu =1}^s
\frac{1}{|n+ j_{\nu +1}|} \right ) \frac{1}{|n-j_{s+1}|}
r(n+j_1)\cdots r(j_s +j_{s+1}) \sigma_1 (n,1, j_{s+1}) $$ $$ \leq
\sigma_2 (n,s+1,n) \cdot \sup_m \sigma_1 (n,1,m) \leq \frac{1}{K}
(\varepsilon_n/2)^{s+2} \cdot \|r\| \leq (\varepsilon_n/2)^{s+2}.
$$

Otherwise, let $\mu <s$  be the largest index such that
$\delta_\mu = -1,\; 1\leq \mu < n. $ Then we have

$$ \tilde{\sigma}(\delta_1, \ldots,\delta_{s-1},+1,-1)
=\sum_{j_1,\ldots,j_\mu \neq \pm n} \left ( \prod_{\nu=1}^{\mu-1}
\frac{1}{|n+\delta_\nu j_{\nu +\tilde{\delta}_\nu}|} \right )
\frac{1}{|n-j_\mu|}$$ $$ \times \; \sum_{j_{\mu+1}, \ldots,
j_{s+1} \neq \pm n} \frac{r(j_\mu +j_{\mu+1})}{|n+j_{\mu+2}|}
\cdots
 \frac{r(j_{s-1}+j_{s})}{|n+j_{s+1}|}
\frac{r(j_{s}+j_{s+1})}{|n-j_{s+1}|} \sum_{j_{s+2}\neq \pm n}
\frac{r(j_{s+1}+j_{s+2})}{|n-j_{s+2}|}$$
 $$ \leq \tilde{\sigma}(\delta_1,\ldots,\delta_{\mu-1})
 \cdot \sup_m  \sigma_2 (n,s+1- \mu,m) \cdot \sup_k \sigma_1 (n,1,k)$$
 $$
\leq (\varepsilon_n/2)^{\mu} \cdot \frac{1}{K}
(\varepsilon_n/2)^{s+2-\mu}  \|r\| \leq (\varepsilon_n/2)^{s+2}.
$$ Hence (\ref{p70}) holds for $s\geq 2.$

Now (\ref{p53}), (\ref{p67}) and (\ref{p70}) imply (\ref{p52}),
which completes the proof of Lemma \ref{lemp1}.
\end{proof}

Now we are ready to accomplish the proof of Theorem \ref{thm91}.

\subsection{Proof of the main theorem on the deviations $\|P_n - P_n^0\|_{L^1 \to L^\infty} $}

We need -- because we want to use (\ref{p24b}) -- to give
estimates of $A(n,s)$ from (\ref{p24c}), or (\ref{p29}). By
(\ref{p33a}) and (\ref{p33}), we reduce such estimates to analysis
of quantities $A_j (n,s), \; j=1, \ldots, 7.$

With $\rho_n \in (\ref{p34})$ and $\varepsilon_n \in (\ref{p60}),
$ we set
\begin{equation}
\label{18b1} \kappa_n = \max\{ \rho_n, \varepsilon_n\}.
\end{equation}
Then, by Lemma \ref{lemp1} (and Corollary \ref{cor1}), i.e., by
the inequality (\ref{p520}), we have (in view of
(\ref{p37}),(\ref{p40}),(\ref{p42}) and (\ref{p44})--(\ref{p47}))
the following estimates for $A_j:$ $$A_1 \leq 4 \kappa_n^{s+1},
\qquad A_j \leq 2^{s+1} \cdot 2 \kappa_n^{s+1}, \quad j=2,4;$$
$$A_j \leq 2^{s+1} \sum_{\nu=1}^s \left (2 \kappa_n^\nu \cdot
\kappa_n^{s-\nu +1} \right ) = s 2^{s+2} \kappa_n^{s+1},\quad
j=3,5;$$ $$A_6 \leq 2^{s+1} \sum_{\nu=1}^s \left (\kappa_n^\nu
\cdot \kappa_n^{s-\nu +1} \right ) = s 2^{s+1} \kappa_n^{s+1};$$
$$A_7 \leq 2^{s} \sum_{1\leq \nu + \mu \leq s} \left (4
\kappa_n^\nu \cdot \kappa^{\mu-\nu} \cdot \kappa_n^{s-\mu +1}
\right ) = s(s-1) 2^{s+1} \kappa_n^{s+1}.$$ In view of
(\ref{p27}), (\ref{p60}) and (\ref{p33}), these inequalities imply
$$ A(n,s) \leq  (2+s)^2  (2\kappa_n)^{s+1}.$$ Therefore, the
right--hand side of (\ref{p24b}) does not exceed $$
\sum_{s=0}^\infty A(n,s) \leq  (4\kappa_n)\sum_{s=0}^\infty
(s+1)(s+2) (2\kappa_n)^s =\frac{8\kappa_n}{(1-2\kappa_n)^3}. $$
Therefore, if $\kappa_n < 1/4$ (which holds for $ n \geq N^* $
with a proper choice of $N^*$), then $\sum_{s=0}^\infty A(n,s)
\leq 64\kappa_n. $ Thus, by (\ref{p24b}) and the notations
(\ref{p210}),
\begin{equation}
\label{18d} \| P_n -P_n^0 \|_{L^1 \to L^\infty} \leq \sum_{k,m}
|B_{km} (n)| \leq 64 \kappa_n, \quad n \geq N^*,
\end{equation}
 where $ \kappa_n
\in (\ref{18b1}).$

This completes the proof of Theorem \ref{thm91}. Of course,
Proposition \ref{prop91} follows because $ \|T\|_{L^2 \to L^2}
\leq \|T\|_{L^1 \to L^\infty} $ for any well defined operator $T.$
\end{proof}

\subsection{Miscellaneous}

As we have already noticed many times, Theorem~\ref{thm91} (or
Proposition~\ref{prop91}) is an essential step in the proof of our
main result Theorem~\ref{thm44.2} (see an announcement in
\cite{DM17}, Thm. 9, or \cite{DM16}, Thm. 23) about the
relationship between the rate of decay of spectral gap sequences
(and deviations) and the smoothness of the potentials $v$ under
the {\em a priori } assumption that $v$ is a singular potential,
i.e., that $v \in H^{-1}_{Per}. $  But Theorem \ref{thm91} is
important outside this context as well. We will mention now the
most obvious corollaries.\vspace{2mm}

The following theorem holds.

\begin{Theorem}
\label{thm97} In the above notations, the $L^p$-norms, $1\leq p
\leq \infty, $ on Riesz subspaces $E^N = Ran \,S_N, $ and $E_n =
Ran \,P_n, \; n \geq N,$  are uniformly equivalent; more
precisely,
\begin{equation}
\label{a1} \|f\|_1  \leq \|f\||_\infty \leq C(N) \|f\|_1,  \quad
\forall f \in E^N,
\end{equation}
and
\begin{equation}
\label{a2} \|f\||_\infty \leq 3 \|f\|_1,  \quad  \forall f \in
E_n, \quad n \geq N^* (v),
\end{equation}
where
\begin{equation}
\label{a3} C(N) \leq 50 N \ln N.
\end{equation}
\end{Theorem}

\begin{proof}
By (\ref{p21}), if $N$ is large enough,
\begin{equation}
\label{a5} \|P_n - P_n^0\|_{L^1 \to L^\infty} \leq \frac{1}{2},
\quad n\geq N.
\end{equation}
If we are more careful when using (\ref{18b1}),(\ref{18d}),
(\ref{p34}) and (\ref{p60}), we may claim (\ref{a5}) for $N$ such
that
\begin{equation}
\label{a6} 2^{9}(1+\|r\|) \left ( \mathcal{E}_{\sqrt{N}} (r) +
\frac{2}{N^{1/4}} \left ( \|r\|^{1/2} + (\ln 6N)^{1/4} \right )
\right ) \leq \frac{1}{2}.
\end{equation}
If $f \in E_n, \, n \geq N,$ we have
\begin{equation}
\label{a6a} f = P_n f = (P_n - P_n^0 )f + P_n^0 f,
\end{equation}
where, for $bc = Per^\pm$
\begin{equation}
\label{a32} P_n^0 f = f_n e^{inx} + f_{-n} e^{-inx}, \quad f_k =
\frac{1}{\pi} \int_0^\pi f(x)e^{-ikx} dx,
\end{equation}
and, for $bc = Dir, $
\begin{equation}
\label{a33} P_n^0 f = 2g_n \sin nx, \quad g_n =
\frac{1}{\pi}\int_0^\pi f(x) \sin nx dx.
\end{equation}
In either case $\|P_n f \|_\infty  \leq 2 \|f\|_1, $ and
therefore, if $ \|f\|_1 \leq 1 $ we have
\begin{equation}
\label{a34} \|f\|_\infty \leq \|(P_n - P_n^0 )f\|_\infty +\| P_n^0
f\|_\infty  \leq 1/2 + 2 \leq 3.
\end{equation}
Remind that a projection
\begin{equation}
\label{a35} S_N = \frac{1}{2\pi i} \int_{\partial R_N}
(z-L_{bc})^{-1} dz,
\end{equation}
where, as in (5.40), \cite{DM16},
\begin{equation}
\label{a41} R_N = \{ z\in \mathbb{C}: \; -N < Re z < N^2 +N, \;
|Im z | < N,
\end{equation}
is finite--dimensional (see \cite{DM16}, (5.54), (5.56), (5.57)
for $\dim S_N $). Now we follow the inequalities proven in
\cite{DM16} to explain (\ref{a1}) and (\ref{a3}). Lemma 20,
inequality (5.41) in \cite{DM16}, states that
\begin{equation}
\label{a42} \sup \{\|K_\lambda V K_\lambda \|_{HS}: \; \lambda
\not \in R_N, \;  Re \lambda \leq N^2 - N \} \leq C \left (
\frac{(\log N)^{1/2}}{N^{1/4}} \|q\| + \mathcal{E}_{4\sqrt{N}} (q)
\right ).
\end{equation}
But by (\ref{a35})
\begin{equation}
\label{a43} S_N -S_N^0 = \frac{1}{2\pi i} \int_\Gamma K_\lambda
\sum_{m=1}^\infty (K_\lambda V K_\lambda)^m K_\lambda d\lambda,
\end{equation}
where we can choose $\Gamma$ to be the boundary $\partial \Pi $ of
the rectangle
\begin{equation}
\label{a44} \Pi (H)= \{z \in \mathbb{C}: \; -H \leq Re \,z \leq
N^2 + N, \; |Im \,z| \leq H \}, \quad H \geq N.
\end{equation}
Then by (\ref{a42}) and (\ref{a43}) the norm of the sum in the
integrand can be estimated by
\begin{equation}
\label{a45} \left \|\sum_1^\infty  \right \|_{2\to 2} \leq
\sum_1^\infty \|K_\lambda V K_\lambda\|^m_{HS} \leq 1, \quad
\forall \lambda \in
\partial \Pi (H)
\end{equation}
if (compare with (\ref{a6})) $ N \geq N^*(q)  $ and $N^* = N^*(q)$
is chosen to guarantee that
\begin{equation}
\label{a51} \text{``the right side in (\ref{a42})``} \leq 1/2  \;
\; \text{for} \; N \geq N^*.
\end{equation}
The additional factor $K_\lambda $ is a multiplier operator
defined by the sequence $\tilde{K} =\{1/\sqrt{\lambda -k^2}\},$
 so its norms $ \|K_\lambda : L^1 \to L^2\|$ and
 $ \;\|K_\lambda : L^2 \to L^\infty\|$
are estimated by $2\tilde{\kappa},$
 where
\begin{equation}
\label{a53}  \tilde{\kappa}= \|\tilde{K_\lambda} : \; \ell^\infty
\to \ell^2 \| = \|\tilde{K_\lambda} : \; \ell^2 \to \ell^1
\|=\sum_k \frac{1}{|\lambda - k^2|}.
\end{equation}
Therefore, by (\ref{a45}) and (\ref{a53}),
\begin{equation}
\label{a54}  \alpha(\lambda):= \|K_\lambda \left (\sum_1^\infty
(K_\lambda V K_\lambda)^m \right )K_\lambda : \; L^1 \to L^\infty
\| \leq \sum_k \frac{4}{|\lambda - k^2|}.
\end{equation}
By Lemma 18(a) in \cite{DM16} (or, Lemma 79(a) in \cite{DM15})
\begin{equation}
\label{a55}
 \sum_k \frac{1}{|n^2 - k^2|+b} \leq C_1 \frac{\log b}{\sqrt{b}} \quad
 \text{if} \;\; n \in \mathbb{N}, \; b \geq 2.
\end{equation}
(In what follows $C_j, \, j=1,2, \ldots$ are absolute constants;
$C_1 \leq 12.$) These inequalities are used to estimate the norm
$\alpha(\lambda)$ on the boundary $\partial \Pi (H) = \cup I_k
(H), \; k=1,2,3,4,$ where $$ I_1 (H)= \{z: \; Re \, z = -H, \; |Im
\, z| \leq H\} $$ $$ I_2 (H)= \{z: \;  -H \leq Re \, z \leq N^2 +N
, \; Im \, z = H\} $$ $$ I_3 (H)= \{z: \; Re \, z = N^2 +N, \; |Im
\, z| \leq H\} $$ $$ I_2 (H)= \{z: \;  -H \leq Re \, z \leq N^2 +N
, \; Im \, z = - H\} $$ Then we get $$ \int_{I_1}\alpha(\lambda)
|d\lambda| \leq C_2  \frac{\log H}{\sqrt {H}} \cdot H,  $$ $$
\int_{I_k}\alpha(\lambda) |d\lambda| \leq C_3  \frac{\log H}{\sqrt
{H}}\cdot N^2, \quad k=2,4. $$ $$ \int_{I_3}\alpha(\lambda)
|d\lambda| \leq C_4 \int_0^H \frac{\log (N+y)}{\sqrt{N+y}} dy \leq
C_5\sqrt{H} \log H. $$ If we put $H= N^2$ and sum up these
inequalities we get by (\ref{a43})
\begin{equation}
\label{a71} \| S_N -S_N^0\|_{L^1 \to L^\infty} \leq C_6 N \log N,
\end{equation}
where $C_6 $ is an absolute constant $\leq 600.$

Now, as in (\ref{a6a}) and (\ref{a32}), let us notice that for $g
\in E^N$
\begin{equation}
\label{a72} g = S_N g = (S_N -S_N^0) g + S_N^0 g,
\end{equation}
where
\begin{equation}
\label{a73} S_N^0 g = \sum_{|k|\leq N} g_k e^{ikx}, \quad k
\;\text{even for} \;bc =Per^+,\;\;\text{odd for} \;bc =Per^-,
\end{equation}
and
\begin{equation}
\label{a74} S_N^0 g = 2\sum_{|k|\leq N} \tilde{g}_k \sin kx, \quad
bc = Dir,
\end{equation}
where
\begin{equation}
\label{a75} g_k = \frac{1}{\pi}\int_0^\pi g(x) e^{ikx}dx, \quad
\tilde{g}_k = \frac{1}{\pi}\int_0^\pi g(x) \sin kx dx.
\end{equation}
In either case
\begin{equation}
\label{a76} \|S_N^0 g\|_\infty \leq 2N \|g\|_1.
\end{equation}
Therefore, by (\ref{a71}) and (\ref{a76}), if $\|f\|_1 \leq 1$ we
have
\begin{equation}
\label{a81} \|f\|_\infty \leq C_6 N \log N + 2N \leq C_7 N \log N,
\quad N \geq N^* \in (\ref{a51}).
\end{equation}
Let us fix $N_0 \geq N^*, N_*,$ where $N_*$ is determined by
(\ref{a6}), i.e., (\ref{a6}) holds if $N \geq N_*.$ Then, by
(\ref{a81}),
\begin{equation}
\label{a82} \|S_{N_0}\|_{L^1 \to L^\infty} \leq C N_0 \log N_0,
\end{equation}
and for $N >N_0 $ we may improve the estimate in (\ref{a81}).
Indeed, $$ S_N = (S_N - S_{N_0}) + S_{N_0}= S_{N_0} + \sum_{N_0
+1}^N P_k $$ and, by (\ref{a5}) and (\ref{a32}),  $\|P_k\|_{L^1
\to L^\infty} \leq 3.$ Therefore, by (\ref{a82}), $$\|S_N \|_{L^1
\to L^\infty} \leq C N_0 \log N_0 + (N-N_0) \leq 3N + C N_0 \log
N_0. $$
\end{proof}
\vspace{2mm}

Of course, any estimates of the kind
\begin{equation}
\label{a91} \|S_N - S^0_N\|_{L^1 \to L^\infty} \leq C(N)
\end{equation}
with $C_N \to \infty $ as $N \to \infty $ are weaker than the
claim
\begin{equation}
\label{a92} \omega_N =\|S_N - S^0_N\|_{L^1 \to L^\infty} \to 0
\quad \text{as} \;\;N \to \infty
\end{equation}
or even that $\omega_N$ is a bounded sequence. For real--valued
potentials $v \in H^{-1} $  and  $bc = Dir, $ (\ref{a92}) would
follow from Theorem 1 in \cite{Sa} if its proof given in \cite{Sa}
were valid. For complex--valued potentials $v \in H^{-1}, $ when
the system of eigenfunctions is not necessarily orthogonal the
statement of Theorem 1 in \cite{Sa} is false. Maybe it could be
corrected if the ''Fourier coefficients'' are chosen as $$ c_k (f)
= \langle f,w_k \rangle $$ where the system $\{ w_k\}$ is
bi--orthonormal with respect to $\{ u_k\}, $ i.e., $$ \langle u_j,
w_k \rangle = \delta_{jk} $$ (not the way as it is done in
\cite{Sa}). But more serious oversight, not just a technical
misstep, seems to be a crucial reference to \cite{SS03}, without
specifying lines or statements in \cite{SS03}, to claim something
that cannot be found there. Namely, the author of \cite{Sa}
alleges that in \cite{SS03} the following statement is
proven\footnote{Recently, in Preprint ArXiv:0806.3016, 18 Jun 2008
-- without giving any proofs of the 2003 claims in \cite{Sa} -- I.
Sadovnichaya states a weaker  form of (\ref{a92}), namely $$
\tilde{\omega}_N =\|S_N - S^0_N\|_{L^2 \to L^\infty} \to 0 \quad
\text{as} \;\;N \to \infty, $$ but again the proof has a series of
steps without any justification. For example, the text on page 3,
lines 1-2 from below, claims that

(*)`` the sequence $\{\|\psi_{n,2} (x)\|_C \}_{n=N}^\infty $
belongs to the space $\ell^1... $``.

 This is important for explanation of the
inequality (21) in Step 3, page 10. But no proof of (*) is given;
there is no reference either.}

{\em Let $\{ y_k (x)\}$ be a normalized system of eigenfunctions
of the operator $$ L= - d^2/dx^2 + v, \quad v \in H^{-1}
([0,\pi]), $$ considered with Dirichlet boundary conditions. Then
\begin{equation}
\label{a111} y_k (x) = \sqrt{2} \sin kx + \psi_k (x),
\end{equation}
where
\begin{equation}
\label{a112} \sup_{[0,\pi]} |\psi_k (x) | \in \ell^2.
\end{equation}
(Two more sup--sequences coming from derivatives $y^\prime_k$ are
claimed to be in $\ell^2$ as well.)}

However, what one could find in \cite{SS03}, Theorem 2.7 and
Theorem 3.13(iv),(v), is the claim
\begin{equation}
\label{a113} \sup_{[0,\pi]} \sum_k |\psi_k (x) |^2   < \infty.
\end{equation}
Of course, (\ref{a112}) implies (\ref{a113}) but if $\{\psi_k \}$
is a sequence of $L^\infty$-functions then (\ref{a113}) does not
imply (\ref{a112}).

\end{document}